%% file: NLHAFEM9.tex
\documentclass[leqno,final]{article}

\usepackage{fancyhdr}
\usepackage{float}
\usepackage{tikz}
\usepackage{graphicx}
\usepackage{epstopdf}
\usepackage{caption}
\usepackage{subfig}
\usepackage{algorithmic}
\ifpdf
  \DeclareGraphicsExtensions{.eps,.pdf,.png,.jpg}
\else
  \DeclareGraphicsExtensions{.eps}
\fi

\usepackage{cancel}
\usepackage{amsmath,amssymb,amsthm,amsfonts}
\newtheorem{theorem}{Theorem}[section]
\newtheorem{lemma}[theorem]{Lemma}
\newtheorem{remark}[theorem]{Remark}
\newtheorem{corollary}[theorem]{Corollary}
\newtheorem{assumption}[theorem]{Assumption}

\numberwithin{theorem}{section}
\numberwithin{equation}{section}
\numberwithin{figure}{section}
\numberwithin{table}{section}

\usepackage{mathrsfs}
\usepackage{bm}
\usepackage{enumerate}
\usepackage[normalem]{ulem}
\let\isout\sout \renewcommand{\sout}[1]{\ifmmode\text{\isout{\ensuremath{#1}}}\else\isout{#1}\fi}

\usepackage{color}
\usepackage{xcolor}
\colorlet{inlinkcolor}{green!50!black}
\colorlet{exlinkcolor}{red!50!black}
\usepackage[colorlinks = true,
  allcolors = inlinkcolor,
  urlcolor = exlinkcolor,
            ]{hyperref}

\usepackage{indentfirst}
\usepackage{booktabs} 
\usepackage{multirow}  
\usepackage[compress]{cite}
\usepackage{diagbox}

\usepackage{mathtools}
\newcommand{\wideBar}[1]{\overbracket[0.3pt][0pt]{\overbracket[0.3pt][0pt]{\mkern-2.8mu #1}}}
\input{NLHAFEM_def}  

\title{Adaptive Finite Element Method for a Nonlinear Helmholtz Equation with High Wave Number}
\newcommand{\email}[1]{\protect\href{mailto:#1}{#1}}
\author{  Run Jiang\thanks{Department of Mathematics, Nanjing University, Jiangsu, 210093, People's Republic of China.  
This work of these two authors was partially supported by the NSF of China under grants 12171238 and 12261160361.
(\email{dz1821002@smail.nju.edu.cn}, \email{hjw@nju.edu.cn}). }
\and Haijun Wu\footnotemark[1] \and Yifeng Xu\thanks{Department of Mathematics and Scientific Computing Key Laboratory of Shanghai Universities, Shanghai Normal
University, Shanghai 200234, China. 
 The work of this author was partially supported by National Natural Science Foundation of China (Projects 12250013, 12261160361 and 12271367) and Science and Technology Commission of Shanghai Municipality (Projects 17ZR1420800, 20JC1413800 and 22ZR1445400). 
(\email{yfxu@shnu.edu.cn}).
}
\and Jun Zou\thanks{Department of Mathematics, The Chinese University of Hong Kong, Shatin, N.T., Hong Kong, P. R. China.
The work of JZ was substantially supported by Hong Kong RGC General Research Fund (projects 14308322 and 14306719) and NSFC/Hong
Kong RGC Joint Research Scheme 2022/23 (project N\_CUHK465/22).
(\email{zou@math.cuhk.edu.hk}).
}
}

\date{}  

\begin{document}
\maketitle

\begin{abstract}

  A nonlinear Helmholtz (NLH) equation with high frequencies and corner singularities is discretized by the linear finite element method (FEM). After deriving some wave-number-explicit stability estimates and the singularity decomposition for the NLH problem,  {\it a priori} stability and  error estimates are established for the FEM on shape regular meshes including the case of locally refined meshes.
 Then {\it a posteriori} upper and lower bounds using a new residual-type error estimator, which is equivalent to the standard one, are derived for 
the FE solutions to the NLH problem. These {\it a posteriori} estimates have confirmed 
a significant fact that is also valid for the NLH problem, namely 
the residual-type estimator seriously underestimates the error of the FE solution 
in the preasymptotic regime,
which was first observed by Babuška et al. [Int J
Numer Methods Eng 40 (1997)] for a one-dimensional linear problem. Based on the new {\it a posteriori} error estimator,  
both the convergence and the quasi-optimality of the resulting adaptive finite element algorithm are proved the first time 
for the NLH problem, when the initial mesh size lying in the preasymptotic regime. 
Finally,  numerical examples are presented to validate the theoretical findings and demonstrate that applying the continuous interior penalty (CIP) technique with appropriate penalty parameters can reduce the pollution errors efficiently. 
 In particular, the nonlinear phenomenon of optical bistability with Gaussian incident waves is successfully simulated by the adaptive CIPFEM. 
\end{abstract}

{\bf Key words.}
 Adaptive algorithm, {\it a posteriori} error estimate, nonlinear Helmholtz equation, high wave number, Newton's method, preasymptotic error estimate.

{\bf AMS subject classifications. }
65N12, 
65N15, 
65N30, 
78A40  

\section{Introduction}\label{sec:Introduction}
In this work, we are mainly concerned with the construction and analysis of an effective adaptive finite element method 
for a nonlinear Helmholtz (NLH) problem of the form:
\begin{alignat}{2}
-\Delta u - k^2(1+\vep\oneo\abs{u}^2)u &= f \quad \text{in } \Om, \label{eq:Helm}\\
\frac{\partial u}{\partial n} + \mbi k u &= g \quad \text{on } \GaI, \label{eq:Imp}\\
u&=0 \quad \text{on } \GaD\,. \label{eq:Dir}
\end{alignat}
This is a popular model for the wave scattering in nonlinear optical media for the challenging physical scenarios, 
i.e., high frequencies and large incident wave energy \cite{Boyd:2008,Wu:Zou:2018}. 
In the model, $u$ stands for the total wave, $f\in L^2(\Om)$ is the source, 
$\Om\subset\R^2$ is a polygonal physical domain. The cubic nonlinear term in the model corresponds 
to the nonlinear medium involved, the Kerr medium here, which occupies the subdomain $\Om_0$, sitting 
inside the physical domain $\Om$. We are particularly interested in the high wave number case and the Kerr medium,  
i.e., $k\gg 1$ and the Kerr constant $\vep\ll 1$. 
The boundary $\pa \Om$ is divided into two parts, i.e., the impedance part $\GaI$ and the Dirichlet part $\GaD$, 
see the dashed and solid lines respectively in figure \ref{fig:region}.
\eqref{eq:Imp} is called an impedance boundary condition, which is usually regarded as the lowest order absorbing boundary condition.
\eqref{eq:Dir} is the homogeneous Dirichlet boundary condition which generally represents the perfectly conducting boundary. 
We assume $g\in \wt H^{\frac12}(\GaI)$, it means the extension of $g$ by zero outside $\GaI$ belongs to $H^{\frac12}(\pa \Omega)$.

\begin{figure}[tbp]
  \begin{minipage}[t]{0.47\textwidth}
  \centering
  \begin{tikzpicture}
    \draw [dashed] (0,0)--(4,0)--(4,3)--(0,3)--(0,0);
    \draw (2,2)--(2.8,2.7)--(3.6,2.6)--(3.8,2.2)--(3.4,1.5)--(2.6,1.5)--(2,2);
    \draw [fill=cyan!20] (0.25,1) to [out=90,in=180] (0.75,1.5)
    to [out=0,in=180] (2,1.2) to [out=0,in=10] (1.5,0.5)
    to [out=190 ,in=270] (0.25,1);
    \node [scale=0.8]at (1,1) {$\Omega_0$};
    \node [scale=0.8]at (1,2.2) {$\Omega$};
    \node [scale=0.8]at (-0.4,1.5) {$\GaI$};
    \node [scale=0.8]at (2.7,2.2) {$\GaD$};
  \end{tikzpicture}
  \caption*{(a) \itshape{A non-trapping obstacle.}}
  \end{minipage}
  \begin{minipage}[t]{0.47\textwidth}
  \centering
  \begin{tikzpicture}
    \draw (0,3)--(0.5,3)--(1,1)--(3,1)--(3.5,3)--(4,3);
    \draw [dashed] (4,3)--(3,4)--(1,4)--(0,3);
    \draw [fill=cyan!20] (1,2) to [out=90,in=180] (1.5,2.5)
    to [out=0,in=180] (2.75,2.2) to [out=0,in=10] (2.25,1.5)
    to [out=190 ,in=270] (1,2);
    \node [scale=0.8]at (1.7,2) {$\Omega_0$};
    \node [scale=0.8]at (2,3) {$\Omega$};
    \node [scale=0.8]at (0.2,3.7) {$\GaI$};
    \node [scale=0.8]at (0.35,2) {$\GaD$};
  \end{tikzpicture}
  \caption*{(b) \itshape{A cavity.}}
  \end{minipage}
  \caption{\itshape{Setting of the NLH problem.}}
  \label{fig:region}
\end{figure}

For wave scattering problems with high wave number, it is known that 
finite element discretizations may suffer from the so-called pollution effect because of its highly indefiniteness.
For example, it was proved in the one-dimensional case \cite{Ihlenburg:Babuska:1995,Ihlenburg:Babuska:1997} that when $kh\leq\tilde\alpha<\pi$ 
(for a positive constant $\tilde\alpha$), the $H^1$-norm error of the linear FE solution on equidistant mesh is less than $O(kh+k^3h^2)$, where $h$ is the mesh size. The first term is of the same order as the well-known interpolation error, while 
the second term is the so-called pollution error term, which would dominate if $k^2h$ is large. 
An error estimate containing a non-negligible pollution term is named preasymptotic error estimate, otherwise it is called asymptotic error estimate.
We refer to \cite{Zhu:Wu:2013,Wu:2013,Du:Wu:2015,Graham:Sauter:2020,Graham:Spence:Zou:2020} for preasymptotic error estimates of FEMs for linear Helmholtz equations. 
The $hp$-FEM was shown to be pollution-free if its order is proportional to $\ln k$ \cite{Melenk:Sauter:2010,Melenk:Sauter:2011}. 
For NLH equations with the impedance boundary or perfect match layer (PML), 
{\it a priori} preasymptotic error estimates was developed \cite{Wu:Zou:2018,Jiang:Li:Wu:Zou:2022}.

While the advances in the {\it a priori} error analysis are significant, low regularity and high oscillation in the exact solutions to the NLH system \eqref{eq:Helm}--\eqref{eq:Dir} make a practically highly unacceptable impact on the efficiency and accuracy of FEMs over uniformly refined meshes. This impact may be attributed to three important factors: (1) singularities with the exact solutions around reentrant corners of the boundary (see Fig.\,\ref{fig:region}); (2) a locally rapid change inherent in the exact solutions due to the nonlinear term $\abs{u}^2u$ in \eqref{eq:Helm}; (3) the discontinuity across the interface $\pa \Omz$.
Some adaptive techniques, especially those featuring {\it a posteriori} error estimates and local mesh refinements, 
are proved to be very effective to improve the numerical accuracies. This is the main motivation of the current work, 
to couple FEMs with adaptivity for the NLH system \eqref{eq:Helm}--\eqref{eq:Dir} to reduce the computational cost and 
maintain the accuracy.
Since the groundbreaking work \cite{Babuska:Rheinboldt:1978} by Babu\v{s}ka and Rheinboldt in 1978, adaptive finite element methods (AFEMs) based on {\it a posteriori} error estimates have played a pivotal role in the realm of scientific and engineering computing. For positive definite elliptic problems, the theory of AFEMs has been well developed 
(see, e.g., \cite{Morin:Nochetto:Siebert:2002,Mekchay:Nochetto:2005,Cascon:Kreuzer:Nochetto:Siebert:2008}). 
For linear Helmholtz equations with high wave number, there are already some studies; see \cite{Dorfler:Sauter:2013,Zhou:Zhu:2015,Bespalov:Haberl:Praetorius:2017,Chaumont-Frelet:Ern:Vohralik:2021,Duan:Wu:2023}. In particular, Duan and Wu \cite{Duan:Wu:2023} proved the convergence and quasi-optimality of a linear AFEM with initial mesh size in the preasymptotic regime, which is slightly different from the standard adaptive algorithm, i.e., in each adaptive iteration, slightly more elements need to be refined than the marked elements. Moreover, it is proved that the {\it a posteriori} error estimator seriously underestimates the error of the FE solution in the preasymptotic regime, which was first observed by  Babu\v{s}ka, {\it et al.} \cite{Babuska:Ihlenburg:Strouboulis:Gangaraj:1997}. We remark that the usual techniques based on Galerkin orthogonality (or quasi-othogonality) for AFEM can only give convergence and quasi-optimality of AFEMs for linear Helmholtz equations in asymptotic regime \cite{Zhou:Zhu:2015}.  An ``elliptic projection argument" was proposed in \cite{Duan:Wu:2023} to derive preasymptotic estimates for AFEMs, 
which first builds some equivalent relationships between the {\it a posteriori} error estimators of the FE solution $u_n$ 
produced by the AFEM and the elliptic projections $\tilde u_n$ of the exact solution, and 
then convert the analysis of $u_n$ to that of $\tilde u_n$.
As regards AFEMs for nonlinear elliptic problems, major efforts have been concentrated on quasi-linear problems of $p$-Laplacian and strongly monotone type over the past decades; see e.g., \cite{Veeser:2002,Garau:Morin:Zuppa:2012,Garau:Morin:Zuppa:2011,Haberl:Praetorius:Schimanko:Vohralik:2021,Heid:Wihler:2020a}.
For the semilinear elliptic equations with the Dirichlet boundary condition and 
the uniformly elliptic diffusion part, 
there were some studies on the convergence and the optimality of AFEMs; see 
\cite{Amrein:Heid:Wihler:2023,Becker:Brunner:Innerberger:Melenk:Praetorius:2022,Becker:Brunner:Innerberger:Melenk:Praetorius:2023}. 
It should be pointed out that the nonlinear term $g(x,\cdot)$ in these works were assumed to be monotone  \cite{Becker:Brunner:Innerberger:Melenk:Praetorius:2022,Becker:Brunner:Innerberger:Melenk:Praetorius:2023}
or to possess asymptotic linear growth in the infinity \cite{Amrein:Heid:Wihler:2023} with respect to the second  argument, i.e., $\partial_{\xi}g(x,\xi)$ is nonnegative or bounded from above and below to be precise. As a result, the associated nonlinear operator allows one to use the monotone operator theory in functional analysis or the monotone method of sub- and supersolutions in the theory of semilinear PDEs. However, the monotonicity or linear growth conditions fail for the nonlinear problem \eqref{eq:Helm}--\eqref{eq:Dir} of our interest in this work. 
So our theoretical developments and analyses here are essentially different from those in the existing references, and several new strategies will play a pivotal role in our analyses.

The first focus of this work is to establish a preasymptotic {\it a posteriori analysis}. 
One of the key developments in the analysis is to extend the ``elliptic projection argument" proposed in \cite{Duan:Wu:2023} for the linear Helmholtz problem to the nonlinear problem \eqref{eq:Helm}--\eqref{eq:Dir}. 
To concentrate on the main technical developments and results, 
we shall only focus on the two-dimensional case in this work, 
as the singularity decomposition and some crucial analytic tools and analyses for the theoretical developments 
here will be much more complicated in three dimensions. 
To put it in perspective, our major theoretical developments involve three main parts. 
First, we present some preliminaries including wave-number-explicit stability estimates and a singularity decomposition for the NLH problem and its linear FEM along with {\it a priori} preasymptotic error estimates on shape regular meshes, 
which is crucial to the subsequent AFEMs. Then {\it a posteriori} upper and lower bounds (reliability and efficiency) of residual-type are derived for the FE solutions to the NLH problem. 
For the purpose, it is important that we will adopt 
an error estimator on a finite element that is an elegant modification of the standard one, namely constructed by adding contributions of element residuals of neighbouring elements that have a common edge with this element 
(see  \S\,\ref{s:error estimators}), although the two error estimators are equivalent on the entire finite element mesh. 
With all these preparations, we will be also able to demonstrate that the key observation by Babu\v{s}ka {\it et al.} \cite{Babuska:Ihlenburg:Strouboulis:Gangaraj:1997} for a one-dimensional linear Helmholtz problem still holds for the current NLH problem, that is, the residual type error estimator seriously underestimates the true error of the FE solution for $h$ in the preasymptotic regime. 
As the last development, based on the above-mentioned modified error estimator, the convergence as well as the quasi-optimality of the standard adaptive algorithm will be established when the initial mesh size lies in the preasymptotic regime.

We remark that it is the first time to develop an AFEM for the NLH problem here, 
and with a systematic analysis on the convergence and quasi-optimality of the AFEM. 
As we shall see, the theoretical developments here for the NLH system 
are highly nontrivial and much more technical in comparison with those for the linear Helmholtz problem, 
due to some substantially technical differences and difficulties as addressed above and 
those more as described below.
The first major difficulty comes from the treatment of the nonlinearity in \eqref{eq:Helm}. 
For this, we need a wave-number-explicit interior $L^\infty$ estimate for the FEM, which is crucial for both {\it a priori} and {\it a posteriori} estimates. The existing studies of NLH problems hold only for quasi-uniform meshes (see \cite[Lemma~3.4, Theorem~3.6]{Wu:Zou:2018},\cite[Lemma~4.3, Theorem~4.6]{Jiang:Li:Wu:Zou:2022}). There exist some results on non-quasi-uniform meshes for FEMs for 
standard linear elliptic problems (see \cite{Li:2017}), but they require that the meshes obey some particular pattern which may not be satisfied by the locally refined meshes generated in general adaptive finite element iterations.  By using the elliptic projection of the exact solution as a bridge,  the inequality $\Linf{w_h}\ls |\ln h_{{\cal T}, {\rm min}}|^{1/2} \|w_h\|_{1, \Omega_0}$
on FE spaces, and through some subtle estimations, we are able to achieve a desired interior $L^\infty$ estimate for the discrete NLH problem on general regular meshes (see Lemma~\ref{lem:interior-estimate} and Theorem~\ref{thm:dpstab}). 
We think that similar arguments can be also used to help improve the interior $L^\infty$ estimates for FEMs for 
standard elliptic problems.
Another key challenge is the establishment of the convergence and quasi-optimality of 
our new adaptive algorithm, based on our modified error estimator, which does not satisfy 
the estimator reduction property as for the elliptic problems
(see, e.g., \cite[Corollary 3.4]{Cascon:Kreuzer:Nochetto:Siebert:2008}). 
We have to introduce a new auxiliary equivalent error estimator by carefully adjusting the weights of the contributions of element residuals of neighbouring elements, to make the application of the ``elliptic projection argument" possible (see  \S\,\ref{s:4} for details).

The rest of this paper is organized as follows. In \S\,\ref{s:2}, we derive some {\it a priori} stability estimates for the continuous and discrete NLH problems along with preasymptotic error estimates on shape regular meshes for the later. 
  \S\,\ref{s:3} is devoted to {\it a posteriori} upper and lower bounds for the errors of the FEM. By establishing some equivalent relationships between error estimators for the FE solution and the elliptic projection, the convergence and quasi-optimality of the proposed AFEM are proved in   \S\,\ref{s:4}.
In   \S\,\ref{s:5}, 
we firstly introduce the continuous interior penalty finite element method (CIPFEM) to reduce the pollution errors, then provide some numerical examples to verify the performances of the AFEM and adaptive CIPFEM (ACIPFEM), at last we  simulate the optical bistability phenomenon with Gaussian incident waves by using the ACIPFEM.

For the simplicity of notation, $C$ is used to denote a generic positive constant that does not depend on $k,\vep, h, f,g$,  and the penalty parameters. Additionally we use the abbreviated tokens $A\ls B$ and $B\gtrsim A$ for the inequality $A\leq CB$. $A\eqsim B$ is a notation for the situation that $A\ls B$ and $A\gtrsim B$. We also use $\theta_0$ and $C_0$ to denote two generic positive constants independent of $k,\vep, h, f,g$,  and the penalty parameters, which are sufficiently small but allow  different values to be taken in different lemmas or theorems. For example, the condition ``$k^3h^{1+\alpha}\le C_0$" means that ``$k^3h^{1+\alpha}$ is less than or equal to some sufficiently small constant independent of $k,\vep, h, f,g$,  and the penalty parameters". 
We also use the standard Sobolev spaces, norms and inner products as in 
\cite{Brenner:Scott:2008}. In particular, $(\cdot,\cdot)_Q$ and $\ine{\cdot}{\cdot}$ denote the $L^2$-inner product on complex-valued $L^2(Q)$ and $L^2(e)$ spaces, respectively.
We shall often shorthand the subscript index of the norm and semi-norm of the Sobolev space $H^s(\Om)$ by $\norm{\cdot}_{s}$ and $\abs{\cdot}_{s}$. When the function space is not in the whole domain, we may add the domain or the boundary 
in the subscript like $\norm{\cdot}_{s,\Omz}$ or $\norm{\cdot}_{s,\GaI}$.
In particular, we shall use the energy norm $\He{\cdot}:=\big(\abs{\cdot}_1^2+k^2 \norm{\cdot}_0^2\big)^\frac12$ on $H^1(\Om)$.

\section{Preliminaries}\label{s:2}

In this section, for the purpose of the {\it a posteriori} analysis of the FEM for the NLH equation, we first introduce an auxiliary linearized problem and give its wave-number-explicit stability estimates and an decomposition of the solution into a regular part and a singular part. Then we derive stability estimates for the NLH by considering its Newton's iteration. Afterwards, we similarly derive {\it a priori} stability and preasymptotic error estimates for the finite element discretization of the NLH equation.

We begin with some geometric assumptions (see figure \ref{fig:region}). 
We assume that angles of the corners on $\GaI$ are less than $\pi$ and either $\overline\GaI\cap\overline\GaD=\emptyset$ or the angles where $\GaD$ and $\GaI$ meets are less than $\pi/2$. It means the singularities only can appear at the corners on $\GaD$. The reentrant angles on $\GaD$ are denoted by $\{\omega_i\}_{1\leq i\leq N}$ and $\omega_{\rm max}$ is defined as $\max_{1\leq i\leq N} \omega_i$, where $N$ is the number of non-convex vertices of $\GaD$. 
We also assume that $\Omz$ is a domain with Lipschitz boundary, and that $\Om$ is strictly ``star-shaped'', i.e., there exists a point $x_0\in \R^2$ and two constants $\gamma_{\rm Dir}, \gamma_{\rm imp}>0$ such that
  \eq{\label{eq:star shaped}
    \begin{aligned}
    (x-x_0)\cdot n(x)&\leq -\gamma_{\rm Dir} \quad \forall x\in \GaD;\quad (x-x_0)\cdot n(x)&\geq \gamma_{\rm imp} \quad \forall x\in \GaI.
  \end{aligned}
  }

\subsection{An auxiliary linearized problem}
The variational formulation of \eqref{eq:Helm}--\eqref{eq:Dir} reads: find $u\in H^1_{\GaD}(\Omega):=\{v\in H^1(\Om): v|_{\GaD}=0\}$, such that
\eq{a_{u}(u,v)=(f,v)+\langle g,v\rangle_{\GaI} \quad \forall v\in \HoG,\label{eq:NLHva}}
where $a_{\psi}(u,v):=(\na u,\na v)-k^2\big((1+\vep\oneo\abs{\psi}^2)u,v\big)+\mbi k\langle u,v\rangle_{\GaI}.$
For given $\psi\in L^\infty(\Om_0)$ and $F\in L^2(\Omega)$, consider the following auxiliary linearized problem:
find $w_\psi\in \HoG$ such that
\eq{\label{eq:auxiliary cp}
  a_0(w_\psi,v)-k^2\vep\big(2\abs{\psi}^2w_\psi+\psi^2\overline{w_\psi},v\big)_{\Omz}=(F,v)+\langle g,v\rangle_{\GaI}\quad \forall v\in \HoG.
}


Next, we recall the stability estimate and  singularity decomposition for the auxiliary problem with $\psi=0$, i.e., the linear Helmholtz problem $-\Delta w_0-k^2 w_0=F$ with boundary conditions \eqref{eq:Imp}--\eqref{eq:Dir}. 
Let $\chi (r)$ be a $C^{\infty}$ cutoff function that equals $1$ in a neighborhood of the origin and has a support sufficiently close to origin such that $\chi(r)$ centered at every reentrant vertice vanish at other corners and in $\Omz$.
The corresponding local polar coordinates system at $x_j$ is $(r_j,\theta_j)$ such that the two sides of the reentrant corner at $x_j$ locate on $\theta_j=0$ and $\theta_j=\omega_j$, respectively.
\begin{lemma}\label{lem:linear problem:singularity decomposition}
 There exists a unique solution $w_0$ to \eqref{eq:auxiliary cp} with $\psi=0$ satisfying
  \begin{equation}\label{eq:linear problem:stability}
    \He{w_0}\ls \MFg,
  \end{equation}
  and $w_0$ has a decomposition: $w_0=w_{R}+\sum_{j=1}^NC_k^jS_j$ where $w_R\in \HoG\cap H^2(\Om)$, $S_j(x)=\chi(r_j)r_j^{\alpha_j}\sin(\alpha_j\theta_j)$ with $\alpha_j=\frac{\pi}{\omega_j}$, and the constants $C_k^j\in \mathbb{C}$. Moreover,
  \eqn{
    \begin{aligned}
     \Ht{w_R}\ls k\hMFg\qaq \big\vert C_k^j\big\vert \ls k^{\alpha_j-\frac12}\MFg,
    \end{aligned}
  }
  where $\MFg=\Lt{F}+\LtD{g}{\GaI}$, $\hMFg=\MFg+k^{-1}\left\Vert g\right\Vert_{\frac 12,\GaI}.$
\end{lemma}
\begin{proof}
  For \eqref{eq:linear problem:stability}, we refer to \cite{Cummings:Feng:2006,Hetmaniuk:2007}.
  On the other hand, the singularity decomposition comes form \cite{Chaumont-Frelet:Nicaise:2018,Duan:Wu:2023}.
\end{proof}

The following lemma gives estimates of the auxiliary problem with nonzero $\psi$.
\begin{lemma}\label{lem:auxiliary stability}
If $k\vep\Linf{\psi}^2\leq \theta_0$, then the solution to \eqref{eq:auxiliary cp} satisfies
  \begin{equation}\label{eq:auxiliary problem stability}
    \He{w_\psi}\ls \MFg \qaq \Linf{w_\psi}\ls k^{-\frac12}\MFg.
  \end{equation}
Moreover, $w_\psi$ has a decomposition $w_\psi=w_{R}+\sum_{j=1}^N C_k^jS_j$ where $w_R\in \HoG\cap H^2(\Om)$ and $C_k^j\in \mathbb{C}$ satisfy:
\eq{\label{eq:auxiliary problem decomposition}
  \begin{aligned}
   \Ht{w_R}\ls k\hMFg \qaq \big\vert C_k^j\big\vert \ls k^{\alpha_j-\frac12}\MFg.
  \end{aligned}
}
\end{lemma}
\begin{proof}
  We rewrite \eqref{eq:auxiliary cp} as
  \eqn{a_{0}(w_\psi,v)=(\wt F,v)+\langle g,v\rangle_{\GaI}\text{ with }\wt F:=F+k^2\vep\oneo(2\abs{\psi}^2w_\psi+\psi^2\overline{w_\psi}), \quad \forall v\in \HoG.}
  Then by Lemma \ref{lem:linear problem:singularity decomposition},
  \eqn{
    \He{w_\psi}\ls M(\wt F,g)\ls \MFg+k\vep\Linf{\psi}^2k\Lt{w_\psi}\ls \MFg+\theta_0\He{w_\psi},
  }
which implies the first estimate in \eqref{eq:auxiliary problem stability} if $\theta_0$ is sufficiently small. Moreover, $w_\psi$ has a decomposition $w_\psi=w_R+C^j_kS_j$ satisfying the estimates:
  \eqn{
    \begin{aligned}
    \Ht{w_R}&\ls k\widehat{M}(\wt F,g)\ls k\hMFg+kk\vep\Linf{\psi}^2k\Lt{w_\psi}\ls k\hMFg \quad \text{and}\\
    \big\vert C^j_k\big\vert&\ls k^{\alpha_j-\frac12}M(\wt F,g)\ls k^{\alpha_j-\frac12}\MFg.
    \end{aligned}
  }
That is, \eqref{eq:auxiliary problem decomposition} holds.

 Finally, the second estimate in \eqref{eq:auxiliary problem stability} may be proved by following the proof of \cite[Lemma 2.3]{Wu:Zou:2018} but an addition term $\big\|\frac{\pa w_\psi}{\pa n}\big\|_{\GaD}$ appears,
 which can be estimated by following the proof of \cite[Proposition 3.3]{Hetmaniuk:2007}. Here we omit the details.
The proof is completed.
\end{proof}

\subsection{Estimates of the NLH problem}

In this subsection, we derive estimates for the NLH solution of \eqref{eq:NLHva} by Newton's iteration (see, e.g.,\cite{Jiang:Li:Wu:Zou:2022}), which reads: given an initial guess $u^0\in \HoG$, find $u^{l+1}\in \HoG$ for $l=0,1,2,\ldots$ such that
\eq{\label{eq:iteration2var}
  a_0(u^{l+1},v)-k^2\vep\big(2\big\vert u^l\big\vert^2 u^{l+1}+(u^l)^2\overline{u^{l+1}},v\big)_{\Omz}=(f^l,v)+\langle g,v\rangle_{\GaI}\quad \forall v\in \HoG,
}
where $f^l:=f-2k^2\vep\oneo\big\vert u^l\big\vert^2 u^l.$

We show the stability results and singularity decompositions of $\{u^l\}$ in the following Lemma.

\begin{lemma}\label{lem:iteration stability}
 If
  \eq{\label{u0}
    k\norm{u^0}_{0,\Omega_0}\ls\Mfg,\quad \Linf{u^0}\ls k^{-\frac12}\Mfg,\qaq\vep\Mfg^2\leq\theta_0,
  }
   then the Newton's iterative sequence $\{u^l\}(l=1,2,\ldots)$ from \eqref{eq:iteration2var} satisfies
  \begin{equation}\label{eq:iteration C stability}
    \He{u^l}\ls \Mfg \qaq \Linf{u^l}\ls k^{-\frac12}\Mfg.
  \end{equation}
Moreover, $u^l$ has a decomposition $u^l=u^l_{R}+\sum_{j=1}^N C_k^{j,l}S_j$ where $u_R\in \HoG\cap H^2(\Om)$ and the constants $C_k^{j,l}\in \mathbb{C}$ with the estimates:
\eq{\label{eq:iteration decomposition}
  \begin{aligned}
   \Ht{u^l_R}\ls k\hMfg \qaq \big\vert C_k^{j,l}\big\vert \ls k^{\alpha_j-\frac12}\Mfg.
  \end{aligned}
}
\end{lemma}
\begin{proof}
  We prove this lemma by induction. We let $C_t, C_{\infty}, C_r, C_s, C_t^0, C_\infty^0$ be the hidden constants in \eqref{eq:auxiliary problem stability}, \eqref{eq:auxiliary problem decomposition}, and \eqref{u0}, respectively.
  Denote by $\wt C_t=\max\{2C_t,C_t^0\},\wt C_{\infty}=\max\{2C_{\infty},C_\infty^0\},\wt C_r=2C_r$, and $\wt C_s=2C_s$.    Suppose $\vep\Mfg^2\leq\hat\theta := \min \Big\{ \frac{\theta_0}{\wt C_{\infty}^2}, \frac1{2\widetilde C_t \widetilde C_\infty^2}\Big\}$ with $\theta_0$ from Lemma~\ref{lem:auxiliary stability}.
Next we prove by induction that \eqref{eq:iteration C stability} and \eqref{eq:iteration decomposition} hold with hidden constants $\wt C_t, \wt C_{\infty}$ and $\wt C_r, \wt C_{s}$, respectively. It suffices to prove them under the conditions:
 \eq{\label{ul-1}
    k\norm{u^{l-1}}_{0,\Omega_0}\le\wt C_t\Mfg,\quad \Linf{u^{l-1}}\le k^{-\frac12}\wt C_{\infty}\Mfg,}
  since \eqref{ul-1} holds for $l=1$. Using \eqref{ul-1}, we have
  \eq{\label{eq:fl-1}M(f^{l-1},g)\leq (1+2\wt C_t\wt C_{\infty}^2\hat\theta)\Mfg\leq 2\Mfg.}
  Therefore, with the aid of Lemma \ref{lem:auxiliary stability}, we can get
  \eqn{
    &\He{u^l}\leq C_t M(f^{l-1},g)\leq \wt C_t\Mfg\qaq
    \Linf{u^l}\leq k^{-\frac12}C_{\infty}M(f^{l-1},g)\leq k^{-\frac12}\wt C_{\infty}\Mfg,
  }
  and $u^l=u^l_{R}+\sum_{j=1}^N C_k^{j,l}S_j$ with the estimates: 
  \eqn{
    \Ht{u^l_R}\leq kC_r\widehat M(f^{l-1},g)\leq k\wt C_r\hMfg \qaq \big\vert C_k^{j,l}\big\vert \leq k^{\alpha_j-\frac12}C_sM(f^{l-1},g)\leq k^{\alpha_j-\frac12}\wt C_s\Mfg.
  }
  The proof is completed.
\end{proof}

Then the well-posedness of \eqref{eq:NLHva} is derived as below:
\begin{theorem}\label{thm:cpstab}
  If $\vep\Mfg^2\leq\theta_0$,
  then \eqref{eq:NLHva} attains a unique solution $u$ satisfying the stability estimates
  \eq{\label{eq:cpstab}
  \He{u} \ls  \Mfg\qaq\Linf{u}\ls k^{-\frac12}\Mfg.
  }
  Moreover, u has a decomposition $u=u_{R}+\sum_{j=1}^N C_k^jS_j$ where $u_R\in \HoG\cap H^2(\Om)$ and the constants $C_k^j\in \mathbb{C}$ with the estimates:
\eq{\label{eq:NLH decomposition}
  \begin{aligned}
   \Ht{u_R}\ls k\hMfg \qaq \big\vert C_k^j\big\vert \ls k^{\alpha_j-\frac12}\Mfg.
  \end{aligned}
}
\end{theorem}
\begin{proof}
  Let $v^l=u^{l+1}-u^l$, through \eqref{eq:iteration2var} and simple computations, we get for $l\geq 1$
  \eqn{
    \begin{aligned}
      a_0(v^l,v)-k^2\vep\big(2\abs{u^l}v^l+(u^l)^2\overline v^l,v\big)_{\Omz}=k^2\vep\big(((u^l)^2-(u^{l-1})^2)\overline u^l-2\abs{u^{l-1}}^2v^{l-1},v\big)_{\Omz} \quad \forall v\in \HoG.
    \end{aligned}
  }
  Then with the aid of Lemma \ref{lem:auxiliary stability} and \eqref{eq:iteration C stability}, we have
  \eqn{
    \begin{aligned}
      \He{v^l}&\leq \wt C_t k^2\vep\big\Vert u^l\big(|u^l|^2-|u^{l-1}|^2\big)\big\Vert_{0,\Omz}+2\wt C_tk^2\vep\big\Vert|u^{l-1}|^2v^{l-1}\big\Vert_{0,\Omz}\\
      &\leq \wt C_tk^2 \vep \big \Vert u^l \big \Vert_{L^\infty(\Omz)} \big \Vert |u^{l}|+|u^{l-1}| \big \Vert_{L^\infty(\Omz)} \big \Vert v^{l-1} \big \Vert_{0} + 2\wt C_tk^2 \vep \big \Vert u^{l-1} \big \Vert_{L^\infty(\Omz)}^2 \big \Vert v^{l-1} \big \Vert_{0}\\
      &\leq \wt C \vep\Mfg^2\He{v^{l-1}}.
    \end{aligned}
  }
  So we let $\vep\Mfg^2$ be sufficiently small to get
  \eqn{\He{v^l}\leq \frac12 \He{v^{l-1}},}
  which means $\{u^l\}$ is a Cauchy sequence in the energy norm.

  From the singularity decompositions in Lemmas \ref{lem:auxiliary stability}--\ref{lem:iteration stability}, we know $v^l = u^{l+1}-u^{l} = u^{l+1}_R-u^l_R+\sum_{j=1}^N(C^{j,l+1}_k-C^{j,l}_k)S_j = v^l_R+\sum_{j=1}^N\wt C^{j,l}_kS_j$ and deduce
  \eqn{
    \begin{aligned}
      \Ht{v^l_R} \ls k \left(k^2\vep\big\Vert u^l\big(|u^l|^2-|u^{l-1}|^2\big)\big\Vert_{0,\Omz}+2 k^2\vep\big\Vert|u^{l-1}|^2v^{l-1}\big\Vert_{0,\Omz}\right) \leq C\theta_0 k k \Lt{v^{l-1}} \leq C\theta_0 k \He{v^{l-1}},
    \end{aligned}
  }
  which implies the sequence $\{u^l_R\}$ is a Cauchy sequence in $H^2(\Om)$, then $u_R:=\lim\limits_{l\to\infty}u^l_R$.

  Similarly, we can get $\big\vert\wt C^{j,l}_k\big\vert\leq Ck^{\alpha_j-\frac12}k\Lt{v^{l-1}}$, then $\lim\limits_{l \to \infty} \wt C^{j,l}_k=0$. Consequently, $C^j_k:=\lim\limits_{l \to \infty}C^{j,l}_k$.
  Then $u=u_R+\sum_{j=1}^NC^j_kS_j$ and the two components satisfying \eqref{eq:NLH decomposition}.

  Finally considering uniqueness, we assume that there exists another solution of \eqref{eq:NLHva}, denoted by $u^\ast$. Since \eqref{eq:iteration2var} holds for both $u$ and $u^\ast$, letting $w=u-u^\ast$ we may obtain as above
  \[
    \begin{aligned}
      a_0(w,v) - k^2 \vep \big( 2 \abs{u}^2 w + u^2 \overline w,v \big)_{\Omz} = k^2 \vep \big(( u^2-(u^{\ast})^2 ) \overline u^\ast - 2 \abs{u}^2 w,v \big)_{\Omz} \quad \forall v\in \HoG.
    \end{aligned}
  \]
  Then the same argument as for $\He{v^l}\leq \frac12 \He{v^{l-1}}$ implies that $\He{w}\leq \frac12 \He{w}$ when $\vep\Mfg^2$ is sufficiently small. Thus the uniqueness follows with $u=u^\ast$.
\end{proof}

\begin{remark}\label{rm:multiple solutions}
  {\rm (i)} Theorem~\ref{thm:cpstab} states that under the condition $\vep\Mfg^2\leq\theta_0$ indicating that the nonlinearity may not be too strong, the NLH problem attains a unique low-energy solution satisfying the stability estimates and the singularity decomposition. However this result does not exclude the possibility of multiple solutions to the NLH
  problem, nor does it cover the case of strong nonlinearity. The quadratic convergence of Newton's iteration \eqref{eq:iteration2var} when the initial guess is sufficiently close to the exact solution is proved in \cite{Jiang:Li:Wu:Zou:2022}, where the boundary condition is the trucated PML. For the impedance boundary condition, the proof is similar and the resulting quadratic convergence is valid for the high-energy solutions. The details are omitted.

  {\rm (ii)} Here are two other iteration methods.
  The first and simplest one is the frozen-nonlinearity iteration: given $u^0\in\HoG$, find $u^{l+1}\in \HoG$ for $l=0,1,2,\ldots$ such that
  \eqn{
  a_{u^{l}}(u^{l+1},v)=(f,v)+\langle g,v\rangle_{\GaI}\quad \forall v\in \HoG.
  }
 In \cite{Wu:Zou:2018,Maier:Verfurth:2022}, the linear convergence of the iteration sequence for the low-energy solution is derived, that is to say, this iteration converges when $\vep\Mfg^2$ is small enough.

 The another one is the modified Newton's method proposed by \cite{Yuan:Lu:2017} which utilizes $\overline{u^{l+1}}$ in \eqref{eq:iteration2var} instead of $\overline{u^l}$:
given $u^0\in\HoG$, find $u^{l+1}\in \HoG$ for $l=0,1,2,\ldots$ such that
 \eqn{
   a_{\sqrt 2u^{l}}(u^{l+1},v)=(f-k^2\vep\oneo\vert u^l\vert^2u^l,v)+\langle g,v\rangle_{\GaI}\quad \forall v\in \HoG.
 }
 Numerical results show that the modified Newton's method is only linearly convergent but robust with respect to
 the initial guess. 

\end{remark}

\subsection{{\it A priori} estimates of the FEM}\label{s:FEM error estimate}
In this subsection, we derive {\it a priori} stability and preasymptotic error estimates of the linear FEM for the NLH problem on shape regular meshes.

Let $\T$ be a conforming shape regular triangulation of $\Omega$. For simplicity, we assume that the triangulation $\T$ fits the interfaces $\pa\Om_0$, that is, $\pa\Om_0$ does not pass through the interior of any element $K\in\T$.
For any triangle element $T\in\T$, let $h_T:=|T|^\frac12$ be the size of $T$. Denote by $h_\T:=\max_{ T\in\T}h_T$ and $h_{\T,{\rm min}}:=\min_{T\in\T, T\subset\Omz}h_T$. The linear finite element space on $\T$ is denoted by
\eqn{ V_\T :=\{ v \in \HoG :\  v|_{T} \in P_1(T)  \quad \forall T \in \T \}.}
The linear FEM for \eqref{eq:NLHva} reads as: find $u_h\in V_\T$ such that
\eq{a_{u_h}(u_h,v_h)=(f,v_h)+\langle g,v_h\rangle_{\GaI} \quad \forall v_h\in V_\T.\label{eq:FEMva}}

The corresponding auxiliary discrete problem and Newton's iteration read as: find $w_{h}, u_h^{l+1}\in V_\T$ such that
\begin{align}
  a_0(w_{h},v_h)-k^2\vep\big(2\big\vert \psi\big\vert^2 w_{h}+\psi^2\overline{w_{h}},v_h\big)_{\Omz}=(F,v_h)+\langle g,v_h\rangle_{\GaI}\quad \forall v_h\in V_\T,  \label{eq:auxiliary dp1}\\
  a_0(u^{l+1}_h,v_h)-k^2\vep\big(2\big\vert u^l_h\big\vert^2 u^{l+1}_h+(u^l_h)^2\overline{u^{l+1}_h},v_h\big)_{\Omz}=(f^l_h,v_h)+\langle g,v_h\rangle_{\GaI}\quad \forall v_h\in V_\T,\label{eq:iteration dp2}
\end{align}
where $f^l_h:=f-2k^2\vep\oneo\vert u^l_h\vert^2u^l_h$ and $u_h^0\in V_\T$. The other two iterations can be similarly constructed.

Next we introduce the standard FE interpolation operator $I_h$ onto $V_\T$ (see, e.g., \cite[\S 3.3]{Brenner:Scott:2008}). For the singular functions $S_j$, the interpolation error estimates in \cite[Lemma 5.1]{Chaumont-Frelet:Nicaise:2018} are as follows:
\eq{\label{eq:singular interpolation error}
\Lt{S_j-I_hS_j}+h_\T\Ho{S_j-I_hS_j}\ls h^{1+\alpha_j}_\T,\quad j=1,2,\ldots,N.
}
For any $\phi\in \HoG$, its elliptic projection $P_h\phi\in V_\T$ is defined by
\eq{\label{eq:elliptic projection definition}
  b(P_h\phi,v_h)=b(\phi,v_h)\ \ \forall v_h\in V_\T, \ b(u,v):=(\na u,\na v)+(u,v).}
Naturally, the following Galerkin orthogonality holds 
\eq{\label{eq:GO}
b(\phi-P_h\phi,v_h)=0\quad \forall v_h\in V_\T,
}
 and we have the following projection error estimates (see, e.g., \cite[Lemma 2.4]{Duan:Wu:2023}):
  \eq{\label{eq:elliptic projection estimate}
  h_\T^{-\alpha}\Lt{\phi-P_h\phi}+h_\T^{-\frac{\alpha}2}\LtD{\phi-P_h\phi}{\GaI}\ls\Ho{\phi-P_h\phi}=\inf_{v_h\in V_\T}\Ho{\phi-v_h},
  }
  where $\alpha:=\frac{\pi}{\omega_{\rm max}}=\min\limits_{1\leq j\leq N}\alpha_j$. 
  Since $b(u,v)$ is symmetric, it also holds that $b(v_h,P_h\phi)=b(v_h,\phi)\ \ \forall v_h\in V_\T$, and the similar Galerkin orthogonality.
In particular, we show the relationship between $a_\psi(\cdot,\cdot)$ and $b(\cdot,\cdot)$ for convenience,
\eq{\label{eq:a to b relationship}
  a_\psi(u,v)=b(u,v)-(k^2+1)(u,v)-k^2\vep\big(\abs{\psi}^2u,v\big)_{\Om_0}+\mbi k\langle u,v\rangle_{\GaI}.
}

For simplicity, denote by
\eq{\label{eq:t}
t(k,h):=k^{\frac12}h^{\frac{\alpha}2}+k^{\frac12+\alpha}h^\alpha+k^{\frac32}h^{\frac{1+\alpha}2}+k^2h^{1+\frac{\alpha}2}+k^{\frac32+\alpha}h^{2\alpha}+k^3h^{1+\alpha}.}
Noting that 
 \eqn{
  t(k,h)&=k^{\frac{1-2\alpha}{2(1+\alpha)}}(k^3h^{1+\alpha})^{\frac{\alpha}{2(1+\alpha)}}+k^{\frac{(2\alpha-1)(\alpha-1)}{2(1+\alpha)}}(k^3h_\T^{1+\alpha})^{\frac{\alpha}{1+\alpha}}+(k^3h^{1+\alpha})^{\frac12}\\
  &\quad +(kh)^{\frac12}(k^{\frac32}h^{\frac{1+\alpha}2})+k^{\frac{(2\alpha-1)(\alpha-3)}{2(1+\alpha)}}(k^3h^{1+\alpha})^{\frac{2\alpha}{1+\alpha}}+k^3h^{1+\alpha},
 }
 we know that
 \eq{\label{tkh}
 t(k,h)\ls 1 \quad\text{if}\quad k^3 h^{1+\alpha}\ls 1.}
  Moreover, $t(k,h)$ is sufficiently small on condition that $k^3 h^{1+\alpha}$ is small enough. 

The following lemma gives stability estimate in energy norm and the interior stability estimate in $L^{\infty}$ norm for $w_h$ on locally refined meshes, which are discrete versions of the estimates in \eqref{eq:auxiliary problem stability} for $w_\psi$. Since the triangulation in the current case is only locally quasi-uniform, the proof of $L^{\infty}$ interior stability estimate in \cite{Wu:Zou:2018} is not valid. 
\begin{lemma}\label{lem:interior-estimate}
  Suppose $k\vep\Linf{\psi}^2\leq \theta_0$ and $k^3h_\T^{1+\alpha}\leq C_0$. Then the solution of \eqref{eq:auxiliary dp1} satisfies
  \begin{equation}
    \He{w_h}\ls \MFg \qaq \Linf{w_h}\ls \lnhmin k^{-\frac12}\hMFg. 
  \end{equation}
  \begin{proof}
The proof of the first estimate is similar to that of \cite[Lemma 3.2]{Wu:Zou:2018} with the help of \eqref{eq:elliptic projection estimate} and \eqref{eq:singular interpolation error} due to corner singularities. We prove it in the appendix.

Next we prove the second estimate. First, by using the singularity decomposition of $w_\psi$ in Lemma~\ref{lem:auxiliary stability} and following the proof of \cite[Lemma~4.1]{Jiang:Li:Wu:Zou:2022}, which uses the modified duality argument and regards the second term in \eqref{eq:auxiliary dp1} as a small perturbation, we can derive the following error estimates in energy and $L^2$ norms (see also \cite[Theorem 5.5]{Chaumont-Frelet:Nicaise:2018}), under the condition that $k\vep\Linf{\psi}^2\le\theta_0$ and $k^3h_\T^{1+\alpha}\le C_0$,  the details are omitted. 
   \eq{\label{eq:auxiliary problem error estimate}
    \begin{split}
    \He{w_\psi-w_h}&\ls\big(kh_\T+k^{-\frac12}(kh_\T)^{\alpha}+k^3h_\T^2\big)\hMFg, \\ 
    \Lt{w_\psi-w_h}&\ls \big(k^{-1}(kh_\T)^{2\alpha}+k^2h_\T^2\big)\hMFg.
    \end{split}}
On the other hand, from \cite[Lemma 2.3]{Bramble:Xu:1991}, there holds \eq{\label{eq:xu Linf estimate}\Linf{w_h}\ls \lnhmin\HoD{w_h}{\Omz}.} Indeed, the result there can be easily verified for the non-quasi-uniform meshes by replacing $\abs{\ln h}^{\frac12}$ with $\lnhmin$.

To proceed, let $\zeta_h=w_h-P_hw_\psi$, then we have
\eq{\label{eq:Linf interior estimate}
  \Linf{w_h}\leq \Linf{w_\psi}+\Linf{w_\psi-I_h w_\psi}+\Linf{I_h w_\psi-P_hw_\psi}+\Linf{\zeta_h}.
}
The first term can be estimated by directly using \eqref{eq:auxiliary problem stability}. 
Noting that $w_\psi|_{\Omz}=w_R|_{\Omz}$ since $w_\psi=w_{R}+\sum_{j=1}^N C_k^jS_j$ and the support of $S_j(1\leq j\leq N)$ does not intersect $\Omz$, so the second term is estimated by using Lemma~\ref{lem:auxiliary stability} and the interpolation error estimate $\norm{w_R-I_hw_R}_{L^\infty(\Omz)}\ls h_\T\sHtD{w_R}{\Omz}$ (see, e.g., \cite[(4.4.8)]{Brenner:Scott:2008}):
\eq{\label{eq:Linf interior estimate term2}
  \begin{aligned}
 \Linf{w_\psi-I_h w_\psi}&=\Linf{w_R-I_h w_R}\ls kh_\T\hMFg= k^{-\frac12}(k^3h_\T^2)^{\frac12}\hMFg.
  \end{aligned}
}
To estimate the third term, we apply \eqref{eq:xu Linf estimate}, \eqref{eq:elliptic projection estimate}, \eqref{eq:singular interpolation error}, and Lemma~\ref{lem:auxiliary stability} to obtain
\eq{\label{eq:Linf interior estimate term3a}
  \begin{aligned}
 \Linf{I_h w_\psi-P_hw_\psi}
  &\ls\lnhmin\HoD{I_h w_\psi-P_hw_\psi}{\Omz}\\
  &\ls \lnhmin\big(\HoD{w_R-I_h w_R}{\Omz}+\Ho{w_\psi-P_hw_\psi}\big)\\
 &\ls \lnhmin\Ho{w_\psi-I_h w_\psi}\\
  &\ls \lnhmin k^{-\frac12}\big((k^3h_\T^2)^{\frac12}+(kh_\T)^\alpha\big)\hMFg.
  \end{aligned}
}

It remains to estimate $\Linf{\zeta_h}$. By simple calculations, we find
\eqn{
  b(\zeta_h,v_h)=\big((k^2+1+2k^2\vep\oneo\abs{\psi}^2)(w_h-w_\psi)+k^2\vep\oneo\psi^2(\overline {w_h}-\overline {w_\psi}),v_h\big)+\mbi k\langle w_\psi-w_h,v_h\rangle_{\GaI} \ \ \forall v_h\in V_\T.
  }
Then we define $\zeta\in \HoG$ as the solution of the following variational problem:
  \eq{\label{eq:zeta definition}
   \begin{aligned}
    b(\zeta,v)=\big((k^2+1+2k^2\vep\oneo\abs{\psi}^2)(w_h-w_\psi)&+k^2\vep\oneo\psi^2(\overline{w_h}-\overline{w_\psi}),v\big)\\
    &\ +\mbi k\langle w_\psi-w_h,v\rangle_{\GaI} \ \ \forall v\in \HoG.
   \end{aligned}
  }
  It is clear that $\zeta_h=P_h\zeta$. By the singularity decomposition of the elliptic equations in \cite{Grisvard:1985}, we derive $\zeta=\zeta_R+\sum_{j=1}^Nc_jS_j$, and
  \eq{\label{eq:zeta decomposition estimate}
    \begin{aligned}
      \Ht{\zeta_R}+\sum_{j=1}^N\abs{c_j}&\ls \big\Vert(k^2+1+k^2\vep\oneo\abs{\psi}^2)(w_\psi-w_h)\big\Vert_0+k\norm{w_\psi-w_h}_{\frac12,\Gamma}\\
      &\ls \big(k^2(k^{-1}(kh_\T)^{2\alpha}+k^2h^2_\T)+k(kh_\T+k^{-\frac12}(kh_\T)^{\alpha}+k^3h^2_\T)\big)\hMFg\\
      &\ls (k^4h^2_\T+k^2h_\T+k^{\frac12}(kh_\T)^{\alpha})\hMFg,
  \end{aligned}
  }
  where we use \eqref{eq:auxiliary problem error estimate} and the trace inequality. Then taking \eqref{eq:zeta decomposition estimate} into \eqref{eq:elliptic projection estimate} and using \eqref{eq:singular interpolation error} we have
  \eq{\label{eq:zeta error estimate}
    \Lt{\zeta-\zeta_h}+h_\T^{\alpha}\Ho{\zeta-\zeta_h}\ls h^{2\alpha}_{\T}(k^4h^2_\T+k^2h_\T+k^{\frac12}(kh_\T)^{\alpha})\hMFg.
  }
  Like the analyses in \eqref{eq:Linf interior estimate term2} and \eqref{eq:Linf interior estimate term3a}, we get
  \eq{\label{eq:Linf interior estimate term3}
  \begin{aligned}
    \Linf{\zeta-\zeta_h}&\ls \Linf{\zeta-I_h \zeta}+\Linf{\zeta_h-I_h\zeta}\\
&\ls h_\T\sHt{\zeta_R}+\lnhmin\big(\Ho{\zeta-\zeta_h}+\Ho{\zeta-I_h \zeta}\big)\\
&\ls h_\T\sHt{\zeta_R}+\lnhmin\Ho{\zeta-I_h \zeta}\\
&\ls h^{\alpha}_{\T}(k^4h^2_\T+k^2h_\T+k^{\frac12}(kh_\T)^{\alpha})(1+\lnhmin)\hMFg\\
&\le k^{-\frac12}\big(k^\frac32 h_\T+2k^{-\frac12}\big)t(k,h)(1+\lnhmin)\hMFg.
  \end{aligned}
  }
  After rewriting \eqref{eq:zeta definition} as: 
  \eqn{
    a_{0}(\zeta,v)-k^2\vep\big(2\abs{\psi}^2\zeta+\psi^2\bar\zeta,v\big)_{\Omz}&= \big((k^2+1+2k^2\vep\oneo\abs{\psi}^2)(P_h w_\psi - w_{\psi} + \zeta_h - \zeta),v\big)\\
    &\quad+k^2\vep\big(\vert\psi\vert^2\overline{(P_h w_\psi-w_{\psi}+\zeta_h-\zeta)},v\big)_{\Omz}\\
    &\quad+\mbi k\langle w_{\psi} - P_h w_{\psi}+\zeta-\zeta_h,v\rangle_{\GaI} \quad \forall v\in \HoG,
  }
  it follows from Lemma~\ref{lem:auxiliary stability}, \eqref{eq:elliptic projection estimate}, \eqref{eq:zeta error estimate}, the trace inequality, and the condition $k^3h_\T^{1+\alpha}\leq C_0$ that
  \eq{\label{eq:Linf interior estimate term4}
  \begin{aligned}
\Linf{\zeta}&\ls k^{-\frac12}\big(k^2\Lt{P_h w_\psi - w_{\psi} + \zeta_h - \zeta}+k\LtD{P_h w_\psi - w_{\psi} + \zeta_h - \zeta}{\GaI}\big)\\
&\ls k^{-\frac12}\big(k^2h_\T^\alpha(kh_\T+k^{-\frac12}(kh_\T)^\alpha)+k^2h^{2\alpha}_\T(k^4h^2_\T+k^2h_\T+k^{\frac12}(kh_\T)^{\alpha})\\
&\quad+kh_\T^{\frac{\alpha}2}(kh_\T+k^{-\frac12}(kh_\T)^\alpha)+kh^{\frac{3\alpha}2}_\T(k^4h^2_\T+k^2h_\T+k^{\frac12}(kh_\T)^{\alpha})\big)\hMFg\\
& \ls k^{-\frac12}\big(t(k,h)+t^2(k,h)\big)\hMFg.
  \end{aligned}
  }
  Finally, the result comes from combining \eqref{eq:Linf interior estimate}, \eqref{eq:auxiliary problem stability}, \eqref{eq:Linf interior estimate term2}, \eqref{eq:Linf interior estimate term3}, \eqref{eq:Linf interior estimate term4}, \eqref{tkh}, and the fact $\frac12<\alpha\le 1$.
  \end{proof}
\end{lemma}

Then for the Newton's iterative sequences $\{u_h^l\}$, we argue as in the proof of Lemma \ref{lem:iteration stability} to have
\begin{lemma}\label{lem:dp iteration stability}
  If $k\vep\Linf{u_h^0}^2\leq \lnhminn\vep\hMfg^2\leq\theta_0$  and $k^3h_\T^{1+\alpha}\leq C_0$, then
  \begin{equation}
    \He{u_h^l} \ls \Mfg \qaq \Linf{u_h^l}\ls \lnhmin k^{-\frac12}\hMfg \quad
    \mbox{for} ~~l=1,2,\ldots
    \end{equation}
\end{lemma}
Using the convergence of the iteration sequence, we can get:
\begin{theorem}\label{thm:dpstab}
 If $\lnhminn\vep\hMfg^2\leq\theta_0$ and $k^3h_\T^{1+\alpha}\leq C_0$,
  then \eqref{eq:FEMva} attains a unique solution $u_h$ satisfying the following stability estimates
  \eq{\label{eq:dpstab}
    \He{u_h}\ls \Mfg \qaq \Linf{u_h}\ls \lnhmin k^{-\frac12}\hMfg.}
    Moreover, there holds the error estimates
  \eq{\label{eq:error estimate}
  \begin{aligned}
    \He{u-u_h}&\ls(kh_\T+k^{-\frac12}(kh_\T)^{\alpha}+k^3h_\T^2)\hMfg,\\
    k\Lt{u-u_h}&\ls \big((kh_\T)^{2\alpha}+k^3h_\T^2\big)\hMfg.
  \end{aligned}
    }
\end{theorem}
\begin{proof}
  \eqref{eq:dpstab} can be directly obtained from the convergence of iteration methods, which can be proved by following the proof of Theorem~\ref{thm:cpstab}. We omit the details for simplicity.
  The proof of \eqref{eq:error estimate} is a little more complex but it is still based on the convergences of $\{u^l\}$ and $\{u^l_h\}$.

  We define $\hat u_h^0=u_h^0$ and $\hat u_h^l\in V_\T (l=1,2,\ldots)$ to be the solution to the following problem
  \eqn{
    a_0(\hat u^l_h, v_h)-k^2\vep\big(2\abs{u^{l-1}}^2\hat u^l_h+(u^{l-1})^2\overline{\hat u^l_h},v_h\big)_{\Omz}=(f^{l-1}, v_h)+\langle g, v_h\rangle_{\GaI}\quad \forall v_h\in V_\T.
  }
  Then from \eqref{eq:auxiliary problem error estimate} and \eqref{eq:fl-1}, we have
  \eq{\label{eq:dpstab proof1}
  \begin{aligned}
    \He{u^l-\hat u^l_h}&\ls(kh_\T+k^{-\frac12}(kh_\T)^{\alpha}+k^3h_\T^2)\hMfg,\\
        k\Lt{u^l-\hat u^l_h}&\ls \big((kh_\T)^{2\alpha}+k^3h_\T^2\big)\hMfg.
  \end{aligned}
  }
  Next we need to estimate $\hat \eta^l_h:=\hat u^l_h-u^l_h$. Through simple calculation, we find $\hat\eta_h^l$ satisfies
  \eqn{
    a_0(\hat \eta_h^l,v_h)-k^2\vep\big(2\abs{u^{l-1}}^2\hat\eta_h^l+(u^{l-1})^2\overline{\hat\eta_h^l},v_h\big)_{\Omz}&=k^2\vep\big(2\big(\abs{u^{l-1}}^2-\abs{u^{l-1}_h}^2\big)u^l_h+\big((u^{l-1})^2-(u^{l-1}_h)^2\big)\overline{u^l_h},v_h\big)_{\Omz}\\
    &\quad + 2k^2\vep\big(\abs{u^{l-1}_h}^2u^{l-1}_h-\abs{u^{l-1}}^2u^{l-1},v_h\big)_{\Omz}\qquad \forall v_h\in V_\T.
  }
  It is shown from Lemmas~\ref{lem:interior-estimate}, \ref{lem:iteration stability}, and \ref{lem:dp iteration stability} that
  \eqn{
    \He{\hat \eta^l_h}&\ls k^2\vep\LtD{\big(\abs{u^{l-1}}^2-\abs{u^{l-1}_h}^2\big)(u^l_h-u^{l-1}_h)+\big((u^{l-1})^2-(u^{l-1}_h)^2\big)\overline{u^l_h}+\abs{u^{l-1}}^2(u^{l-1}_h-u^{l-1})}{\Omz}\\
    &\ls k\vep\lnhminn\hMfg^2\Lt{u^{l-1}-u^{l-1}_h}.
  }
  So when $\vep\lnhminn\hMfg^2$ is sufficiently small, we have
  \eq{\label{eq:hatetahl contraction}
    \He{\hat \eta^l_h}\leq \frac12\He{u^{l-1}-\hat u^{l-1}_h}+\frac12\He{\hat\eta^{l-1}_h}\qaq \Lt{\hat \eta^l_h}\leq \frac12\Lt{u^{l-1}-\hat u^{l-1}_h}+\frac12\Lt{\hat\eta^{l-1}_h},
   }
    and then
  \eq{\label{eq:dpstab proof2}
    \He{\hat \eta^l_h}\leq \sum_{j=0}^{l-1}2^{j-l}\he{u^j-\hat u^j_h}\ls(kh_\T+k^{-\frac12}(kh_\T)^{\alpha}+k^3h_\T^2)\hMfg+2^{-l}\He{u^0-u^0_h}.
  }
  Combining \eqref{eq:dpstab proof1} and \eqref{eq:dpstab proof2}, we derive
  \eqn{
    \He{u^l-u^l_h}\ls (kh_\T+k^{-\frac12}(kh_\T)^{\alpha}+k^3h_\T^2)\hMfg+2^{-l}\He{u^0-u^0_h}.
  }
  Similarly,
   \eqn{
    k\Lt{u^l-u^l_h}\ls \big((kh_\T)^{2\alpha}+k^3h_\T^2\big)\hMfg+2^{-l}\Lt{u^0-u^0_h}.
  } 
  Then \eqref{eq:error estimate} comes from letting $l\rightarrow \infty$ in the above two estimates. This completes the proof of the theorem.
\end{proof}

\begin{remark}
  {\rm (i)}
  Denote by $\T_0$ the initial mesh. $\lnhminn/ \abs{\ln h_{\T_0}}\eqsim$ the maximum number of times the mesh is refined locally, that is to say, although the condition here is more stringent than the one in \cite{Jiang:Li:Wu:Zou:2022} ($\vep\hMfg^2\leq \theta_0$ for the impedance boundary condition),
  it is also acceptable. In addition, to the best of our knowledge, nothing is known about a $L^{\infty}$ interior estimate of $u_h$ for the NLH problem on a shape regular mesh, especially a locally refined mesh. 
  Moreover, another technique for analysing the case of a shape regular mesh is directly using the $H^{1}$ global estimate instead of the $L^{\infty}$ interior estimate and requires more stringent conditions (see, e.g. \cite{Maier:Verfurth:2022}).

  {\rm (ii)} Similar to Remark~\ref{rm:multiple solutions} {\rm (i)}, Theorem~\ref{thm:dpstab} states that if $\lnhminn\vep\hMfg^2\leq\theta_0$ and $k^3h_\T^{1+\alpha}\leq C_0$, FEM \eqref{eq:FEMva} attains a unique low-energy solution $u_h$ and $u_h$ approximates the low-energy solution $u$ to the NLH problem \eqref{eq:NLHva}. 
  Like the proof in \cite{Jiang:Li:Wu:Zou:2022}, we can prove the local quadratic convergence of Newton's iteration \eqref{eq:iteration dp2} which also holds for the high-energy solutions. The details are omitted.
\end{remark}

\section{{\it A posteriori} error estimates}\label{s:3}
In this section, we derive upper and lower bounds for the FEM and the elliptic projection and establish some equivalent relationships between them which is of significance for proving the convergence and quasi-optimality.

\subsection{Error estimators}\label{s:error estimators}
Denote by $\E$ the set of all edges of the elements in the triangulation $\T$ and $\E^{I}$ the set of all interior edges in $\T$.
Then, for any $e\in\E$, we define $\Omega_e:=\cup \{T\in \T:\; e\subset \partial T\}$ and for any function $v\in H^1(\Om)$, we let the jump of $\na v$ across an interior edge $e=T\cap T'\in\E^I$ be $\jump{\na v}:=\na v|_{T} \cdot n_{T}+\na v|_{T'} \cdot n_{T'}$,
where  $n_{T}$ and $n_{T'}$ denote the  unit outward normal on $\partial T$ and $\partial T'$. In addition, we define $\Sigma_{\T}(T):=\{T^{'}\in\T:T^{'}\cap T\in\E\}$ to be the set of elements having common edge with $T$ and let $\Omega_{\T}(T):=\bigcup_{T'\in\Sigma_\T(T)}T'$.

For any $v_h\in V_\T$, the element residuals, jump residuals, error estimators and oscillations for $v_h$ regarded as an approximation to the NLH problem \eqref{eq:Helm}--\eqref{eq:Dir} are defined as follows:
\eqn{
  \begin{aligned}
    R_T(v_h):&=\big(f+k^2(1+\vep\oneo\abs{v_h}^2)v_h\big)|_T,\ T\in\T,\\
    R_e(v_h):&=\begin{cases}
      -\frac12[\![\na v_h]\!], & e\subset\Om,\\
      -\frac{\partial v_h}{\partial n}-ikv_h+g, & e\subset\GaI,\\
      0, & e\subset\GaD
    \end{cases} \quad R_{\partial T}|_e:=R_e,\ \forall e\subset\partial T,\\
    \eta_T^2(v_h):&=\sum_{T'\in\Sigma_{\T}(T)}h_{T'}^2\LtD{R_{T'}(v_h)}{T'}^2+h_T\LtD{R_\pT(v_h)}{\partial T}^2,\\
    {\rm osc}_T^2(v_h):&=h_{T}^2\Ltd{R_{T}(v_h)-\wideBar{R_{T}(v_h)}}{T}^2+h_T\Ltd{R_\pT(v_h)- \wideBar{R_\pT(v_h)}}{\pT}^2\\
    &=h_{T}^2\Ltd{f-\wideBar{f}}{T}^2+h_T\LtD{g- \wideBar g}{\pT\cap \GaI}^2,
  \end{aligned}
}
where $\wideBar v|_T$ is the $L^2(T)$ projection of $v|_T$ onto $P_{3}(T)$, and $\wideBar v|_e$ is the $L^2(e)$ projection of $v|_e$ onto $P_{1}(e)$.

The elliptic projection $P_hu$ for the solution $u$ of \eqref{eq:NLHva} can be seen as the FE discretization of the elliptic problem:
\eqn{
  \left\{\begin{aligned}
    -\Delta \tilde u+\tilde u&=f+\big(k^2(1+\vep\oneo\abs{u}^2)+1\big)u \ \ \text{in} \ \Om, \\
    \frac{\partial \tilde u}{\partial n}&=-iku+g \ \ \text{on} \ \GaI;\quad \tilde u=0 \ \ \text{on} \ \GaD.
  \end{aligned}\right.
}
Thus, the corresponding error estimators for $v_h$ can be stated:
\eq{\label{eq:wtestimators}
  \begin{aligned}
    \wt R_T(v_h)&:=\big(f+\big(k^2(1+\vep\oneo\abs{u}^2)+1\big)u-v_h\big)|_T,\ T\in\T,\\
    \wt R_e(v_h)&:=\begin{cases}
      -\frac12[\![\na v_h]\!], & e\subset\Om,\\
      -\frac{\partial v_h}{\partial n}-iku+g, & e\subset\GaI,\\
      0, &e\subset\GaD
    \end{cases} \quad \widetilde R_{\partial T}|_e:=\widetilde R_e,\ \forall e\subset\partial T,\\
    \wt\eta_T^2(v_h)&:=\sum_{T'\in\Sigma_{\T}(T)}h_{T'}^2\Ltd{\widetilde{R}_{T'}(v_h)}{T'}^2+h_T\Ltd{\widetilde{R}_\pT(v_h)}{\partial T}^2,\\
    \wt\osc_T^2(v_h)&:=h_{T}^2\Ltd{\wt R_{T}(v_h)-\wideBar{\wt R_{T}(v_h)}}{T}^2+h_T\Ltd{\wt R_\pT(v_h)- \wideBar{\wt R_\pT(v_h)}}{\partial T}^2.
  \end{aligned}
}
For any submesh $\M\subset$ of $\T$, we denote $\wt u_h:=P_hu$, and
\begin{align*}
    &R_T:=R_T(u_h),\ R_e:=R_e(u_h),\ \eta_T:=\eta_T(u_h),\ \osc_T:=\osc_T(u_h),\\
    &\eta_\M:=\Big(\sum_{T\in\M}\eta_T^2\Big)^{\frac12},\ \osc_\M:=\Big(\sum_{T\in\M}\osc_T^2\Big)^{\frac12},
\end{align*}
Similarly, we omit the variables in the brackets in \eqref{eq:wtestimators} when $v_h=\wt u_h$ and also for $\wt \eta_\M$ and $\wt\osc_\M$.
\begin{remark}
  {\rm (i)} Our error estimator $\eta_T$ on a element $T$ is slightly different from the standard residual-type error estimator defined as $\eta_T^{\rm std}:=\big(\LtD{R_{T}}{T}^2+h_T\LtD{R_\pT}{\partial T}^2\big)^\frac12$, but, apparently,  their corresponding error estimators on the whole mesh are equivalent, that is, $\eta_\T\eqsim\eta_\T^{\rm std}$. We introduce this new error estimator for the purpose of proving the convergence and quasi-optimality of the standard adaptive algorithm.

  {\rm (ii)} The error estimators for the elliptic projection are introduced only for the theoretical analysis and will not be used in the adaptive algorithm.
\end{remark}

In the subsequent part of this paper, we consider the solutions $u$ and $u_h$ to \eqref{eq:NLHva} and \eqref{eq:FEMva} from Theorems~\ref{thm:cpstab} and \ref{thm:dpstab}. In other words, we specifically take the case of low energy into account, namely 
we make the assumption. 
\begin{assumption}
  Let $u$ and $u_h$ be the solutions to \eqref{eq:NLHva} and \eqref{eq:FEMva} from Theorem~\ref{thm:cpstab} and Theorem~\ref{thm:dpstab} respectively.
\end{assumption}

\subsection{Upper and lower bounds}

We give {\it a posteriori} upper and lower bounds for the errors of $u_h$ and $P_hu$ in this subsection.
\begin{theorem}[upper bounds]\label{thm:upbound}
 If $\lnhminn\vep\hMfg^2\leq\theta_0\ \mbox{and}\ k^3h_\T^{1+\alpha}\leq C_0$, then we have the estimates:
  \begin{align}
    \Lt{\nabla(u-u_h)}&\ls (1+k^2h_\T)\eta_\T, \label{thm:upbound1} \\
    \Lt{u-u_h}&\ls (kh_\T+k^{-\frac12}(kh_\T)^\alpha)\eta_\T, \label{thm:upboundL2}\\
    \LtD{u-u_h}{\GaI}&\ls (h_\T^{\frac{\alpha}2}+kh_\T+k^2h_\T^{1+\frac{\alpha}2})\eta_\T, \label{thm:upboundboundary}\\
    \Ho{u-P_h u}&\ls\eta_\T,\label{thm:upbound2}\\
    \Ho{u-P_h u}&\ls\wt\eta_{\T}.\label{thm:upbound3}
  \end{align}
\end{theorem}
\begin{proof}
  Denote by $e:=u-u_h$, $\rho:=u-P_hu$, and $\zeta_h:=P_hu-u_h$. We show the following identity which we will often use subsequently: for any $\phi\in \HoG$ and $v_h\in V_\T$,
  \eq{\label{eq:estimator indetity}
  a_{u}(e,\phi)=\sum_{T\in\T}\big((R_T,\phi-v_h)_T+\langle R_\pT,\phi-v_h\rangle_\pT\big)+k^2\vep\big((\abs{u}^2-\abs{u_h}^2)u_h,\phi\big)_{\Omz}.}
Indeed, with the aid of elementwise integration by parts and $a_u(u,v_h)-a_{u_h}(u_h,v_h)=0$ by \eqref{eq:NLHva}, \eqref{eq:FEMva}, we derive
\eqn{
  a_{u}(e,\phi)&=a_{u}(u,\phi-v_h)-a_{u_h}(u_h,\phi-v_h)+\big(a_{u_h}(u_h,\phi)-a_u(u_h,\phi)\big)\\
  &=\big(f+k^2(1+\vep\oneo\abs{u_h}^2)u_h,\phi-v_h\big)+\langle g,\phi\rangle_{\GaI}-\mbi k\langle u_h,\phi-v_h\rangle_{\GaI}\\
  & \quad\; -\sum_{T\in\T}\big\langle\frac{\partial u_h}{\partial n},\phi-v_h\big\rangle_\pT+k^2\vep\big((\abs{u}^2-\abs{u_h}^2)u_h,\phi\big)_{\Omz}\\
  &=\sum_{T\in\T}\big((R_T,\phi-v_h)_T+\langle R_\pT,\phi-v_h\rangle_\pT\big)+k^2\vep\big((\abs{u}^2-\abs{u_h}^2)u_h,\phi\big)_{\Omz}.
}

  Firstly we estimate $\zeta_h$. Taking $\phi=v_h=\zeta_h$ into \eqref{eq:estimator indetity}, and combining with \eqref{eq:a to b relationship} and \eqref{eq:GO}, we have
  \begin{align}
    a_u(\zeta_h,\zeta_h)&=a_u(e,\zeta_h)-a_u(\rho,\zeta_h)\notag\\
    &=k^2\vep\big((\abs{u}^2-\abs{u_h}^2)u_h+\abs{u}^2\rho,\zeta_h\big)_{\Omz}+(k^2+1)(\rho,\zeta_h)-\mbi k\langle\rho,\zeta_h\rangle_{\GaI}. \label{eq:zeta_h varform NL}
  \end{align}
 Taking the imaginary part of \eqref{eq:zeta_h varform NL} and noting $e = \rho + \zeta_h$, by $L^\infty$ interior estimates in \eqref{eq:cpstab} and \eqref{eq:dpstab} we have
 \eqn{
    \begin{aligned}
  k\norm{\zeta_h}^2_{0,\GaI}&\ls k\theta_0(\Lt{e}+\Lt{\rho})\Lt{\zeta_h}+k^2h_\T^\alpha\Ho{\rho}\Lt{\zeta_h}+k^{\frac12}h_\T^{\frac{\alpha}2}\Ho{\rho}k^{\frac12}\LtD{\zeta_h}{\GaI}\\
  &\ls k \Lt {\zeta_h}^2  +k^2h_\T^\alpha\Ho{\rho}\Lt{\zeta_h}+k^{\frac12}h_\T^{\frac{\alpha}2}\Ho{\rho}k^{\frac12}\LtD{\zeta_h}{\GaI}
    \end{aligned}
  }
 where \eqref{eq:elliptic projection estimate} and the trace theorem are also used. Then by Young's inequality, we proceed
 \eqn{
    k\LtD{\zeta_h}{\GaI}^2\ls (k+k^3h_\T^{\alpha})\Lt{\zeta_h}^2+kh_\T^\alpha\Ho{\rho}^2,
 }
 which means
 \eq{\label{eq:zeta_h L2Gamma NL}
  \LtD{\zeta_h}{\GaI}\ls (1+kh_\T^{\frac{\alpha}2})\Lt{\zeta_h}+h_\T^{\frac{\alpha}2}\Ho{\rho}.
 }
 Taking the real part of \eqref{eq:zeta_h varform NL}, we have the following inequalities by using \eqref{eq:elliptic projection estimate} and \eqref{eq:zeta_h L2Gamma NL}
 \eqn{
  \Lt{\nabla\zeta_h}^2&\ls (k^2+\theta_0 k)\Lt{\zeta_h}^2+k\theta_0 h_\T^\alpha\Ho{\rho}\Lt{\zeta_h}+k^2h_\T^\alpha\Ho{\rho}\Lt{\zeta_h}+kh^{\frac{\alpha}2}_\T\Ho{\rho}\LtD{\zeta_h}{\GaI}\\
  &\ls k^2\Lt{\zeta_h}^2+(kh^{\frac{\alpha}2}_\T+k^2h_\T^\alpha)\Ho{\rho}\Lt{\zeta_h}+kh_\T^\alpha\Ho{\rho}^2\\
  &\ls k^2(1+kh_\T^\alpha)\Lt{\zeta_h}^2+kh_\T^\alpha\Ho{\rho}^2.
 }
Consequently, by $kh_\T^\alpha\ls 1$ (see \eqref{tkh}), we derive
\eq{\label{eq:zeta_h H1 NL}
\Lt{\nabla \zeta_h}\ls k\Lt{\zeta_h}+k^{\frac12}h_\T^{\frac{\alpha}2}\Ho{\rho}.
}

  Next, for the sake of estimating $\Lt{\zeta_h}$, we consider the duality problem: find $z\in \HoG$ such that
  \eq{\label{eq:duality argument}
    a_{u}(v,z)=(v,\zeta_h)\quad \forall v\in \HoG.
  }
  By the similar proof in Lemma \ref{lem:auxiliary stability}, we can derive the stability estimates and singularity decomposition for the above problem. Then with the help of \eqref{eq:singular interpolation error}, we have
  \begin{equation}\label{eq:duality argument estimate}
    k\Lt{z}\ls \Lt{\zeta_h},\quad \Ho{z-I_hz}\ls (kh_\T+k^{-\frac12}(kh_\T)^\alpha)\Lt{\zeta_h}.
  \end{equation}
  Taking $v=\zeta_h$ in \eqref{eq:duality argument} and arguing as in \eqref{eq:zeta_h varform NL}, we find
  \eqn{
    \norm{\zeta_h}_0^2&=a_{u}(\zeta_h,z)=a_{u}(\zeta_h,z-P_hz)+a_u(\zeta_h,P_hz)\\
    &=-(k^2+1)(\zeta_h,z-P_hz)-k^2\vep(\abs{u}^2\zeta_h,z-P_hz)_{\Omz}+\mbi k\langle \zeta_h,z-P_hz\rangle_{\GaI}+a_u(e,P_hz)-a_u(\rho,P_hz)\\
    &=-(k^2+1)(\zeta_h,z-P_hz)-k^2\vep(\abs{u}^2\zeta_h,z-P_hz)_{\Omz}+\mbi k\langle \zeta_h,z-P_hz\rangle_{\GaI}\\
    &\quad+k^2\vep\big((\abs{u}^2-\abs{u_h}^2)u_h,P_hz-z+z\big)+a_u(\rho,z-P_hz)-(\rho,\zeta_h)\\
    &\ls (k^2+k\theta_0) \Lt{\zeta_h}\Lt{z-P_hz}+k\LtD{\zeta_h}{\GaI}\LtD{z-P_hz}{\GaI}\\
    &\quad+k\theta_0(\Lt{\zeta_h}+\Lt{\rho})(\Lt{z-P_hz}+k^{-1}\Lt{\zeta_h})+\He{\rho}\He{z-P_hz}+\Lt{\rho}\Lt{\zeta_h}.
  }
  From \eqref{eq:elliptic projection estimate}, \eqref{eq:zeta_h L2Gamma NL}, \eqref{eq:duality argument estimate}, $(kh_\T)^\alpha\ls 1$, and letting $\theta_0$ sufficiently small, we deduce that
  \eqn{
    \norm{\zeta_h}_0^2&\ls \Big(k^2h_\T^\alpha \Lt{\zeta_h}+kh_\T^{\frac{\alpha}2}\big((1+kh_\T^{\frac{\alpha}2})\Lt{\zeta_h}+h_\T^{\frac{\alpha}2}\Ho{\rho}\big)+\He{\rho}\Big)\Ho{z-P_hz}+h^\alpha_\T\Ho{\rho}\Lt{\zeta_h}\\
     &\ls(kh_\T^{\frac{\alpha}2}+k^2h_\T^{\alpha})\big(kh_\T+k^{-\frac12}(kh_\T)^\alpha\big)\Lt{\zeta_h}^2+\big(kh_\T+k^{-\frac12}(kh_\T)^\alpha+h_\T^\alpha\big)\He{\rho}\Lt{\zeta_h}\\
     &\ls \big(k^2h_\T^{1+\frac{\alpha}2}+k^3h_\T^{1+\alpha}+(kh_\T)^\alpha k^{\frac12}h^{\frac{\alpha}2}+k^{\frac32+\alpha}h_\T^{2\alpha}\big)\Lt{\zeta_h}^2+\big(kh_\T+k^{-\frac12}(kh_\T)^\alpha\big)\He{\rho}\Lt{\zeta_h}\\
     &\ls t(k,h)\Lt{\zeta_h}^2+\big(kh_\T+k^{-\frac12}(kh_\T)^\alpha\big)\He{\rho}\Lt{\zeta_h}.
  }
  Therefore, if $k^3h_\T^{1+\alpha}$ is sufficiently small, we have
  \eq{\label{eq:zeta_h L2 to rho NL}
    \Lt{\zeta_h}\ls \big(kh_\T+k^{-\frac12}(kh_\T)^\alpha\big)\He{\rho}.
  }
  Altogether with \eqref{eq:zeta_h H1 NL} and \eqref{eq:zeta_h L2Gamma NL}, we get
\begin{align}
   \Lt{\nabla \zeta_h}&\ls \big(k^{\frac12}h_\T^{\frac{\alpha}2}+k^{\frac12}(kh_\T)^{\alpha}+k^2h_\T\big)\He{\rho},\label{eq:zeta_h H1 to rho NL}\\
   \LtD{\zeta_h}{\GaI}&\ls \big(h^{\frac{\alpha}2}_\T+kh_\T+k^{-\frac12}(kh_\T)^\alpha+k^2h_\T^{1+\frac{\alpha}2}+k^{\frac12+\alpha}h_\T^{\frac{3\alpha}2}\big)\He{\rho}.\label{eq:zeta_h L2Gamma to rho NL}
\end{align}

  Secondly, we prove \eqref{thm:upbound2}. Using \eqref{eq:GO}--\eqref{eq:a to b relationship}, \eqref{eq:estimator indetity} with $\phi=\rho$ and $v_h$ to be chosen as its Scott-Zhang interpolant \cite{Scott:Zhang:1990}, \eqref{eq:zeta_h L2 to rho NL}, and \eqref{eq:zeta_h L2Gamma to rho NL}, we derive that
\begin{align*}
  \Ho{\rho}^2&= b(\rho,\rho)=b(e,\rho)=a_{u}(e,\rho)+\big((k^2+1+k^2\vep\oneo\abs{u}^2)e,\rho\big)-\mbi k\langle e,\rho\rangle_{\GaI}\\
  &\ls\eta_\T\Lt{\nabla\rho}+k^2\Lt{e}\Lt{\rho}+k\LtD{e}{\GaI}\LtD{\rho}{\GaI}\\
  &\ls \big(\eta_\T+k^2h_\T^\alpha(\Lt{\zeta_h}+\Lt{\rho})+kh^{\frac{\alpha}2}_\T(\LtD{\zeta_h}{\GaI}+\LtD{\rho}{\GaI})\big)\Ho{\rho}\\
  &\ls \eta_\T\Ho{\rho}+\big(k^3h_\T^{1+\alpha}+k^{\frac32+\alpha}h_\T^{2\alpha}+k^2h_\T^{1+\frac{\alpha}2}+k^2h_\T^{2\alpha}+kh_\T^\alpha\big)\Ho{\rho}^2\\
  &\ls \eta_\T\Ho{\rho}+t(k,h)\Ho{\rho}^2,
\end{align*}
which implies $\Ho{\rho}\ls \eta_\T$ when $k^3h_\T^{1+\alpha}$ is sufficiently small, that is, \eqref{thm:upbound2} holds. Meanwhile, $\He{\rho}\ls(1+kh_\T^{\alpha})\eta_\T\ls\eta_\T$
since $kh_\T^\alpha\ls 1$.

Thirdly, \eqref{thm:upboundL2}--\eqref{thm:upboundboundary} comes from \eqref{eq:zeta_h L2 to rho NL}--\eqref{eq:zeta_h L2Gamma to rho NL}, the triangle inequality and $k^3h_\T^{1+\alpha}\leq C_0$. Indeed, from \eqref{tkh} and $\frac12<\alpha\le 1$,
\eqn{
  \LtD{e}{\GaI}&\leq \LtD{\rho}{\GaI}+\LtD{\zeta_h}{\GaI}\\
  &\ls \big(h_\T^{\frac{\alpha}2}(1+k^{\alpha-\frac12}h_\T^{\frac\alpha2}+k^{\frac12+\alpha}h^\alpha)+kh_\T+k^2h_\T^{1+\frac{\alpha}2}\big)\He{\rho}\\
  &\ls \big(h_\T^{\frac{\alpha}2}(1+t(k,h))+kh_\T+k^2h^{1+\frac{\alpha}2}_\T\big)\He{\rho}\\
  &\ls (h_\T^{\frac{\alpha}2}+kh_\T+k^2h_\T^{1+\frac{\alpha}2})\eta_\T.
}

Finally, the proof of \eqref{thm:upbound3} is standard (see, e.g. \cite{Morin:Nochetto:Siebert:2002}) and is omitted here. 
This proof is completed.
\end{proof}

\begin{remark}
  In this proof, we first estimate $P_hu-u_h$ instead of directly estimating $u-u_h$. The upper bound for the error on the boundary \eqref{thm:upboundboundary} is more precise than (3.10) in \cite{Duan:Wu:2023} which is of the linear case.

\end{remark}


\begin{theorem}[lower bounds]
  There exist two positive constants $C_{\rm lb}$ and $\wt C_{\rm lb}$ such that if $\lnhminn\vep\hMfg^2\leq\theta_0$ and $k^3h_\T^{1+\alpha}\leq C_0$, then the following estimates hold:
    \begin{align}
      &\eta_\T^2\leq C_{\rm lb}\big(\Ho{u-\phu}^2+{\rm osc}^2_\T\big),\label{thm:lowbound1} \\
      &\widetilde\eta_\T^2\leq\wt C_{\rm lb}\big(\Ho{u-\phu}^2+{\rm osc}^2_\T\big).\label{thm:lowbound2}
    \end{align}
  Meanwhile, there holds the following local lower bound:
 \eq{\label{eq:local lower bound}
 \eta_T\ls \He{u-u_h}_{\Omega_{\T}(T)}+\osc_{\Sigma_{\T}(T)}.
}
\end{theorem}
\begin{proof}
There exists a nonnegative bubble function $\varphi_T$ (see, e.g., \cite{Ainsworth:Tinsley:2000}) whose support is $T$ satisfying
 \eq{\label{eq:element bf}
   \LtD{{\wideBar R}_T}{T}^2\eqsim\int_T \varphi_T\abs{{\wideBar R}_T}^2 \qaq \LtD{\varphi_T{\wideBar R}_T}{T}+h_T\sHoD{\varphi_T{\wideBar R}_T}{T}\ls\LtD{{\wideBar R}_T}{T}.
 }
 Taking $\phi=\varphi_T{\wideBar R}_T, v_h=0$ in \eqref{eq:estimator indetity} and using the properties \eqref{eq:element bf}, we have
 \eqn{
   \begin{aligned}
    \norm{{\wideBar R}_T}^2_{0,T}&\ls({\wideBar R}_T,\phi)_T=a_{u}(u-u_h,\phi)-k^2\vep\big((\abs{u}^2-\abs{u_h}^2)u_h,\phi\big)_{\Omz}+({\wideBar R}_T-R_T,\phi)_T\\
    &=\big(\nabla(u-\phu),\nabla\phi\big)-k^2\big((1+\vep\oneo\vert u\vert^2)(u-u_h),\phi\big)\\
    &\quad -k^2\vep\big((\abs{u}^2-\abs{u_h}^2)u_h,\phi\big)_{\Omz}+({\wideBar R}_T-R_T,\phi)_T\\
    &\ls \norm{u-\phu}_{1,T}h_T^{-1}\norm{{\wideBar R}_T}_{0,T}+k^2\LtD{u-u_h}{T}\norm{{\wideBar R}_T}_{0,T}\\
    &\quad +\lnhminn k\vep\hMfg^2\norm{u-u_h}_{0,T}\norm{{\wideBar R}_T}_{0,T}+ \norm{{\wideBar R}_T-R_T}_{0,T}\norm{{\wideBar R}_T}_{0,T},
   \end{aligned}
 }
 where we have used $(\nabla (u_h-P_hu),\nabla \phi)=0$ by integration parts in the second equal-sign and then we get
 \eqn{
   h_T\norm{{\wideBar R}_T}_{0,T}\ls\norm{u-\phu}_{1,T}+k^2h_T\norm{u-u_h}_{0,T}+h_T\norm{{\wideBar R}_T-R_T}_{0,T}.}
It follows from \eqref{thm:upboundL2} and the triangle inequality that
\eq{\label{eq:sum R_T estimate}
\begin{aligned}
\Big(\sum_{T\in\T}h^2_{T}\norm{R_{T}}^2_{0,T}\Big)^{\frac12}\ls& \norm{u-\phu}_1+k^2h_\T(kh_\T+k^{-\frac12}(kh_\T)^\alpha)\eta_\T+\Big(\sum_{T\in\T}h^2_{T}\norm{{\wideBar R}_{T}-R_{T}}^2_{0,T}\Big)^{\frac12}.
\end{aligned}}
Furthermore, a nonnegative edge bubble function $\chi_e$ whose support is $\Omega_e$ satisfies
 \eq{\label{eq:edge bf}
  \int_e\chi_e\vert{\wideBar R}_e\vert^2\eqsim\Vert{\wideBar R}_e\Vert_{0,e}^2\qaq h_T^{-\frac12}\Vert\chi_e{\wideBar R}_e\Vert_{0,T}+h_T^{\frac12}\vert\chi_e{\wideBar R}_e\vert_{1,T}\ls\Vert{\wideBar R}_e\Vert_{0,e}\quad \forall T\subset\Omega_e.
 }
Denoting by $\phi_1=\sum\limits_{e\in\E}h_e\chi_e{\wideBar R}_e$, then taking $\phi=\phi_1-P_h\phi_1, v_h=-P_h\phi_1$ into \eqref{eq:estimator indetity} and using \eqref{eq:edge bf}, we have
\eqn{
   \begin{aligned}
     \sum_{T\in\T}&\sum_{e\in\partial T}\langle R_e,\phi_1\rangle_e=a_{u}(u-u_h,\phi_1-P_h\phi_1)-\sum_{T\in\T}\Big(R_T,\sum_{e\in\pT}h_e\chi_e{\wideBar R}_e\Big)_{T}-k^2\vep\big((\abs{u}^2-\abs{u_h}^2)u_h,\phi_1-P_h\phi_1\big)_{\Omz}\\
     &= b(u-P_hu,\phi_1-P_h\phi_1)-\big((k^2+1+k^2\varepsilon\oneo\vert u\vert^2)(u-u_h),\phi_1-P_h\phi_1\big)\\
     & \quad  +\mbi k\langle u-u_h,\phi_1-P_h\phi_1\rangle_{\GaI}-\sum_{T\in\T}\Big(R_T,\sum_{e\in\partial T}h_e\chi_e{\wideBar R}_e\Big)_T-k^2\vep\big((\vert u\vert^2-\vert u_h\vert^2)u_h,\phi_1-P_h\phi_1\big)_{\Omz} \\
     &\ls |b(u-P_hu,\phi_1)|+k^2\norm{u-u_h}_0h^\alpha_\T\norm{\phi_1}_1+k\Vert u-u_h\Vert_{0,\GaI}h^{\frac\alpha2}_\T\Vert\phi_1\Vert_1  \\
     &\quad   +\Big(\sum_{T\in\T}h_T\Vert R_T\Vert^2_{0,T}\Big)^{\frac12}\Big(\sum_{T\in\T}h_T\Vert {\wideBar R}_{\partial T}\Vert^2_{0,\partial T}\Big)^{\frac12}+k\Lt{u-u_h}h^\alpha_\T\Vert\phi_1\Vert_1\\
     &\ls \Big(\Vert u-P_hu\Vert_1+k^2h^\alpha_\T\Lt{u-u_h}+kh^{\frac\alpha2}_\T\Vert u-u_h\Vert_{0,\GaI}+\Big(\sum_{T\in\T}h_T\Vert R_T\Vert^2_{0,T}\Big)^{\frac12}\Big)\Big(\sum_{T\in\T}h_T\Vert {\wideBar R}_\pT\Vert^2_{0,\pT}\Big)^{\frac12},
    \end{aligned}
 }
 where we have used $\Vert\phi_1\Vert^2_1\ls \sum_{e\in \E}h_e^2\HoD{\chi_e {\wideBar R}_e}{\Omega_e}^2 \ls \sum_{T\in\T}h_T\Vert {\wideBar R}_\pT\Vert^2_{0,\pT}$ in the last inequality. Then from \eqref{thm:upboundL2} and \eqref{thm:upboundboundary} we can deduce that
 \eqn{\begin{aligned}
  \sum\limits_{T\in\T}h_T\norm{{\wideBar R}_{\partial T}}^2_{0,\partial T}&\eqsim \sum\limits_{T\in\T}\sum\limits_{e\in\partial T}\langle {\wideBar R}_e,\phi_1\rangle_e=\sum\limits_{T\in\T}\sum\limits_{e\in\partial T}\big(\langle R_e,\phi_1\rangle_e+\langle {\wideBar R}_e-R_e,\phi_1\rangle_e\big)\\
  &\ls \Big(\nnorm{u-\phu}_1+t(k,h_\T)\eta_\T+\Big(\sum_{T\in\T}h_T\Vert R_T\Vert^2_{0,T}\Big)^{\frac12}\Big)\Big(\sum_{T\in\T}h_T\Vert {\wideBar R}_\pT\Vert^2_{0,\pT}\Big)^{\frac12}\\
  &\quad +\Big(\sum_{T\in\T}h_T\Vert {\wideBar R}_\pT-{\wideBar R}_\pT\Vert^2_{0,\pT}\Big)^{\frac12}\Big(\sum_{T\in\T}h_T\Vert {\wideBar R}_\pT\Vert^2_{0,\pT}\Big)^{\frac12}.
 \end{aligned}
 }
By taking the estimate \eqref{eq:sum R_T estimate} into the above inequation, we can derive that
\eq{\label{eq:sum R_pT estimate}\sum\limits_{T\in\T}h_T\Vert R_\pT\Vert^2_{0,\pT}\ls \Vert u-P_hu\Vert_1^2+t^2(k,h_\T)\eta^2_\T+\osc^2_\T.}
Combining \eqref{eq:sum R_T estimate} and \eqref{eq:sum R_pT estimate}, we obtain $\eta_\T^2\ls\Vert u-P_hu\Vert^2_1+t^2(k,h_\T)\eta_\T^2+\osc^2_\T$, which implies \eqref{thm:lowbound1} if $k^3h^{1+\alpha}_\T$ is sufficiently small.

The local lower bound \eqref{eq:local lower bound} can be derived similarly by taking $\phi=\varphi_T\wideBar{R_T}, v_h=0$ and $\phi=\sum_{e\in\pa T}h_e\chi_e\wideBar{R_e},$ $v_h=0$ in \eqref{eq:estimator indetity}, respectively.

On the other hand, from the standard estimate of the {\it a posteriori} lower bound (see, e.g., \cite{Morin:Nochetto:Siebert:2002}), we can obtain
\eq{\label{lowbound elliptic projection}
  \widetilde{\eta}_\T^2\ls\Vert u-P_hu\Vert^2_1+\widetilde{\osc}^2_\T.}
To prove \eqref{thm:lowbound2}, it remains to estimate the difference between $\osc_T$ and $\wt\osc_T$. Similar to \eqref{eq:Linf interior estimate term2} and \eqref{eq:Linf interior estimate term3a}, we have $\Linf{P_hu}\ls \lnhmin k^{-\frac12} \hMfg$ if $k^3h_\T^2\ls 1$. Therefore
\eqn{
  \begin{aligned}
    \vert\wt{\osc}_T&-\osc_T\vert^2\leq h^2_{T}\big\Vert\wt R_{T}-R_{T}-\wideBar{\wt{R}_{T}-R_{T}}\big\Vert^2_{0,T}+h_T\big\Vert\widetilde{R}_\pT-R_\pT-\wideBar{\widetilde{R}_\pT-R_\pT}\big\Vert^2_{0,\pT}\\
    &= h^2_{T}\big\Vert \big(k^2(1+\vep\oneo\abs{u}^2)+1\big)u-\wideBar{\big(k^2(1+\vep\oneo\abs{u}^2)+1\big)u}\big\Vert_{0,T}+h_T\Vert \mbi k(u-\wideBar{u})\Vert_{0,\pa T\cap\GaI}\\
    &\leq h^2_{T}\big\Vert(k^2+1)(u-\phu)+k^2\vep\oneo\big((\vert u\vert^2-\vert P_hu\vert^2)u+\vert P_hu\vert^2(u-P_hu)\big)\big\Vert^2_{0,T}+h_T\Vert\mbi k(u-\phu)\Vert^2_{0,\pT\cap \GaI}\\
    &\ls\big(k^4h^2_T\Vert u-\phu\Vert_{0,T}^2+k^2h_T\Vert u-\phu\Vert^2_{0,\pT\cap \GaI}\big).
  \end{aligned}
}
From the definitions of $\osc_\T$ and $\wt\osc_\T$
, \eqref{thm:upbound2}, $k^3h_\T^{1+\alpha}\leq C_0$, and trace inequality, we have
\eq{\label{eq:oscosc2}
  \begin{aligned}
    \vert\widetilde{\osc}_\T-\osc_\T\vert^2&\leq\sum\limits_{T\in\T}\vert\widetilde{\osc}_T-\osc_T\vert^2\ls k^4h_\T^2\Lt{u-P_hu}^2+k^2h_\T\LtD{u-P_hu}{\GaI}^2\\
  &\ls (k^4h_\T^2 h_\T^{2\alpha}+k^2h_\T h_\T^{\alpha})\eta_\T^2\ls \eta_\T^2.
  \end{aligned}
}
Then together with \eqref{lowbound elliptic projection} and \eqref{thm:lowbound1}, \eqref{thm:lowbound2} is derived.
\end{proof}
\begin{remark}
  {\rm (i)}
  Combining \eqref{thm:upbound2} with \eqref{thm:lowbound1}, we get $\eta_\T\eqsim\Ho{u-P_hu}+\osc_\T$. That is to say, the error estimator of $u_h$ fits the error of $P_hu$ well and underestimates seriously the error of $u_h$ in the preasymptotic regime.
  This behaviour is first observed by Babu\v{s}ka, {\it et al.} \cite{Babuska:Ihlenburg:Strouboulis:Gangaraj:1997} for a one-dimensional linear problem. We prove this observation the first time for the NLH problem. 

  {\rm (ii)}
  From \eqref{eq:oscosc2}, there exists a constant $C_{\rm osc}$ such that for any $\M\subset\T$, the following estimate holds
  \eq{\label{eq:wtosc osc distance}
    \vert\osc_\M-\wt\osc_\M\vert^2\leq C_{\rm osc} (k^4h_\T^2 h_\T^{2\alpha}+k^2h_\T h_\T^{\alpha})\eta_\T^2.
  }
  This estimate will be used twice subsequently.
\end{remark}

\subsection{Equivalent relationships between two error estimators}\label{subs:equivalent relationship}
In this subsection, we establish the equivalence between $\eta_{\M}$ and $\widetilde{\eta}_{\M}$ in the following sense.
\begin{lemma}\label{lem:equivalent relationships}
  There exist $C_{\rm eq}$ and $\wt C_{\rm eq}$ satisfying $C_{\rm eq}\wt C_{\rm eq}\geq 1$ such that if $\lnhminn\vep\hMfg^2\leq\theta_0$ and $\ k^3h_\T^{1+\alpha}\leq C_0$, then the following estimates hold for any $\M\subset\T$:
  \eq{
      \eta_\M&\leq C_{\rm eq}\big(\widetilde{\eta}_{\M}+t(k,h_\T)\eta_{\T}\big),\label{etaM1}\\
      \widetilde{\eta}_\M&\leq \wt C_{\rm eq}\big(\eta_{\M}+t(k,h_\T)\eta_\T\big).\label{etaM2}
  }
\end{lemma}
\begin{proof}
  Similar to \eqref{eq:estimator indetity}, for any $\phi\in \HoG$ we have
  \eqn{b(u-\phu,\phi)=\sum\limits_{T\in\T}\big((\widetilde{R}_T,\phi)_T+\langle\widetilde{R}_\pT,\phi\rangle_\pT\big).}
  On the other hand, it follows from the definitions of $a_u(\cdot,\cdot)$, $b(\cdot,\cdot)$, and \eqref{eq:GO}:
  \eqn{
    &b(u-P_hu,\phi)=b(u-u_h,\phi-P_h\phi)\\
    & =a_{u}(u-u_h,\phi-P_h\phi)+\big((k^2+1+k^2\vep\oneo\abs{u}^2)(u-u_h),\phi-P_h\phi\big)-\mbi k\langle u-u_h,\phi-P_h\phi\rangle_{\GaI}.
  }
  By combining \eqref{eq:estimator indetity} with the above two identities, we obtain
  \eq{\label{eq:RT RpT tRT tRpT}
  \begin{aligned}
  \sum_{T\in\T}\big((R_T,\phi)_T+\langle R_\pT,\phi\rangle_\pT\big)&=\sum\limits_{T\in\T}\big((\widetilde{R}_T,\phi)_T+\langle\widetilde{R}_\pT,\phi\rangle_\pT\big)-k^2\vep\big((\vert u\vert^2-\vert u_h\vert^2)u_h,\phi-P_h\phi\big)_{\Omz}\\
  &\quad -\big((k^2+1+k^2\varepsilon\oneo\vert u\vert^2)(u-u_h),\phi-P_h\phi\big)+\mbi k\langle u-u_h,\phi- P_h\phi\rangle_{\GaI}.
  \end{aligned}
  }
 Denote by $\wt\M:=\bigcup_{T\in\M}\Sigma_\T(T)$. Taking $\phi=\sum\limits_{T\in\wt\M}h^2_{T}\varphi_{T}{\wideBar R}_{T}$ in \eqref{eq:RT RpT tRT tRpT}, we get 
  \eqn{
    \begin{aligned}
      \sum_{T\in\wt\M}(R_{T},\phi)_{T}&=\sum_{T\in\wt\M}(\wt R_{T},\phi)_{T}-k^2\vep\big((\vert u\vert^2-\vert u_h\vert^2)u_h,\phi-P_h\phi\big)_{\Omz}\\
      &\quad -\big((k^2+1+k^2\varepsilon\oneo\vert u\vert^2)(u-u_h),\phi-P_h\phi\big)+\mbi k\langle u-u_h,\phi-P_h\phi\rangle_{\GaI}.
    \end{aligned}
  }
  Then it follows from \eqref{eq:element bf} and \eqref{eq:elliptic projection estimate} that
  \eqn{
    \begin{aligned}
     \sum_{T\in\wt\M}h^2_{T}&\Vert{\wideBar{R}}_{T}\Vert^2_{0,T}\ls \sum_{T\in\wt\M}h^2_{T}({\wideBar R}_{T},\varphi_{T}{\wideBar R}_{T})_{T}=\sum_{T\in\wt\M}h^2_{T}(R_{T},\varphi_{T}{\wideBar R}_{T})_{T}+\sum_{T\in\wt\M}h^2_{T}({\wideBar R}_{T}-R_{T},\varphi_{T}{\wideBar R}_{T})_{T}\\
      &=\sum_{T\in\wt\M}h^2_{T}(\widetilde{R}_{T},\varphi_{T}{\wideBar R}_{T})_{T}+\sum_{T\in\wt\M}h^2_{T}({\wideBar R}_{T}-R_{T},\varphi_{T}{\wideBar R}_{T})_{T}+\mbi k\langle u-u_h,\phi-P_h\phi\rangle_{\GaI}\\
      &\quad -k^2\vep\big((\vert u\vert^2-\vert u_h\vert^2)u_h,\phi-P_h\phi\big)_{\Omz}-\big((k^2+1+k^2\varepsilon\oneo\vert u\vert^2)(u-u_h),\phi-P_h\phi\big)\\
      &\ls \sum_{T\in\wt\M}h^2_{T}\Vert\widetilde{R}_{T}\Vert_{0,T}\Vert{\wideBar R}_{T}\Vert_{0,T}+\sum_{T\in\wt\M}h^2_{T}\Vert R_{T}-{\wideBar R}_{T}\Vert_{0,T}\Vert{\wideBar R}_{T}\Vert_{0,T}\\
      &\quad +k\Vert u-u_h\Vert_{0,\GaI}h^{\frac\alpha 2}_\T\Vert\phi\Vert_1+k\Lt{u-u_h}h_{\T}^{\alpha}\Vert\phi\Vert_1+k^2\Vert u-u_h\Vert_0h_\T^{\alpha}\Vert\phi\Vert_1\\
      &\ls \bigg(\Big(\sum_{T\in\wt\M}h^2_{T}\Vert\wt{R}_{T}\Vert^2_{0,T}\Big)^{\frac12}+t(k,h_\T)\eta_\T+\Big(\sum_{T\in\wt\M}h^2_{T}\Vert R_{T}-{\wideBar R}_{T}\Vert_{0,T}^2\Big)^{\frac12}\bigg)\Big(\sum_{T\in\wt\M}h^2_{T}\Vert{\wideBar{R}}_{T}\Vert^2_{0,T}\Big)^{\frac12},
    \end{aligned}
  }
  where we have used $\Vert\phi\Vert^2_1\ls\sum_{T\in\wt\M}h^4_{T}\Vert\varphi_T{\wideBar R}_T\Vert^2_{1,T}\ls\sum_{T\in\wt\M}h^2_{T}\Vert{\wideBar R}_T\Vert^2_{0,T}$, the upper bounds \eqref{thm:upboundL2} and \eqref{thm:upboundboundary}, and \eqref{eq:t} to derive the last inequality. Then we proceed
  \eq{\label{eq:RT tRT}
    \sum_{T\in\wt\M}h^2_{T}\Vert R_{T}\Vert^2_{0,T}\ls\sum_{T\in\wt\M}h^2_{T}\Vert\widetilde{R}_{T}\Vert^2_{0,T}+t^2(k,h_\T)\eta^2_\T+\sum_{T\in\wt\M}h^2_{T}\Vert R_{T}-{\wideBar R}_{T}\Vert_{0,T}^2.}
    Denote $\E_\M$ the set of all edges in $\M$ and let $\phi=\sum_{e\in \E_\M}h_e\chi_e {\wideBar R}_e$ in \eqref{eq:RT RpT tRT tRpT}. From \eqref{eq:edge bf} and \eqref{eq:elliptic projection estimate}, we have
  \eqn{
  \begin{aligned}
  \bigg|\sum_{T\in\wt \M}\sum_{\substack{e\subset\pa T\\e\in\E_\M}}\langle R_e,\phi\rangle_e\bigg|&= \bigg|-\sum_{T\in\wt \M}(R_T,\phi)_T + \sum\limits_{T\in\wt\M}\big((\widetilde{R}_T,\phi)_T+\langle\widetilde{R}_\pT,\phi\rangle_\pT\big)+\mbi k\langle u-u_h,\phi-P_h\phi\rangle_{\GaI}\\
  &\quad -k^2\vep\big((\vert u\vert^2-\vert u_h\vert^2)u_h,\phi-P_h\phi\big)_{\Omz}-\big((k^2+1+k^2\varepsilon\oneo\vert u\vert^2)(u-u_h),\phi-P_h\phi\big)\bigg|\\
  &\ls \Big(\sum_{T\in\wt\M}h^2_{T}\Vert R_{T}\Vert^2_{0,T}+\sum_{T\in\wt\M}h^2_{T}\Vert\wt R_{T}\Vert^2_{0,T}+\sum_{T\in\M}h_T\Vert\wt R_\pT\Vert^2_{0,\pT}\Big)^{\frac12}\Big(\sum_{T\in\M}h_T\Vert{\wideBar{R}}_\pT\Vert^2_{0,\pT}\Big)^{\frac12}\\
  &\quad +\big(kh^{\frac\alpha 2}_\T\nnorm{u-u_h}_{0,\GaI}+k^2h^\alpha_\T\nnorm{u-u_h}_0\big)\nnorm{\phi}_1.
  \end{aligned}
  }
  We can easily find that
 \[\Ho{\phi}^2\ls \sum_{e\in \E_\M}h_e^2\norm{\chi_e {\wideBar R} _e}_{1,\Omega_e}^2\ls \sum_{T\in\M} h_T \norm{ {\wideBar R}_{\pa T}}_{0,\partial T}^2 ,\]
  then we deduce from \eqref{eq:RT tRT} that
 \begin{align*}
  \sum_{T\in\M} h_T \norm{{\wideBar R}_{\pa T}}_{0,\partial T}^2&\eqsim\sum_{T\in\wt \M}\sum_{\substack{e\subset\pa T\\e\in\E_\M}}\langle{\wideBar R}_e, h_e\chi_e {\wideBar R}_e\rangle_e =\sum_{T\in\wt\M}\sum_{\substack{e\subset\pa T\\e\in\E_\M}}\langle R_e, \phi\rangle_e + \sum_{T\in\wt\M}\sum_{\substack{e\subset\pa T\\e\in\E_\M}}\langle{\wideBar R}_e-R_e, \phi\rangle_e \\
  &\ls\Big(\sum_{T\in\wt\M}h^2_{T}\big(\nnorm{\wt R_{T}}_{0,T}^2+\nnorm{R_{T}-{\wideBar R}_{T}}^2_{0,T}\big) + \sum_{T\in\M}  h_T\nnorm{\wt R_{\pa T}}_{0,\partial T}^2 \\
  &+t^2(k,h_\T)\eta^2_\T  + \sum_{T\in\M} h_T\norm{R_{\pa T}-{\wideBar R}_{\pa T}}_{0,\partial T}^2\Big)^{\frac12}\big(\sum_{T\in\M} h_T \norm{{\wideBar R}_{\pa T}}_{0,\partial T}^2\big)^{\frac12}.
 \end{align*}
 As a consequence, we have the following estimate with the aid of the triangle inequality
 \eqn{
  \sum_{T\in\M} h_T \norm{R_{\pa T}}_{0,\partial T}^2&\ls \sum_{T\in\wt\M}h^2_{T}\big(\nnorm{\wt R_{T}}_{0,T}^2+\nnorm{R_{T}-{\wideBar R}_{T}}^2_{0,T}\big)+ \sum_{T\in\M}h_T\nnorm{\wt R_{\pa T}}_{0,\partial T}^2\\
 &\quad +t^2(k,h_\T)\eta^2_\T  + \sum_{T\in\M} h_T\norm{R_{\pa T}-{\wideBar R}_{\pa T}}_{0,\partial T}^2.
 }
 Altogether with \eqref{eq:RT tRT} and \eqref{eq:wtosc osc distance}, it reads
  \eqn{
  \begin{aligned}
    \eta_\M^2&\eqsim \sum_{T\in\wt\M}h^2_{T}\Vert R_{T}\Vert^2_{0,T}+\sum_{T\in\M} h_T \norm{R_{\pa T}}_{0,\partial T}^2\\
    &\ls\widetilde{\eta}_{\M}^2+t^2(k,h_\T)\eta_\T^2+ \sum_{T\in\wt\M}h^2_{T}\nnorm{R_{T}-{\wideBar R}_{T}}^2_{0,T}+\sum_{T\in\M} h_T\norm{R_{\pa T}-{\wideBar R}_{\pa T}}_{0,\partial T}^2\\
    &\ls \widetilde{\eta}_{\M}^2+t^2(k,h_\T)\eta_\T^2+\wt\eta_\M^2+(k^4h_\T^{2+2\alpha}+k^2h_\T^{1+\alpha})\eta_\T^2\ls \widetilde{\eta}_{\M}^2+t^2(k,h_\T)\eta_\T^2.
  \end{aligned}
    }
By the similar argument, we arrive at
\eqn{
  \wt\eta_{\M}^2&\ls   \eta_{\M}^2 +t^2(k,h_\T)\eta_\T^2+\sum_{T\in\wt\M}h^2_{T}\nnorm{\wt R_{T}-\wideBar{\wt R_{T}}}^2_{0,T}+\sum_{T\in\M} h_T\norm{\wt R_{\pa T}-\wideBar{\wt R_{\pa T}}}_{0,\partial T}^2\\
  &\ls \eta_{\M}^2 +t^2(k,h_\T)\eta_\T^2.
}
Then we get the desired estimates \eqref{etaM1}--\eqref{etaM2}.
\end{proof}
It is natural to derive the following two corollaries. The first one is the equivalence on the whole mesh and the second one is an equivalence between the so-called D\"{o}rfler marking strategies for  $u_h$ and $P_hu$.
\begin{corollary}\label{cor:equivalent of eta}
  Under the assumptions of 
  Lemma~\ref{lem:equivalent relationships}, if in addition $t(k,h_\T)\leq\frac{1}{2C_{\rm eq}}$, then $\eta_\T\eqsim\widetilde{\eta}_\T$.
\end{corollary}
\begin{corollary}\label{cor:dorfler}
  Under the assumptions of 
  Lemma~\ref{lem:equivalent relationships}, it holds for $\theta_{D}\in(0,1], \eta_\M\geq\theta_D\eta_\T$, and $t(k,h_\T)\leq\frac{\theta_D}{C_{\rm eq}}$, 
  \eqn{
    \widetilde\eta_{\M}\geq\widetilde\theta_D\widetilde\eta_\T\quad \text{with} \quad \widetilde{\theta}_D=\frac{\theta_D-C_{\rm eq}t(k,h_\T)}{C_{\rm eq} \wt C_{\rm eq}(1+t(k,h_\T))}.
  }
  \end{corollary}
\begin{proof}
  It follows form \eqref{etaM1}-\eqref{etaM2} that
  \begin{align*}
    \wt\eta_{\M}\ge&\ \frac{\eta_{\M}}{C_{\rm eq}}-t(k,h_\T)\eta_{\T} \ge\ \frac{\theta_D\eta_{\T}}{C_{\rm eq}}-t(k,h_\T)\eta_{\T}=\ \widetilde{\theta}_D\wt C_{\rm eq}(1+t(k,h_\T))\eta_{\T}\ge\ \widetilde{\theta}_D\wt\eta_{\T}.
    \end{align*}
\end{proof}

\section{Convergence and quasi-optimality of the AFEM}\label{s:4}
\subsection{Adaptive algorithm}
We use $n(n \ge 0)$ to represent the iteration number in the adaptive algorithm. 
In the $n$-th loop, we first solve the FEM \eqref{eq:FEMva} on $\T_n$ 
, and then calculate the error estimators $\eta_T$ for each element $T\in\T_n$.
After computing the error estimators, we mark the `bad' mesh $\M_n$ by using the D\"orfler strategy. Finally we refine the mesh $\M_n$ by the newest vertex bisection algorithm (see, e.g., \cite{Cascon:Kreuzer:Nochetto:Siebert:2008,Karkulik:Pavlicek:Praetorius:2013}) at least $b\geq 1$ times and removing the hanging nodes and then set $n:=n+1$ to the first step.
We set $V_{n}:=V_{\T_n}$, $h_n:=h_{\T_n}$, $h_{n,{\rm min}}:=h_{\T_n,{\rm min}}$ for simplicity.



The basic loop of this adaptive algorithm is given below:
\bigskip

\begin{tabular}{|l|}
  \hline
 Given the initial triangulation $\T_0$ and marking parameter $\theta_D \in (0,1],$\\
  set $n:=0$ and iterate,\\
 1. Solve equation \eqref{eq:FEMva} on $\T_n$ to obtain $u_n,$\\
 2. Calculate the error estimator $\eta_T(u_n)$ on every element $T\in \T_{n},$\\
 3. Mark  $\M_n \subset \T_n$ with minimal cardinality such that
 $ \eta_{\M_n} \ge \theta_D \eta_{\T_n},$\\
 4. Refine $\M_n$ using the newest vertex bisection algorithm to get $\T_{n+1}$,\\
 5. Set $n:=n+1$ and then go to step 1.\\
 \hline
 \end{tabular}
 \bigskip

\begin{remark}
  {\rm (i)} 
  Unlike the algorithm in \cite{Duan:Wu:2023} which refines the broader mesh $\wt \M_n:=\bigcup_{T\in\M_n}\Sigma_{\T_n}(T)$ in order to apply the equivalent relationships of the error estimators to derive the convergence and quasi-optimality of ${u_n}$,
  the above adaptive algorithm is the same as the standard adaptive algorithm for elliptic problems since we modify the definitions of error estimators.

  {\rm (ii)} In step 1, we assume solving \eqref{eq:FEMva} is accurate but in practical computation we utilize Newton's iteration until the relative error between two consecutive iterative solutions falls below a prescribed tolerance.
  Indeed, when using the frozen-nonlinearity iteration with the iteration number fixed, it is named an adaptive Ka\v{c}anov method for quasi-linear elliptic partial differential equations (see, e.g., \cite{Garau:Morin:Zuppa:2011}).
  For our NLH problem, the convergence cannot be theoretically guaranteed. However, the numerical results are promising with an iteration number specified. We will study this convergence in our future work.
\end{remark}

In the next two subsections we shall prove the convergence  and quasi-optimality of the AFEM by following the ``elliptic projection argument'' developed in \cite{Duan:Wu:2023} for AFEM for the linear Helmholtz equation, that is, after establishing the equivalent relationships in   \S\,\ref{subs:equivalent relationship} between the error estimators for the FE solutions and those for the elliptic projections, we transform the analysis of the sequence of FE solutions $u_n$ obtained by the AFEM to the analysis of that of the corresponding elliptic projections, denoted by $\wt u_{n}$ of the exact solution $u$ onto $V_{n}$, which can be done by using the well-developed theory of AFEM for elliptic problems (see, e.g., \cite{Cascon:Kreuzer:Nochetto:Siebert:2008}). 

\subsection{Convergence of the AFEM}
 In this subsection we prove the convergence of the AFEM by proving the convergence of the sequence of the elliptic projections $\wt u_{n}$. But the error estimators $\wt\eta_\T$ for the elliptic projections do not satisfy the so-called  ``estimator reduction property" (see, e.g., \cite[Corollary 3.4]{Cascon:Kreuzer:Nochetto:Siebert:2008}) which is crucial in the analysis of AFEM for elliptic problems. 
To overcome this difficulty, we introduce an equivalent auxiliary error estimator for the elliptic projection by tuning the weights before the terms of element residuals:
\eq{\label{eq:wtetaT*}
  (\wt\eta_T^*)^2=h_T^2\Ltd{\wt R_T}{T}^2+h_T\Ltd{\wt R_{\pT}}{\pT}^2+\frac{1}{6}h_T^2\sum_{T'\in \Sigma_\T(T)\setminus  \{T\}}\Ltd{\wt R_{T'}}{T'}^2.
}
It is clear that there exist positive constants $c_1$ and $c_2$ such that
\eqn{
  c_1\wt\eta_T \leq \wt \eta^*_T \leq c_2\wt\eta_T.
}

  The first lemma gives the distance between the error estimators of two different discrete functions.
\begin{lemma}[local perturbation]\label{lem:local perturbation}
  For all $T\in\T, v_{h1},v_{h2}\in V_{\T}$, there exists a positive constant $C_{\rm per}$ which depends only on the shape regularity of $\T_0$, such that
  \eq{\label{eq:local perturbation}
    \big\vert\wt \eta^*_T(v_{h1})-\wt\eta^*_T(v_{h2})\big\vert\leq C_{\rm per}\HoD{v_{h1}-v_{h2}}{\Omega_\T(T)}.
  }
\end{lemma}
\begin{proof}
  It is from the definitions of $\wt\eta_T^*$, the triangle inequality, the trace inequality, and the inverse estimate that
  \eqn{
    \big\vert\wt\eta_T^*(v_{h1})-\wt\eta_T^*(v_{h2})\big\vert^2&\leq h_{T}^2\LtD{\wt R_T(v_{h1})-\wt R_T(v_{h2})}{T}^2+h_{T}\LtD{\wt R_{\pT}(v_{h1})-\wt R_{\pT}(v_{h2})}{\pT}^2\\
    &\quad +\frac1{6}h^2_T\sum_{T'\in \Sigma_{\T}(T)\setminus\{T\}}\LtD{\wt R_{T'}(v_{h1})-\wt R_{T'}(v_{h2})}{T'}^2 \\
    &=h_{T}^2\LtD{v_{h1}-v_{h2}}{T}^2+\frac {h_T}{4}\LtD{\jump{\nabla(v_{h1}-v_{h2})}}{\pT\cap\Om}^2+h_{T}\LtD{\nabla(v_{h1}-v_{h2})\cdot n}{\pT\cap\GaI}^2\\
    &\quad +\frac1{6}h^2_T\sum_{T'\in \Sigma_{\T}(T)\setminus\{T\}}\LtD{v_{h1}-v_{h2}}{T'}^2\\
    &\leq C_{\rm per}^2\big(h^2_T\LtD{v_{h1}-v_{h2}}{\Omega_{\T}(T)}^2+\LtD{\nabla(v_{h1}-v_{h2})}{\Omega_{\T}(T)}^2\big)\\
    &\leq C_{\rm per}^2\HoD{v_{h1}-v_{h2}}{\Omega_{\T}(T)}^2.
   }
   Then the conclusion comes from finding the square root.
\end{proof}

The next lemma is the estimator reduction property for the elliptic projection.
\begin{lemma}[estimator reduction]\label{lem:estimator reduction}
  For every $\delta>0$ and $b>0$, letting $\lambda_b:=\max \{3\cdot 2^{-(b+1)},2^{-\frac b2}\}$, there holds
  \eq{\label{eq:estimator reduction}
    \wt\eta_{\T_{n+1}}^*(\wt u_{n+1})^2\leq (1+\delta)\big(\wt \eta^*_{\T_n}(\wt u_n)^2-(1-\lambda_b)\wt \eta_{\M_n}^*(\wt u_n)^2\big)+(1+\delta^{-1})C_{\rm per}^2\Ho{\wt u_n-\wt u_{n+1}}^2. 
  }
\end{lemma}
\begin{proof}
  First, it can be directly derived from Lemma \ref{lem:local perturbation} and the Young's inequality that
  \eq{\label{eq:reduction eq1}
   \wt\eta_{\T_{n+1}}^*(\wt u_{n+1})^2\leq (1+\delta)\wt\eta_{\T_{n+1}}^*(\wt u_n)^2+ (1+\delta^{-1})C_{\rm per}^2\Ho{\wt u_{n+1}-\wt u_n}^2.
  }

Then for $T\in \T_n,$ we define $T_*:=\{T_1\in \T_{n+1}:\; T_1\subset T\}$, and handle $\wt\eta_{\T_{n+1}}^*(\wt u_{n})^2$ in 
two parts:
  \eqn{
 \wt\eta_{\T_{n+1}}^*(\wt u_{n})^2&=\sum_{\substack {T\in\M_n\\ T_1\in T_*}}\wt\eta_{T_1}^*(\wt u_{n})^2+\sum_{\substack {T\in \T_{n}\setminus\M_n\\ T_1\in T_*}}\wt\eta_{T_1}^*(\wt u_{n})^2.
}
Part \uppercase\expandafter{\romannumeral1}: for $T\in \M_n$, we have
\eqn{
 \sum_{T_1\in T_*}\wt\eta_{T_1}^*(\wt u_{n})^2&=\sum_{T_1\in T_*}\Big(h_{T_1}^2\Ltd{\wt R_{T_1}}{T_1}^2+h_{T_1}\Ltd{\wt R_{\pa T_1}}{\pa T_1}^2+\frac1{6}h^2_{T_1}\sum_{T'\in\Sigma_{\T_{n+1}}(T_1)\setminus\{T_1\}}\Ltd{\wt R_{T'}}{T'}^2\Big)\\
 &=\sum_{T_1\in T_*}\Big(2^{-b}h^2_{T}\Ltd{\wt R_{T_1}}{T_1}^2+2^{-\frac b2}h_{T}\Ltd{\wt R_{\pa T_1}}{\pa T_1}^2+\frac{2^{-b}}{6}h_T^2\sum_{\substack{T'\in\Sigma_{\T_{n+1}}(T_1)\setminus\{T_1\}\\T'\not\subset T}}\Ltd{\wt R_{T'}}{T'}^2\\
 &\quad +\frac{2^{-b}}{6}h_T^2\sum_{\substack{T'\in\Sigma_{\T_{n+1}}(T_1)\setminus\{T_1\}\\T'\subset T}}\Ltd{\wt R_{T'}}{T'}^2\Big)\\
 &\leq 2^{-b}h^2_{T}\Ltd{\wt R_T}{T}^2+2^{-\frac b2}h_{T}\Ltd{\wt R_{\pT}}{\pa T}^2+\frac{2^{-b}}{6}h^2_T\sum_{T'\in\Sigma_{\T_n}(T)\setminus\{T\}}\Ltd{\wt R_{T'}}{T'}^2 +\frac{2^{-b}\cdot 3}{6}h^2_T\Ltd{\wt R_T}{T}^2\\
 &\leq \lambda_b\wt\eta^*_T(\wt u_n)^2.
}
Since $\wt\eta^*_{\M_n}(\wt u_n)^2=\sum_{T\in\M_n}\wt\eta^*_T(\wt u_n)^2$,
we get
\eq{\label{eq:estimator reduction prat1}
 \sum_{\substack {T\in\M_n\\ T_1\in T_*}}\wt\eta_{T_1}^*(\wt u_{n})^2\leq \lambda_b\wt\eta^*_{\M_n}(\wt u_n)^2.
}

Part \uppercase\expandafter{\romannumeral2}: let $\mathcal R$ be the set of refined elements in $\T_n$. Exactly, $\M_n\subset\mathcal R$. For $T\in\T_n\setminus\mathcal R$, $T_*$ equals $\{T\}$, and  
\eqn{
 \sum_{T_1\in T_*}\wt \eta_{T_1}^*(\wt u_n)^2=\wt \eta_{T_1}^*(\wt u_n)^2\leq \wt \eta_{T}^*(\wt u_n)^2.
}
On the other hand, for $T\in(\T_n\setminus\M_n)\cap \mathcal R$, similar to the first estimate in in Part \uppercase\expandafter{\romannumeral1}, we derive
\eqn{
 \sum_{T_1\in T^*}\wt \eta_{T_1}^*(\wt u_n)^2&=\sum_{ T_1\in T_*}\Big(h_{T_1}^2\Ltd{\wt R_{T_1}}{T_1}^2+h_{T_1}\Ltd{\wt R_{\pa T_1}}{\pa T_1}^2+\frac1{6}h^2_{T_1}\sum_{T'\in\Sigma_{\T_{n+1}}(T_1)\setminus\{T_1\}}\Ltd{\wt R_{T'}}{T'}^2\Big)\leq \lambda_1\wt \eta_{T}^*(\wt u_n)^2.
}
As a conclusion, we can get
\eq{\label{eq:estimator reduction prat2}
 \sum_{\substack {T\in \T_{n}\setminus\M_n\\ T_1\in T_*}}\wt\eta_{T_1}^*(\wt u_{n})^2\leq  \sum_{T\in \T_{n}\setminus\M_n}\wt\eta_{T}^*(\wt u_{n})^2 = \wt\eta_{\T_n}^*(\wt u_n)^2-\wt\eta_{\M_n}^*(\wt u_n)^2.
}
Then combining \eqref{eq:estimator reduction prat1} and \eqref{eq:estimator reduction prat2}, we have
\eq{\label{eq:reduction eq2}
 \wt\eta_{\T_{n+1}}^*(\wt u_{n})^2\leq \wt\eta_{\T_n}^*(\wt u_n)^2-(1-\lambda_b)\wt\eta_{\M_n}^*(\wt u_n)^2.
}

Finally, \eqref{eq:estimator reduction} comes from \eqref{eq:reduction eq1} and \eqref{eq:reduction eq2}. This completes the proof of the lemma.
\end{proof}

  The last theorem in this subsection is showing the convergence of the AFEM in the case of low energy.
\begin{theorem}[convergence]\label{thm:contraction property}
  Let $\theta_D\in (0,1]$. There exists a constant $C_{\rm conv} \in (0, 1)$ depending only on $\T_{0}, b$, and $\theta_D$,  such that if $\vert \ln h_{n,{\rm min}}\vert \vep\hMfg^2\leq\theta_0 \ \text{and} \  k^3h^{1+\alpha}_0\le C_0$, then
  \eqn{\He{u-u_n} \ls  C_{\rm conv}^n(1+k^2h_n)\big(kh_0+k^{-\frac12}(kh_0)^{\alpha}\big) \hMfg. }
\end{theorem}
\begin{proof}
    The key point is proving the contraction property: there exist constants $\beta>0$ and $C_{\rm conv}\in (0,1)$, such that
    \eq{\label{eq:contraction property}
    \Ho{u-\wt u_{n+1}}^2+\beta \wt\eta_{\T_{n+1}}^*(\wt u_{n+1})^2 \leq C_{\rm conv}^2\big(\Ho{u-\wt u_n}^2+\beta \wt\eta_{\T_n}^*(\wt u_n)^2\big).
    }

  Combining \eqref{eq:GO} and Lemma \ref{lem:estimator reduction}, we choose $\beta=\frac1{(1+\delta^{-1})C_{\rm per}^2}$ to get
  \eqn{
    \Ho{u-\wt u_{n+1}}^2+\beta\wt\eta_{\T_{n+1}}^*(\wt u_{n+1})^2&\leq\Ho{u-\wt u_n}^2-\Ho{\wt u_{n+1}-\wt u_n}^2+\beta(1+\delta)\big(\wt\eta_{\T_{n}}^*(\wt u_{n})^2-(1-\lambda_b) \wt\eta_{\M_{n}}^*(\wt u_{n})^2\big)\\
    &\quad +\beta(1+\delta^{-1})C_{\rm per}^2\Ho{\wt u_{n+1}-\wt u_n}^2\\
    &=\Ho{u-\wt u_n}^2+\beta(1+\delta)\big(\wt\eta_{\T_{n}}^*(\wt u_{n})^2-(1-\lambda_b) \wt\eta_{\M_{n}}^*(\wt u_{n})^2\big).
  }
  Noting that $\wt\eta_{\M_{n}}(\wt u_{n})\geq \wt\theta_{\rm D} \wt\eta_{\T_n}(\wt u_n)$ (see Corollary~\ref{cor:dorfler}) and $\Ho{u-\wt u_n}\leq C_{\rm up}\wt\eta_{\T_n}$ for some constant  $C_{\rm up}>0$ (see  \eqref{thm:upbound2}), we have
  \eqn{
    \wt \eta^*_{\M_n}\geq c_1\wt\eta_{\M_n}\geq c_1\wt\theta_{\rm D}\wt\eta_{\T_n}\geq\frac{c_1\wt\theta_{\rm D}}{c_2}\wt \eta^*_{\T_n},\quad \Ho{u-\wt u_n}\leq \frac{C_{\rm up}}{c_1}\wt\eta_{\T_n}^*.}
 Then, for any  $\mu \in (0,1)$,
    \eqn{
    \Ho{u-\wt u_{n+1}}^2+\beta(\wt\eta_{\T_{n+1}}^*)^2&\leq\Big(1-\beta(1+\delta)(1-\lambda_b)\mu\Big(\frac{c_1\wt \theta_{\rm D}}{c_2C_{\rm up}}\Big)^2\Big)\Ho{u-\wt u_n}^2\\
    &\quad +\beta(1+\delta)\Big(1-(1-\lambda_b)(1-\mu)\Big(\frac{c_1\wt\theta_{\rm D}}{c_2}\Big)^2\Big)(\wt\eta_{\T_{n}}^*)^2.
  }
   To obtain \eqref{eq:contraction property}, we let
  \eqn{
    C_{\rm conv}^2=\max\Big\{1-\beta(1+\delta)(1-\lambda_b)\mu\Big(\frac{c_1\wt \theta_{\rm D}}{c_2C_{\rm up}}\Big)^2, \,(1+\delta)\Big(1-(1-\lambda_b)(1-\mu)\Big(\frac{c_1\wt\theta_{\rm D}}{c_2}\Big)^2\Big)\Big\}
  }
  and choose $\delta$ sufficiently small such that $C_{\rm conv}<1$.

   Finally, it follows from the upper bounds \eqref{thm:upbound1} and \eqref{thm:upboundL2}, the equivalences of $\eta_\T, \wt\eta_\T,$ and $\wt\eta_\T^*$,  \eqref{eq:contraction property}, and the lower bound \eqref{thm:lowbound2}  that
    \eqn{
      \He{u-u_n} &\ls (1+k^{\frac12}(kh_n)^{\alpha}+k^2h_n) \eta_{\T_n}\ls (1+k^2h_n) \wt\eta_{\T_n}^*\\
    &\ls  (1+k^2h_n)(\Ho{u-\wt u_{n}}^2+ \beta(\wt\eta_{\T_n}^*)^2)^{1/2}\\
    &\ls (1+k^2h_n)(\Ho{u-\wt u_{0}}+ \wt\eta_{\T_0}^*)C_{\rm conv}^n\\
    &\ls (1+k^2h_n) (\Ho{u-\wt u_{0}}+ {\rm osc}_{\T_0})C_{\rm conv}^n. 
    }
    Altogether with
    \begin{align}
    &{\rm osc}_{\T_0} = \Big(\sum_{T\in \T_0} \big(h_T^2\nnorm{f-\wideBar f}^2_{0,T}+ h_T\nnorm{g-\wideBar g}_{0,\pa T\cap \GaI}^2\big)\Big)^\frac12\ls h_0 \big(\nnorm{f}_0+\nnorm{g}_{\frac12, \GaI}\big),\label{eq:osc to Mfg}\\
    &\Ho{u-\wt u_0}\ls (kh_0+k^{-\frac12}(kh_0)^{\alpha})\hMfg,\notag
    \end{align}
    we get the desired convergence result.
  \end{proof}

\subsection{Quasi-optimality of the AFEM}
This subsection concerns the quasi-optimality of the AFEM. Again we consider the quasi-optimality of $\{\wt u_n\}_{n\geq0}$ instead, which can be proved by following the analysis of the AFEM for elliptic problems (see, e.g., \cite[\S 5]{Cascon:Kreuzer:Nochetto:Siebert:2008}), but 
with a key difference that we have to use the error estimators $\eta_T(u_h)$ (instead of $\wt\eta_T(\wt u_h)$ as usual) to characterize the localized upper bound and the optimal marking property associated with the elliptic projections (see Lemmas~\ref{lem:localized upper bound} and \ref{lem:optimal marking} below). Of course, unlike the analysis for the linear Helmholtz equation \cite{Duan:Wu:2023}, we have to deal with the effects of the nonlinear term, where the upper bounds \eqref{thm:upboundL2}--\eqref{thm:upboundboundary} will play an important role. The contraction property \eqref{eq:contraction property} based on our newly introduced auxiliary error estimator in \eqref{eq:wtetaT*} is also essential in our proof of the quasi-optimality. We will only give a sketch of the proof by mainly showing the major differences from those in \cite{Cascon:Kreuzer:Nochetto:Siebert:2008} or \cite{Duan:Wu:2023}.

The first lemma below says that the 
 {\it total error} $\big(\Ho{u-\wt u_n}^2+\wt\osc_{\T}^2(\wt u_n)\big)^{\frac12}$ (see \cite{Cascon:Kreuzer:Nochetto:Siebert:2008,Mekchay:Nochetto:2005}) of the elliptic projection $\wt u_n\in V_n$ of $u$ satisfies the Cea's lemma.

\begin{lemma}[optimality of the total error of $\wt u_h$]\label{lem:qo total error}
  There exists the following estimate
  \eq{\label{eq:qo total error}
   \Ho{u-\wt u_n}^2+\wt\osc_{\T_n}^2(\wt u_n)\leq \inf_{v_h\in V_{n}}\big(\Ho{u-v_h}^2+\wt\osc_{\T_n}^2(v_h)\big).
  }
\end{lemma}
\begin{proof}
  By the orthogonality of $\wt u_n$, we have for any $v_h\in V_{n}$, 
  \eqn{
    \Ho{u-\wt u_n}^2=\Ho{u-v_h}^2-\Ho{\wt u_n-v_h}^2\leq\Ho{u-v_h}^2.
  }
  But we see from the definition of $\wt\osc$ that
  $\wt\osc_{\T_n}(v_h)$ is independent of $v_h\in V_{n}$, hence \eqref{eq:qo total error} is derived.
\end{proof}

\begin{lemma}[localized upper bound] \label{lem:localized upper bound}
 Let $\T_*$ be a refinement of $\T$ and $\mathcal R$ be the set of refined elements. Let $\wt u_{\T_*}$ and $\wt u_\T$ be the corresponding elliptic projections of $u$ in $V_{\T_*}$ and $V_\T$, respectively.
  Then there exists a constant $C_{\rm lub}$ such that if $\vert\ln h_{\T,{\rm min}}\vert \vep\hMfg^2\leq\theta_0 \ \text{and} \  k^3h_\T^{1+\alpha}\le C_0$, then
  \eq{\label{eq: lub}
   \Ho{\wt u_{\T_*}-\wt u_\T}^2\leq C_{\rm lub}\big(\eta_{\mathcal R}^2(u_\T)+t^2(k,h_\T)\eta_\T^2(u_\T)\big).
  }
\end{lemma}
\begin{proof}
  Let $\Omega_{\mathcal R}=\cup\{T:T\in \mathcal R\}$ be the union of refined elements, and denote by $\Omega_i, i=1,\ldots,I$, the connected components of its interior. We denote
  $\T^i=\{T\in\T: T\subset \bar \Omega_i\}$ and $V_{\T^i}$ be the restriction of $V_\T$ to $\Omega_i$.\\
  For the error $E_*:=\wt u_{\T_*}-\wt u_\T$, we construct $\Pi_hE_*\in V_\T$ by
  \eqn{
    \Pi_hE_*:=\pihi E_*\quad \text{in} \ \Omega_i;
         \quad \Pi_hE_*:=E_* \quad \text{elsewhere} .
  }
  where $\pihi: H^1(\Omega_i)\rightarrow V_{\T^i}$ is the standard Scott-Zhang interpolation operator satisfying $\pihi E_*=E_*$ on each $\pa\Om_i$ (see, e.g. \cite{Scott:Zhang:1990}). Indeed, since $E_*$ has conforming boundary values on $\pa\Omega_i$ in $V_{\T^i}$, we have $\Pi_hE_*\in V_\T$.

  We denote $\psi_1=E_*-\Pi_hE_*$. By the orthogonality and error estimate of the elliptic projection, \eqref{eq:estimator indetity} with $\phi=\psi_1-P_h\psi_1$ and $v_h=-P_h\psi_1\in V_\T$, and upper bounds \eqref{thm:upboundL2}--\eqref{thm:upboundboundary}, we derive
  \eqn{
    b(E_*,E_*)&=b(E_*,E_*-\Pi_hE_*)=b(u-\wt u_\T,\psi_1)\\
    &=b(u-\wt u_\T,\psi_1-P_h\psi_1)=b(u-u_\T,\psi_1-P_h\psi_1)\\
    &=a_u(u-u_\T,\psi_1-P_h\psi_1)+\big((k^2+1+k^2\vep\oneo\abs{u}^2)(u-u_\T),\psi_1-P_h\psi_1\big)\\
    &\quad-ik\DP{u-u_\T}{\psi_1-P_h\psi_1}_{\GaI} \\
    &\ls \eta_{\mathcal R}(u_\T)\Lt{\nabla\psi_1}+k^2\vep\abs{\big((\abs{u}^2-\abs{u_\T}^2)u_\T,\psi_1-P_h\psi_1\big)_{\Omz}}\\
    &\quad +k^2\Lt{u-u_\T}h_\T^{\alpha}\Ho{\psi_1}+kh^{\frac{\alpha}2}_\T\LtD{u-u_\T}{\GaI}\Ho{\psi_1}\\
    &\ls \eta_{\mathcal R}(u_\T)\HoD{E_*}{\Omega_{\mathcal R}}+t(k,h_\T)\eta_\T\HoD{E_*}{\Omega_{\mathcal R}}.
  }
This completes the proof of the lemma.
\end{proof}

Recalling \cite[Lemma 2.3]{Cascon:Kreuzer:Nochetto:Siebert:2008}, we can directly get the following Lemma.
\begin{lemma}[complexity of refine]\label{lem:complexity of refine}
  Assume that $\T_0$ satisfies the conditions of \cite[Lemma 2.3]{Cascon:Kreuzer:Nochetto:Siebert:2008}. Let $\{\T_n\}_{n\geq 0}$ be any sequence of refinements of $\T_0$ where $\T_{n+1}$ is generated
from $\T_n$ by refinement step in the n-th loop, then there exist a constant $C_{\rm com}$ solely depending on initial triangulation $\T_0$ and bisection number $b$ such that
\eq{\label{eq:complexity of refine}
   \# \T_n-\#\T_0\leq C_{\rm com}\sum_{j=0}^{n-1}\#\M_j\quad \forall n\geq 1.
}
\end{lemma}

Let $\mathbb T$ be the set of all conforming triangulations obtained by refining $\T_0$ a finite number
of times by the newest vertex bisection algorithm. Let \eqn{
  \mathbb{T}_N:=\{\T\in\mathbb T : \#\T-\#\T_0\leq N\}.
}
Next we define an approximation class $\mathbb A_s$ based on the total error via
\eqn{
  \mathbb A_s:=\Big\{(v,f,g,k):\sup_{N>0}N^s\big(\inf_{\T\in\mathbb T_N}\inf_{V\in V_\T}(\Ho{v-V}^2+\wt\osc^2_\T(V))^{\frac12}\big)<\infty \Big\},
}
and if $(v,f,g,k)\in\mathbb A_s$, we define
\eqn{
  \big\vert(v,f,g,k)\big\vert_s:=\sup_{N>0}N^s\Big(\inf_{\T\in\mathbb T_N}\inf_{V\in V_\T}\big(\Ho{v-V}^2+\wt\osc^2_\T(V)\big)^{\frac12}\Big).
}

  We remark that the definition of $\mathbb A_s$ is standard \cite{Cascon:Kreuzer:Nochetto:Siebert:2008} since it is based on the total error of the elliptic projection $P_hu$.
 The next lemma shows that the exact solution satisfies $(u,f,g,k)\in \mathbb A_{\frac12}$.

\begin{lemma}\label{lem:ufg A1/2}
  If $\vep\Mfg^2\leq\theta_0$, then $(u,f,g,k)\in \mathbb A_{\frac12}$ and $\vert (u,f,g,k)\vert_{\frac12}\ls k\hMfg$.
\end{lemma}
\begin{proof}
  Since $u$ has a decomposition as Theorem \ref{thm:cpstab} said, the proof is similar with \cite[Lemma 5.4]{Duan:Wu:2023} by using \cite[Theorems 5.2 and 5.3]{Gaspoz:Morin:2009} except the estimate for $\wt \osc_\T(V)$: for $V\in V_\T$,
  \eqn{
    \wt \osc_\T^2(V)&\leq \sum_{T\in\T}\Big(h_T^2\LtD{f+\big(k^2(1+\vep\oneo\abs{u}^2)+1\big)u}{T}^2 +h_T\LtD{\mbi ku-g-\wideBar{\mbi ku-g}}{\pT\cap\GaI}^2\Big)\\
    &\ls h^2_\T(\Lt{f}^2+\norm{g}_{\frac12,\GaI}^2)+(k^4+k^2\theta_0^2)h_\T^2\Lt{u}^2+k^2h_\T^2\Ho{u}^2\\
    &\ls (k\hMfg h_\T)^2,
    }
  where the estimate is the same as in \eqref{eq:osc to Mfg} and using \eqref{eq:cpstab}. We omit the other details of the proof here. 
\end{proof}

The next Lemma establishes a link between the total error of the elliptic projection and the D\"{o}rfler strategy.
\begin{lemma}[optimal marking]\label{lem:optimal marking}
  Assume that $\theta_{\rm D}\in(0,\theta_*)$ with $\theta_*^2=\frac1{1+4C_{\rm lb}(C_{\rm lub}+C_{\rm osc}+1)}$.
  Set $\mu:=\frac{\theta_*^2-\theta_{\rm D}^2}{2\theta^2_*}$. Let $\T_*$ be a refinement of $\T$, $\mathcal R$ be the set of refined elements in $\T$,  $\wt u_\T\in V_\T$ and $\wt u_{\T_*}\in V_{\T_*}$ be the elliptic projections of the exact solution $u$. If $\vert\ln h_{\T_*,{\rm min}}\vert\vep\hMfg^2\leq\theta_0$, $k^{3}h_\T^{1+\alpha}\leq C_0$ and
  \eqn{
    \Ho{u-\wt u_{\T_*}}^2+\wt\osc_{\T_*}^2(\wt u_{\T_*})\leq \mu\big(\Ho{u-\wt u_\T}^2+\wt\osc_\T^2(\wt u_\T)\big),
  }
   then it holds that 
   $ 
    \eta_{\mathcal R}(u_\T)\geq \theta_{\rm D}\eta_\T(u_\T).
  $ 
\end{lemma}
\begin{proof}
  It follows from the lower bound \eqref{thm:lowbound1} and the orthogonality of the elliptic projection that
  \eqn{
    \frac{1-\mu}{C_{\rm lb}}\eta_\T^2&\leq (1-\mu)\big(\Ho{u-\wt u_\T}^2+\wt\osc_\T^2(\wt u_\T)\big)\\
    &\leq \Ho{u-\wt u_\T}^2+\wt\osc_\T^2(\wt u_\T)-\big(\Ho{u-\wt u_{\T_*}}^2+\wt\osc_{\T_*}^2(\wt u_{\T_*})\big)\\
    &=\Ho{\wt u_{\T_*}-\wt u_\T}^2+\wt\osc_\T^2(\wt u_\T)-\wt\osc_{\T_*}^2(\wt u_{\T_*}).
    }
Since $\wt\osc_\T(v_\T)$ is independent of $v_\T$, we have
  \eqn{
    \wt\osc_\T^2(\wt u_\T)-\wt\osc_{\T_*}^2(\wt u_{\T_*})\leq \wt\osc_{\mathcal R}^2(\wt u_\T).
  }
  Then with the aid of \eqref{eq:wtosc osc distance}, Lemma \ref{lem:localized upper bound}, we get
  \eqn{
    \frac{1-\mu}{C_{\rm lb}}\eta_\T^2&\leq C_{\rm lub}\big(\eta_{\mathcal R}^2(u_\T)+t^2(k,h_\T)\eta^2_\T(u_\T)\big)+2\osc_{\mathcal R}^2(u_\T)+2C_{\rm osc}(k^4h_\T^{2+2\alpha}+k^2h_\T^{1+\alpha})\eta_\T^2(u_\T)\\
    &\leq 2(C_{\rm lub}+1+C_{\rm osc})\big(\eta_{\mathcal R}^2(u_\T)+(t^2(k,h_\T)+k^4h_\T^{2+2\alpha}+k^2h_\T^{1+\alpha})\eta_\T^2(u_\T)\big),
  }
  implying
  \eqn{
    \eta_{\mathcal R}^2(u_\T)&\geq \Big(\frac{1-\mu}{2C_{\rm lb}(C_{\rm lub}+1+C_{\rm osc})}-\theta_*^2\Big)\eta_\T^2
    \geq \big(2(1-\mu)\theta^2_*-\theta_*^2\big)\eta_\T^2= \theta_{\rm D}^2\eta^2_{\T}(u_\T),
  }
    where we use $t^2(k,h_\T)+k^4h_\T^{2+2\alpha}+k^2h_\T^{1+\alpha}\leq\theta_*^2$ since $k^3h_\T^{1+\alpha}$ is sufficiently small.
\end{proof}

\begin{lemma}[cardinality of $\M_n$]\label{lem:cardinality of Mn}
  If $\vert\ln h_{n,{\rm min}}\vert\vep\hMfg^2\leq\theta_0$, $k^3h_0^{1+\alpha}\le C_0$, and $\theta_{\rm D}\in (0,\theta_*)$(with $\theta_*$ from Lemma~\ref{lem:optimal marking}),
  then
  \eqn{
    \#\M_n\leq C_{\rm card} \Big(1-\frac{\theta_{\rm D}^2}{\theta_*^2}\Big)^{-1}\vert (u,f,g,k)\vert_{\frac12}^{2}\big(\Ho{u-\wt u_n}^2+\wt\osc^2_{\T_n}(\wt u_n)\big)^{-1}.
  }
\end{lemma}
\begin{proof}
  Denote $\epsilon^2:=\mu(\Ho{u-\wt u_n}^2+\wt\osc_{\T_n}^2), N:=\lceil \vert (u,f,g,k)\vert_{\frac12}^{2}\epsilon^{-2} \rceil$ with $\mu=\frac{\theta_*^2-\theta_{\rm D}^2}{2\theta_*^2}$,
  then from Lemma \ref{lem:ufg A1/2}, there exists $\T_{\epsilon}\in \mathbb T, V_{\epsilon}\in V_{\T_{\epsilon}}$ such that
  \eqn{
    \#\T_{\epsilon}-\#\T_0\leq N \qaq N^{\frac12}\big(\Ho{u-V_{\epsilon}}^2+\wt\osc^2_{\T_{\epsilon}}(V_{\epsilon})\big)^{\frac12}\leq \vert (u,f,g,k)\vert_{\frac12}.
  }
  Let $\T_*:=\T_{\epsilon}\oplus \T_n$ (see \cite[ \S 2.2]{Cascon:Kreuzer:Nochetto:Siebert:2008}), $\wt u_{\T_*}\in V_{\T_*}$ is the elliptic projection of $u$ into $V_{\T_*}$. It follows from Lemma \ref{lem:qo total error} that
  \eqn{
   \Ho{u-\wt u_{\T_*}}^2+\wt\osc^2_{\T_*}(\wt u_{\T_*})&\leq \Ho{u-V_{\epsilon}}^2+\wt\osc^2_{\T_{\epsilon}}(V_{\epsilon})\leq \mu\big(\Ho{u-\wt u_n}^2+\wt\osc^2_{\T_n}\big).
  }
  So we can use Lemma \ref{lem:optimal marking} to derive
$ 
    \eta_{\mathcal R}(u_n)\geq \theta_{\rm D}\eta_{\T_n}(u_n),
$  
  where $\mathcal R$ is the set of refined elements in $\T_n$.

  Finally, since step 3 marks $\M_n\subset\T_n$ with minimal cardinality, combining with the property of the overlay of meshes (see \cite[Lemma 3.7]{Cascon:Kreuzer:Nochetto:Siebert:2008}) $\#\T_*\leq\#\T_\epsilon+\#\T_n-\#\T_0$ we conclude
  \eqn{
    \#\M_n\leq\#\mathcal R\leq \#\T_*-\#\T_n\leq\T_\epsilon-\#\T_0\leq N.
  }
  The proof is completed together with the definitions of $N, \epsilon$.
\end{proof}

\begin{theorem}[quasi-optimality]\label{thm:quasi-optimality}
  Assume that $\vert\ln h_{n,{\rm min}}\vert\vep\hMfg^2\leq\theta_0$, $k^3h_0^{1+\alpha}\leq C_0$, and $\theta_{\rm D}\in (0,\theta_*)$ (with $\theta_*$ from Lemma~\ref{lem:optimal marking}), then
  \eq{\label{eq:quasi-optimality}
    \He{u-u_n}\ls (1+k^2h_n)k\hMfg(\#\T_n-\#\T_0)^{-\frac12}.
    }
\end{theorem}
\begin{proof}
  Let $M=(1-\frac{\theta_{\rm D}^2}{\theta_*^2})^{-1}\vert(u,f,g,k)\vert_{\frac12}^{2}$. Using Lemmas \ref{lem:complexity of refine} and \ref{lem:cardinality of Mn}, the lower bound \eqref{lowbound elliptic projection}, the equivalence of $\wt\eta_\T$ and $\wt\eta_\T^*$ and the contraction property \eqref{eq:contraction property}, we derive
  \eqn{
    \#\T_n-\#\T_0&\ls\sum_{j=0}^{n-1}\#\M_j\ls M\sum_{j=0}^{n-1}\big(\Ho{u-\wt u_j}^2+\wt\osc^2_{\T_j}\big)^{-1}\\
    &\ls M\sum_{j=0}^{n-1}\big(\Ho{u-\wt u_j}^2+\beta\wt\eta_{\T_j}^2\big)^{-1}\eqsim M\sum_{j=0}^{n-1}\big(\Ho{u-\wt u_j}^2+\beta(\wt\eta_{\T_j}^*)^2\big)^{-1}\\
    &\ls M\sum_{j=1}^{n}C_{\rm conv}^{2j}\big(\Ho{u-\wt u_n}^2+\beta(\wt\eta^*_{\T_n})^2\big)^{-1}\ls \frac{MC_{\rm conv}^{2}}{1-C_{\rm conv}^{2}}\big(\Ho{u-\wt u_n}^2+\beta\wt\eta^2_{\T_n}\big)^{-1}.
  }
  Finally, from upper bounds \eqref{thm:upboundL2} and \eqref{thm:upbound1} and Corollary \ref{cor:equivalent of eta}, we have
  \eqn{
    \He{u-u_n}&\ls (1+k^2h_n)\frac{C_{\rm conv}}{(1-C_{\rm conv}^{2})^{\frac12}}M^{\frac12}(\#\T_n-\#\T_0)^{-\frac12}\\
    &\ls (1+k^2h_n)k\hMfg(\#\T_n-\#\T_0)^{-\frac12}.
  }
\end{proof}

\begin{remark}
  When the pollution error dominates the error, the mesh is quasi-uniformly refined numerically, so the second term in \eqref{eq:quasi-optimality} is nearly $k^3(\#\T_n-\#\T_0)^{-1}$ which is the quasi-optimal decaying rate of pollution error.
  Once the pollution error disappears, the first term dominates the upper bound in \eqref{eq:quasi-optimality} which seems as the interpolation error of the regular part and it is quasi-optimal.
\end{remark}

\section{Numerical results}\label{s:5}
\subsection{CIP Finite element method (CIPFEM)}
The CIPFEM, initially proposed in \cite{Douglas:Dupont:1976} for elliptic and parabolic problems in the 1970s, has recently demonstrated significant potential in addressing the Helmholtz problem with high wave number \cite{Wu:2013,Zhu:Wu:2013,Du:Wu:2015,Li:Wu:2019}. 
The CIPFEM utilizes the same approximation space as the FEM but adjusts the sesquilinear form of the FEM by adding a penalty term for the jump of the normal derivative of the discrete solution at interior edges.
Let the energy space $V$ and the sesquilinear form $a_\psi^\gamma (\cdot,\cdot)$ on $V\times V$ as
\begin{align*}
V &:= \HoG\cap \prod_{K\in\T} H^{\frac32+\tilde \vep}(K), \\
a_\psi^\gamma (u,v) &:= a_\psi(u,v)+J(u,v) \quad \forall u,v\in V, \\
J(u,v)&:=\sum_{e\in\E^I}\gamma_e h_e \ine{\jump{\na u}}{\jump{\na v}},
\end{align*}
where $\tilde\vep$ is some positive number and the penalty parameters $\gamma_e$ for $e\in\E^I$ are numbers with nonnegative imaginary parts, respectively.
It can be easily verified that, if $u\in H^{\frac32+\tilde\vep}(\Om)$ is the solution to the NLH problem \eqref{eq:Helm}, then $J(u,v)=0\quad\forall v\in V.$
The CIPFEM for \eqref{eq:NLHva}
reads as find $u_h\in V_\T$ such that
\begin{equation}\label{eq:cip-fem}
a^\gamma_{u_h}(u_h,v_h)= (f,v_h)+\langle g,v\rangle_{\GaI}\quad \forall v_h\in V_\T.
\end{equation}
Newton's iteration for solving the CIPFEM \eqref{eq:cip-fem} reads for a given $u_h^0\in V_\T$, find ${u_h^{l+1}} \in V_\T,~l=0,1,2,\cdots,$ such that
\begin{equation}\label{eq:dip0}
    a_{\sqrt 2u_h^l}^\gamma(u_h^{l+1},v_h)-k^2\vep\oneo\big((u_h^l)^2\overline{u_h^{l+1}},v_h\big)=(f-2k^2\vep\oneo\big\vert u_h^l\big\vert^2 u_h^l,v_h)+\langle g,v_h\rangle_{\GaI}\quad \forall v_h\in V_\T.
\end{equation}
The other two iterations can be similarly constructed.

Remark that when $\gamma_e\equiv 0$, the CIPFEM reduces to the standard FEM. 
Combining the analyses in \cite{Zhou:Zhu:2015} and  \S\,\ref{s:FEM error estimate}, we can similarly derive the error estimates \eqref{eq:error estimate} for the CIPFEM.

In the following numerical examples, we use Newton's iteration for its quadratic convergence. Set $\theta_{\rm D}=0.4$.  We choose $\gamma_e\equiv\gamma=-\sqrt{3}/24+0.005\mbi$ where the real part can reduce pollution effect by a dispersion analysis for 2D problems on equilateral triangulations \cite{Han:2012} and
the imaginary part makes use of the stability for ACIPFEM. Let $N$ represent the number of elements.

\subsection{Accuracy}\label{s:5.2}
We consider the NLH problem \eqref{eq:Helm}--\eqref{eq:Dir} with the domain as shown in the left part of figure \ref{fig:ex1pet}.
As seen in this figure, $\Om$ is a polygon formed by seven points $(4R,4R)$, $(-3R,4R)$, $(-3R,-4R)$, $(4R,-4R)$, $(\frac72R,-\frac32R\tan \frac{\pi}{40}),$ $(2R,0)$, $(\frac72R,\frac32R\tan \frac{\pi}{40})$.
$\Omz$ is the square $[-R, R]\times [-R, R]$ inside, $\GaD$ is the four right hand edges and the other three edges constitute $\GaI$.
So $\alpha=\frac{20}{39}$ and we take the exact solution at the corner $u_1=\chi_1(r)J_{\alpha}(kr)\sin(\alpha(\theta-\frac{\pi}{40}))$, where the cut-off function  $\chi_1(r)=(1-\frac{r}{R})^2,\ r\leq R;\ \chi_1(r)=0, \ r>R$.
For the nonlinearity part, we take the Kerr constant to be $\vep=k^{-2}$, and the exact solution $u=u_1+u_2$ 
as the form of a fundamental soliton (cf. \cite{Chamorro-Posada:McDonald:New:2002}) multiplied by a cut-off function $\chi_2(x,y)$:
\eqn{u_2=\chi_2\frac{q\sqrt{2}e^{\mbi y\sqrt{k^2+q^2}}}{\sqrt{\vep}k\cosh(qx)},\quad
 \chi_2(x,y)= \begin{cases}
  (x^2-R^2)^2, & \quad -R\leq x\leq R,\\
  0, & \quad \text{otherwise}.
\end{cases}
}

First, we let $k=80, q=20$ and the element size of the initial mesh $h_{\T_0}=0.04$. The top-left part of figure \ref{fig:ex1pet} shows the global mesh after 20 adaptive iterations using CIPFEM, while the top-right one is the real part of the real solution on this mesh.
We can observe that the solution oscillates rapidly around the line $x=0$, exhibiting high energy in this region. On the other hand, the reentrant corner should be refined due to the presence of singularities around $(\frac12,0)$.
Consequently, the mesh is denser at the reentrant corner and in the high-energy regions of the solution compared to the low-energy regions in theory. These observations are depicted in the bottom of figure~\ref{fig:ex1pet}.
\begin{figure}[tbp]
  \begin{minipage}[c]{0.47\columnwidth}
  \centering
  \includegraphics[width=0.66\columnwidth]{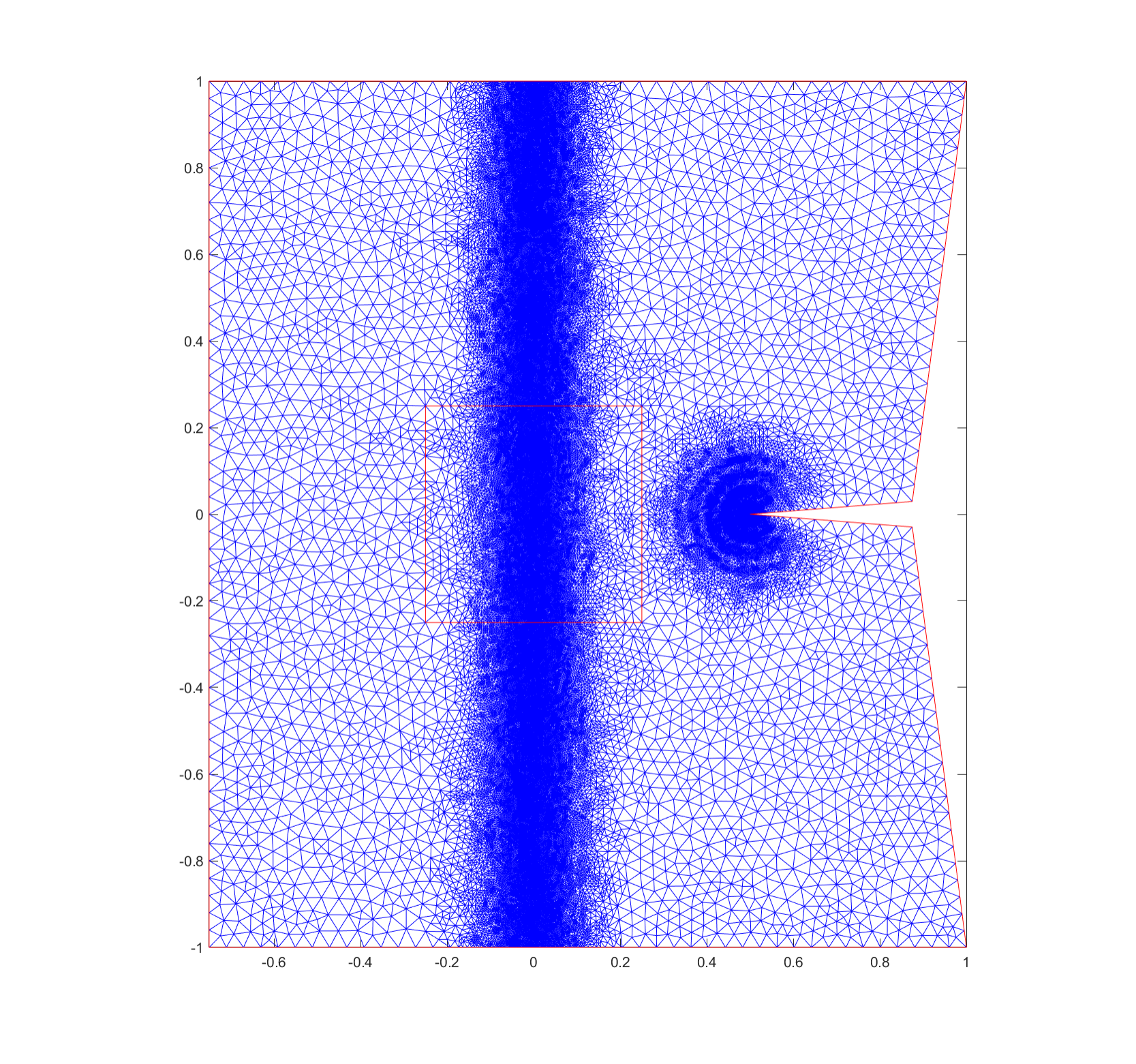}
  \end{minipage}
  \begin{minipage}[c]{0.47\columnwidth}
  \centering
  \includegraphics[width=0.66\columnwidth]{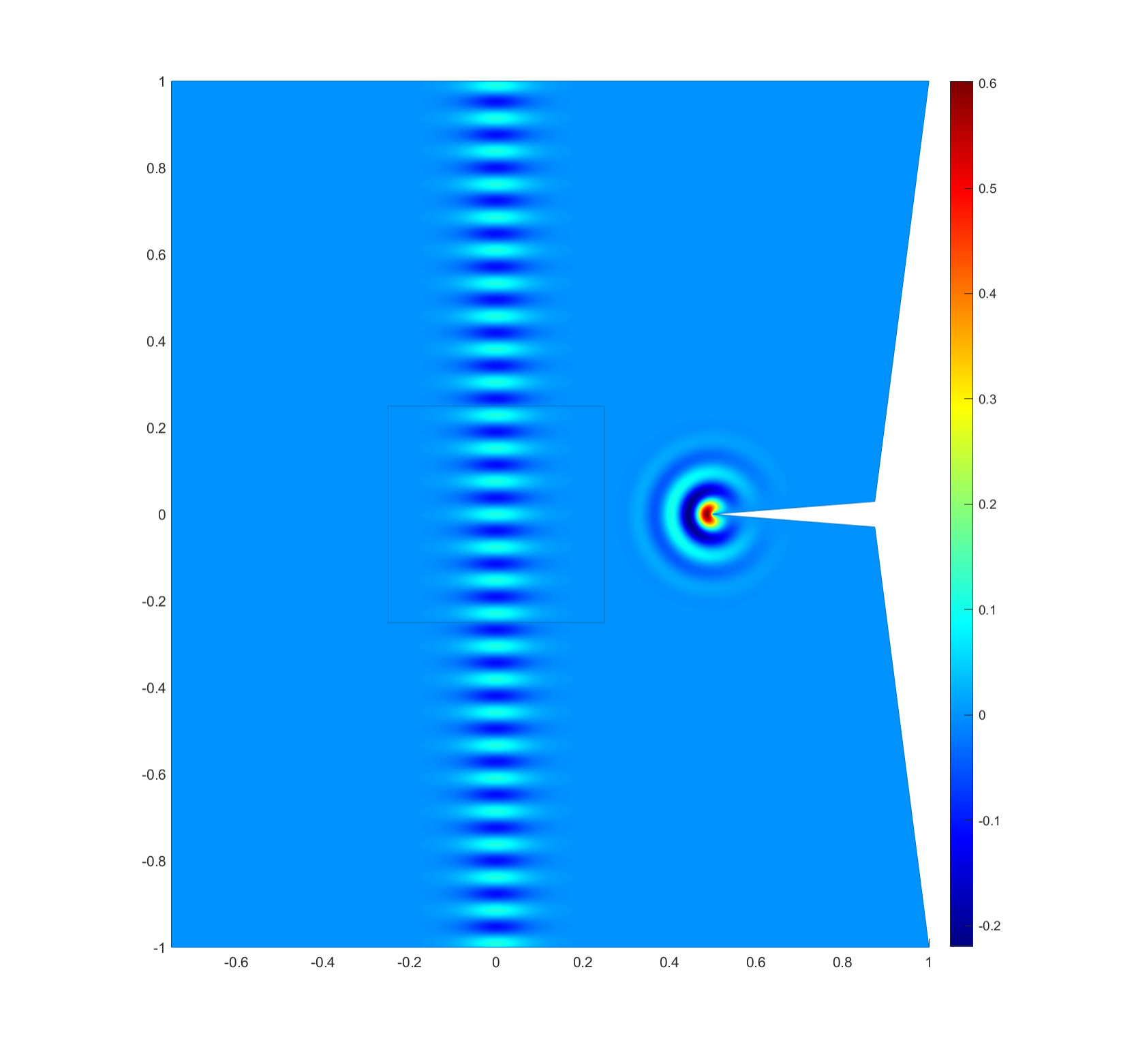}
  \end{minipage}
  \\
  \begin{minipage}[c]{0.47\columnwidth}
    \centering
    \includegraphics[width=0.66\columnwidth]{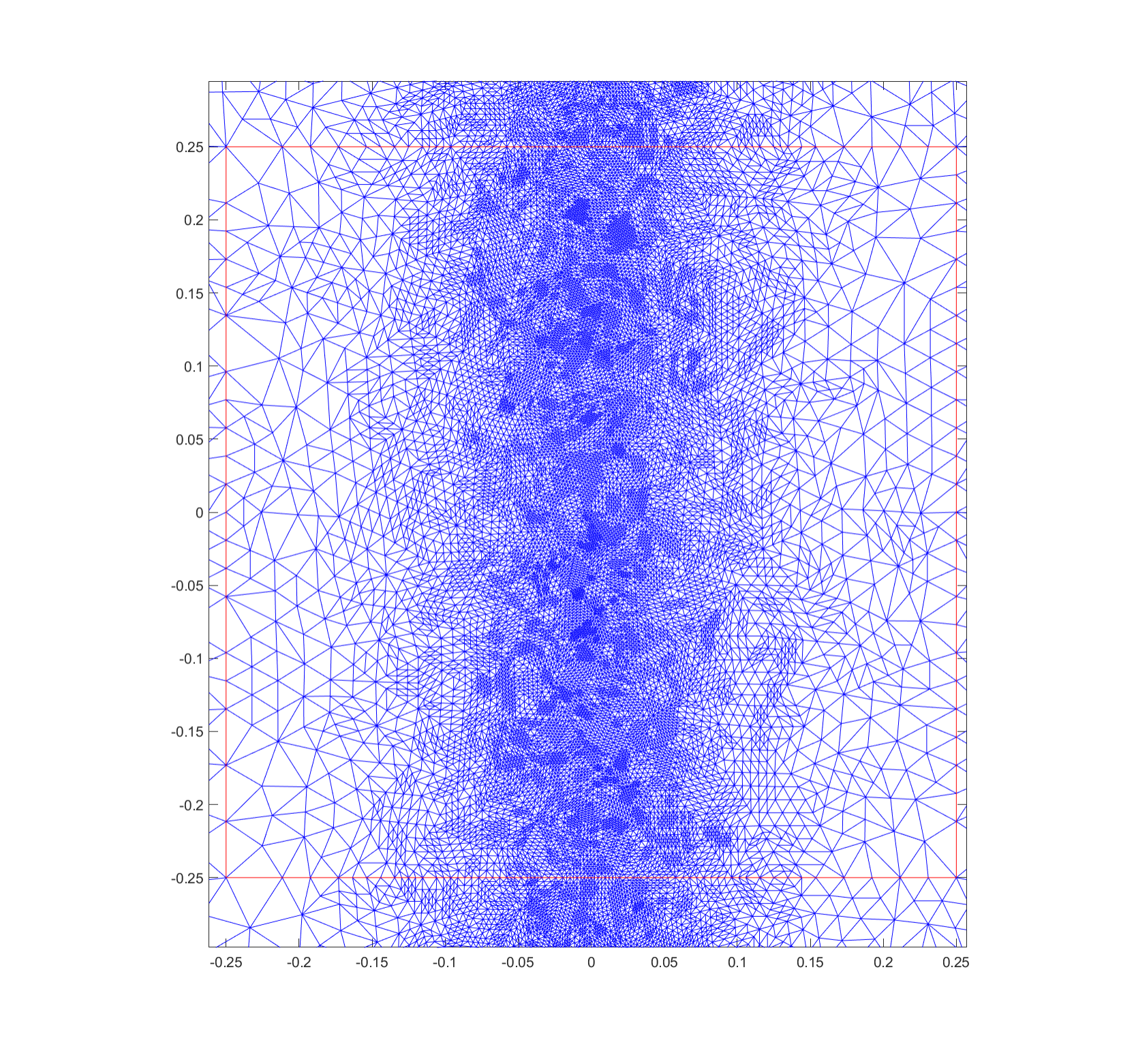}
    \end{minipage}
    \begin{minipage}[c]{0.47\columnwidth}
    \centering
    \includegraphics[width=0.66\columnwidth]{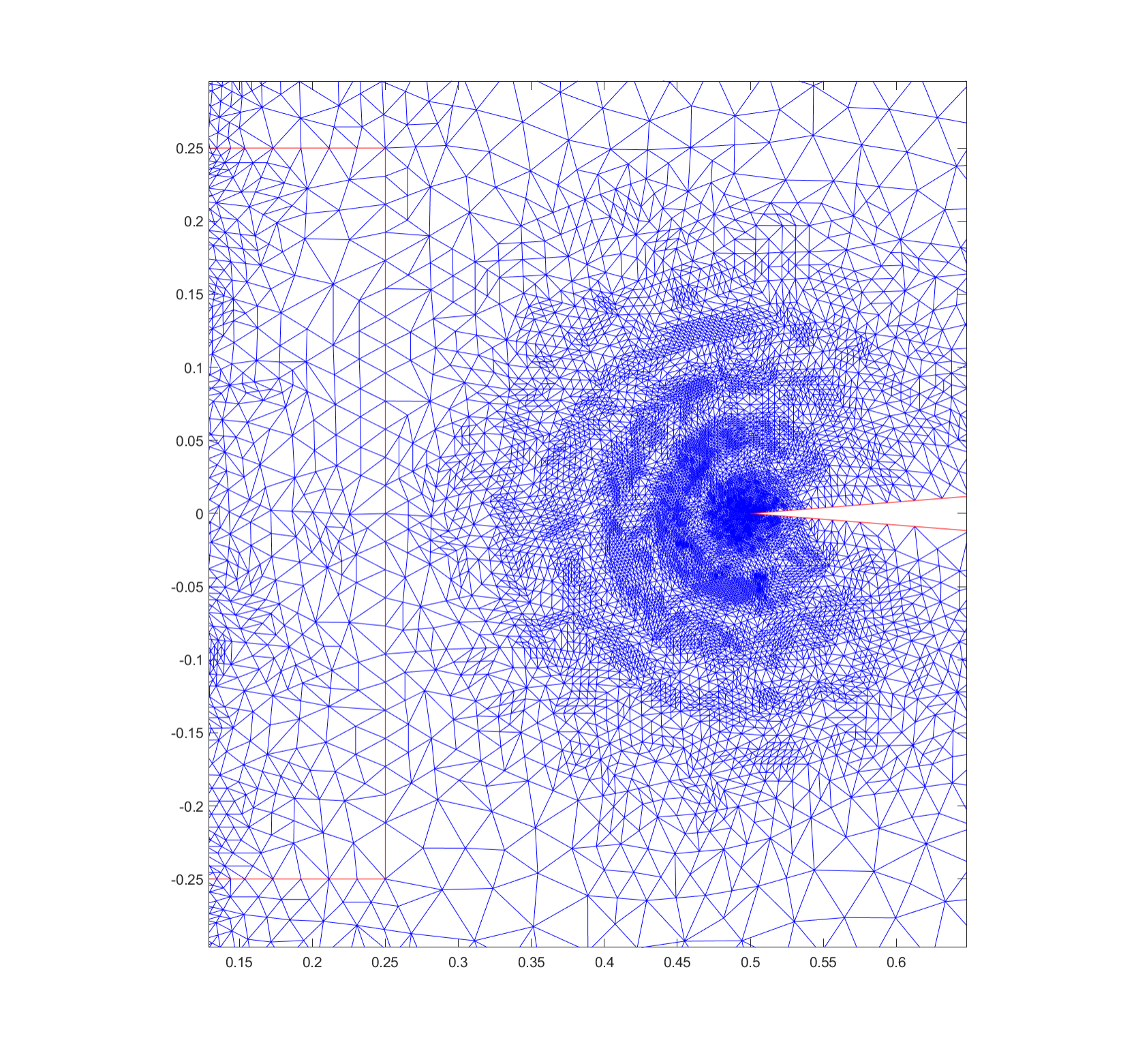}
    \end{minipage}
    \caption{\itshape{Top: The mesh after 20 adaptive iterations and $\Re(u)$ respectively; Bottom: The local meshes around $(\frac12,0)$ and $(0,0)$ respectively.}}
  \label{fig:ex1pet}
\end{figure}

From \eqref{eq:error estimate}, we can find if $h_\T\gs k^{-2}$, the pollution error dominates the error estimate. The interpolation error $O(kh_\T)$ associated with the regular component becomes dominant in the error estimate when $k^{-\frac{3-2\alpha}{2-2\alpha}}\ls h_\T\ls k^{-2}$, while the interpolation error $O(k^{-\frac12}(kh_\T)^{\alpha})$ corresponding to the singularities is dominant when $h_\T\ls k^{-\frac{3-2\alpha}{2-2\alpha}}$.
As illustrated in the left part of figure~\ref{fig:ex2},  the error of the adaptive FE solution initially decreases rapidly at a rate $O(N^{-1})$ due to the pollution effect and then decay at the rate $O(N^{-\frac12})$ when the pollution error disappears.
$\eta_\T$ fits the error of $P_hu$ well and both of them decay at the optimal rate of $O(N^{-\frac12})$. Notably, the error estimator obviously underestimates the true error of the FE solution in the preasymptotic regime. The error of FEM not using adaptive algorithm exceeds that of the adaptive method, that is to say, if we need the same accuracy, the number of elements for adaptive method is much less than not using adaptive method in this NLH problem. For example, when the error decays to 0.126, the number of elements for FEM is about $1.49\times 10^6$ whereas for AFEM, it is about $2\times 10^5$.
Additionally, the error from uniform refinements eventually comes close to the rate $O(N^{-\frac{10}{39}})$. It means at this stage the error of the singular part dominates the error estimate.
Comparing the two subfigures in figure~\ref{fig:ex2}, we observe that the error of ACIPFEM is much smaller than the error of AFEM in the preasymptotic regime, with a similar trend for not using adaptive methods. This demonstrates that ACIPFEM can reduce the pollution error efficiently and the error decays at the optimal rate $O(N^{-\frac12})$ quickly.
On the other hand, the error of ACIPFE solution is close to $\eta_\T$, so the error estimator does not underestimate the error of the CIPFE solution in the preasymptotic regime.

\begin{figure}[tbp]
  \centering
  \includegraphics[scale=0.405]{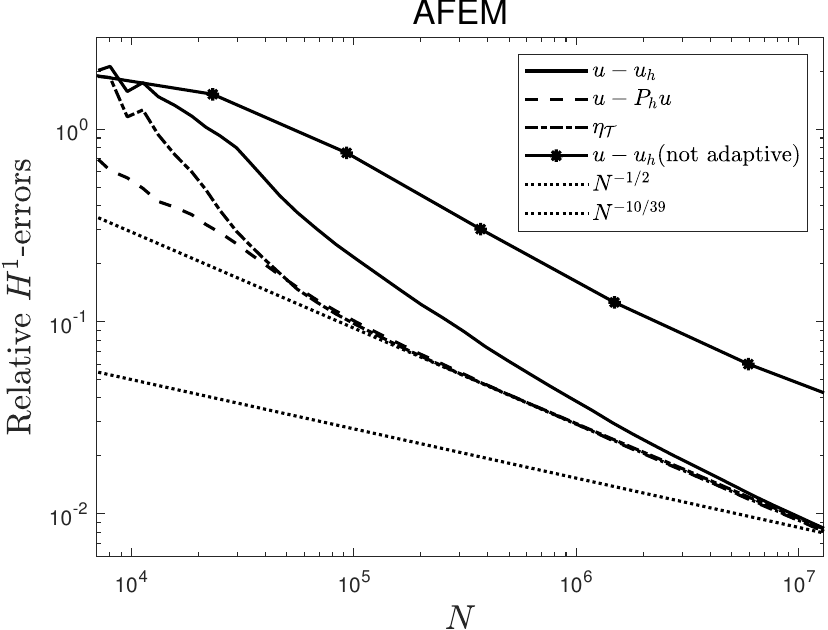}
  \includegraphics[scale=0.405]{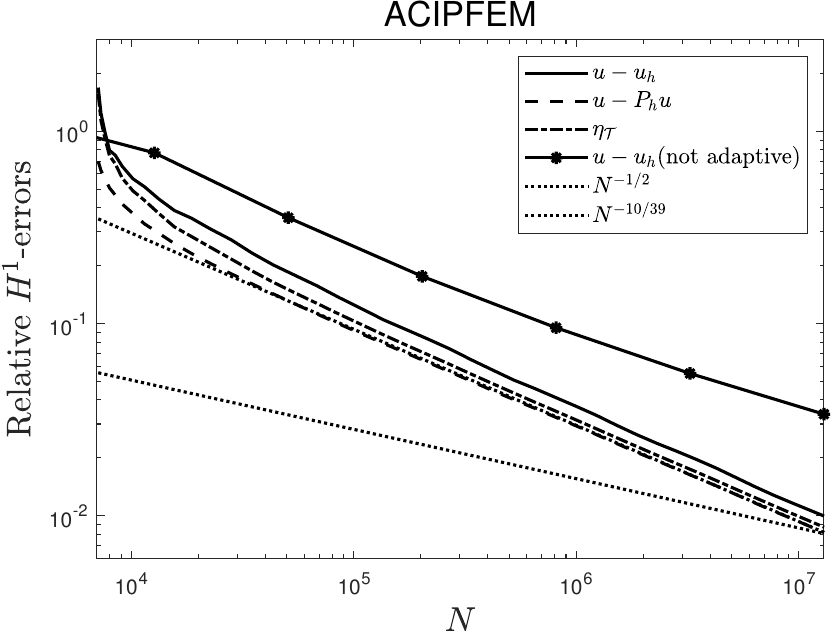}
  \caption{\itshape{The relative $H^1$-error versus the number of elements for $k=80, q=20, h_{\T_0}=0.04.$}}
  \label{fig:ex2}
  \end{figure}

To show the rate of the interpolation error of the singularities more clearly, we set small wave numbers $k=5, q=1, h_{\T_0}=0.1$ for another numerical test. See figure \ref{fig:ex1anota}, the error of AFEM decays at the optimal rate $O(N^{-\frac12})$.
The rate of the error curve of FEM rapidly varies from $O(N^{-\frac12})$ to $O(N^{-\frac{10}{39}})$
along with $h_\T$ becoming small.
This example verifies that the adaptive algorithm can make FEM converging at full order for $h_\T\ls k^{-\frac{3-2\alpha}{2-2\alpha}}$.
\begin{figure}[tbp]
  \centering
  \includegraphics[scale=0.405]{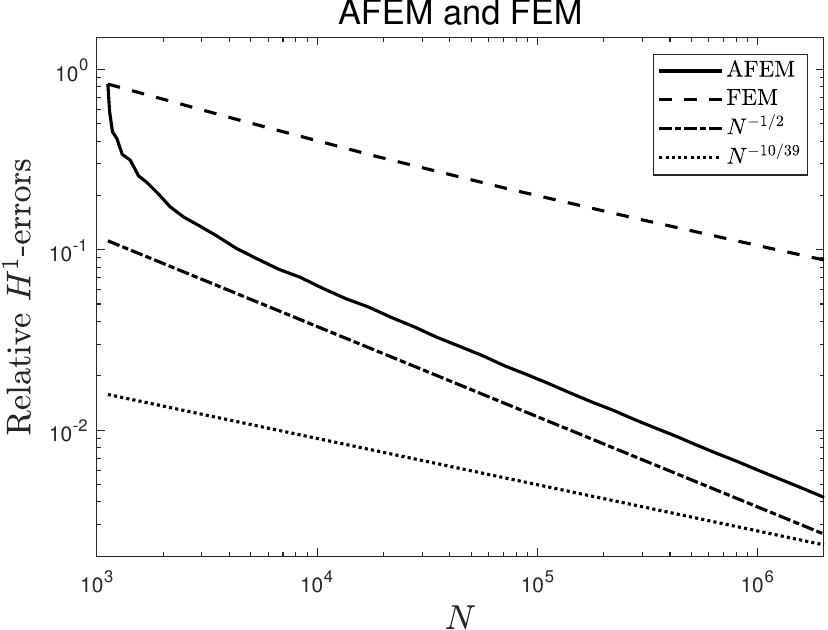}
  \caption{\itshape{The relative $H^1$-errors of AFEM and FEM versus the number of elements for $k=5, q=1$, respectively.}}
  \label{fig:ex1anota}
\end{figure}

\subsection{Optical bistability}
Optical bistability (see, e.g., \cite{Boyd:2008}) is a famous third-order nonlinear optical phenomenon. It refers to the situation in which two different output intensities are possible for a given input intensity.
Interest in optical bistability stems from its potential usefulness as a switch for use in optical communication and in optical computing.

We consider the NLH problem with the pure impedance boundary condition on a domain composed of two regular hexagons with their common center being the origin
and radiuses being $1$ and $1/2$, respectively. For an even number $m>0$, let $\mathcal{G}_m$ be the equilateral triangulation of mesh size $h=1/m$.
Let $k=k_0=9.6$ in $\Omz^c$ and $k=2.5k_0$ in $\Omz$ (cf. \cite{Wu:Zou:2018,Yuan:Lu:2017}). We set $\T_0=\mG_{100}, \vep=10^{-12}$ and $f=0$. The incident wave is specified as a collimated incoming beam $\uinc=Ie^{-k_0^2x^2/2}$
for which the refracted beam is approximately a Gaussian beam (see, e.g. \cite{Baruch:Fibich:Tsynkov:2009}).
The Gaussian mode describes the intended output of most lasers in math, as such a beam can be focused into the most concentrated spot. So using Gaussian beam as incident wave can simulate the optical bistability more accurately than using plane wave in our former work (\cite{Wu:Zou:2018}).
We let $I$ be in the range $[0,450000]$ to simulate the high-energy incident wave and denote a reference incident wave $\uinc^0=I_0e^{-k_0^2x^2/2}$ with $I_0=10^5$ for scaling.


\begin{figure}[tbp]
    \centering
    \includegraphics[scale=0.7]{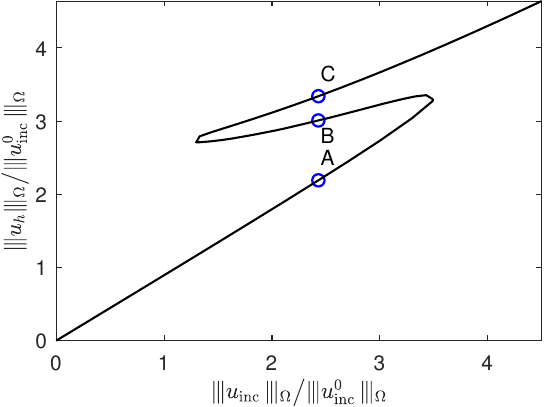}
    \includegraphics[scale=0.7]{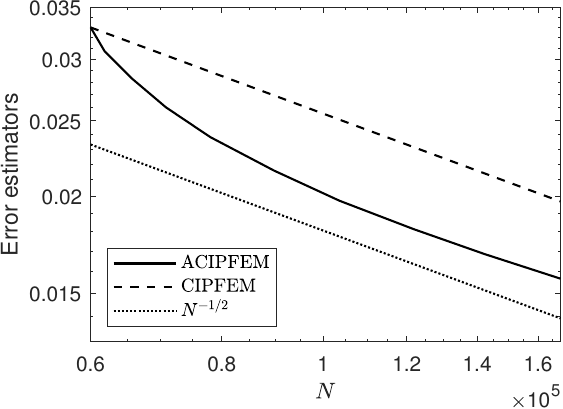}
  \caption{\itshape{Left: the scaled energy of the scattered field; Right: error estimators $\eta_\T$ at the point $C$.}}
  \label{fig:bistability}
\end{figure}

\begin{figure}[t]
  \centering
  \includegraphics[width=13cm]{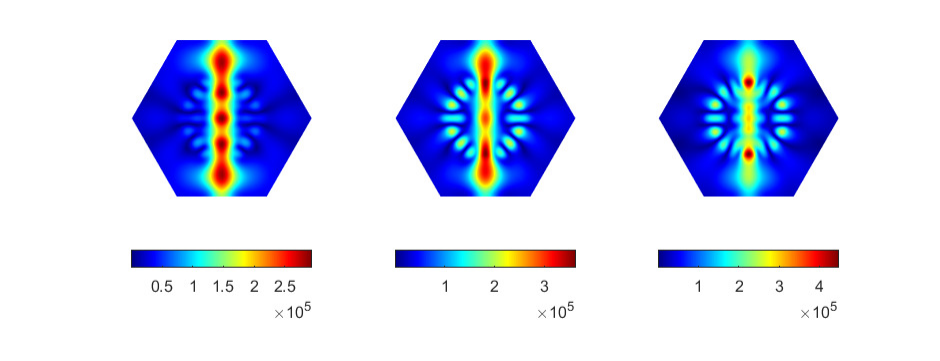}
  \vskip -20pt
  \caption{\itshape{Scattered field patterns (magnitude) of the three solutions marked as the points {\rm A}, {\rm B}, and {\rm C}
  .}}
  \label{fig:ex3ABC}
\end{figure}

The left part of figure \ref{fig:bistability} shows the energy norm of the scattered field computed by the ACIPFEM versus that of the incident wave. As analysed in section \ref{s:3} and shown in the section \ref{s:5.2}, we can use the error estimators to estimate the error of CIPFE solution.
Therefore we let the stopping criterion be $\eta_\T<0.5\eta_{\T_0}$.
Clearly, the curve exhibits three branches. In the lower branch, the intensity of the scattered field increases by enlarging the incident wave's intensity $I$ and jumps to the upper branch at $I=349000$.
At the upper branch, we can compute new points by using the adjacent points on $\T_0$ as the initial guess. 
As shown, when decreasing $I$ to 130000, the upper branch falls down to the lower branch.
For $I=243000$, we select half of the sum of the solutions at point A and point C on $\T_0$ as the initial guess to compute the solution at point B. 
Then we can extend the middle branch using the solution at point B as the initial guess.
This indicates that the NLH problem has three different solutions for $130000<I<349000$. One is unstable in the middle branch and two are most probably stable in the upper and lower branches.
We plot the electric field patterns of all three solutions at points A, B, and C in figure \ref{fig:ex3ABC}.
As shown, optical bistability for the scattering of a Gaussian incident wave has been efficaciously replicated.

It is well known that the large nonlinearity brings the singularity and needs huge amount of computation. Taking point C as an example, the ACIPFEM needs 9 iterations in order to satisfy the stopping criterion.
The left part of figure \ref{fig:ex3pet} illustrates the elementwise $H^1$ semi-norm of $u_h$ varies rapidly in the nonlinear medium $\Omz$. The middle and right parts of this figure depict the global mesh on $\Om$ and the local mesh around the nonlinear medium $\Omz$, respectively. 
We can find the mesh is dense in $\Omz$ and at the interface $\pa\Omz$ although the solution in this problem has $H^2$ regularity in the whole domain $\Om$. For the case of high-order FEM, the singularity at the interface $\pa\Omz$ becomes more heavily and we will study this case in our future work.
The right part of figure \ref{fig:bistability} plots the error estimators of CIPFEM and ACIPFEM, showing that both of them decay at the optimal rate $O(N^{-\frac12})$ and the ACIPFEM's decays very fast initially and is much smaller than CIPFEM's, that is to say, adaptive method is still efficient in solving high-energy NLH problem without corner singularities.

\begin{figure}[tbp]
  \centering
  \includegraphics[scale=0.16]{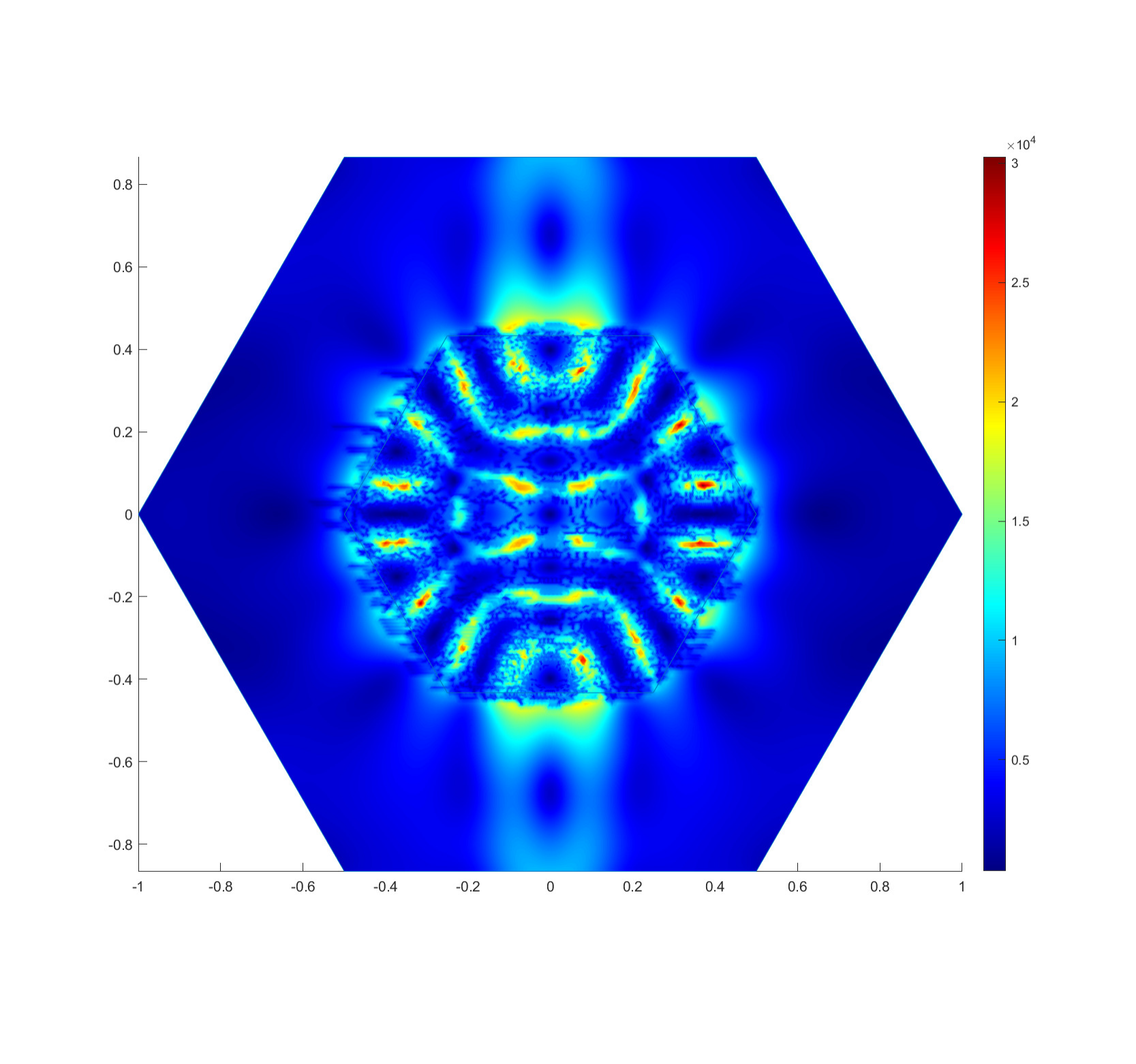}
  \includegraphics[scale=0.16]{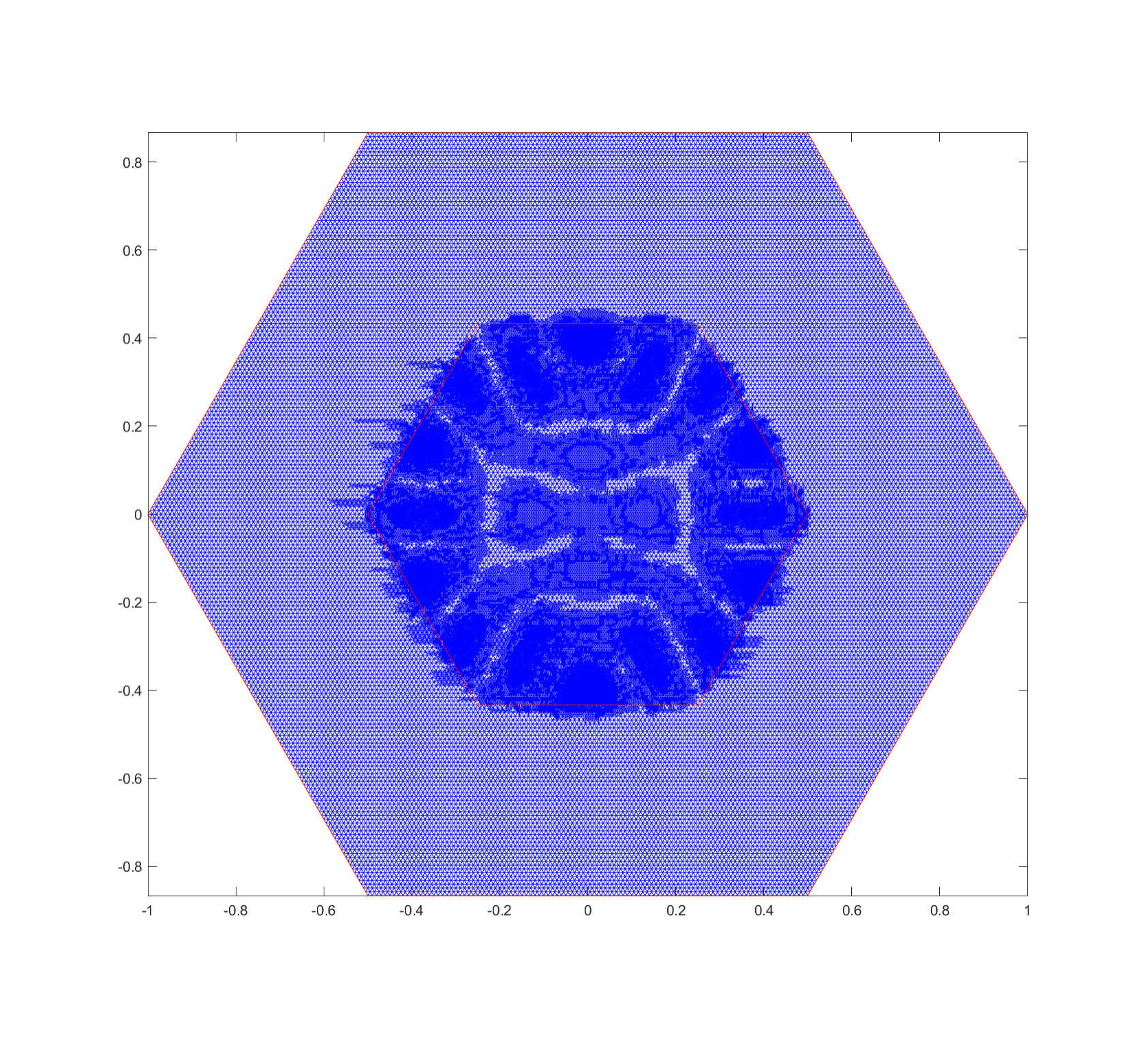}
  \includegraphics[scale=0.16]{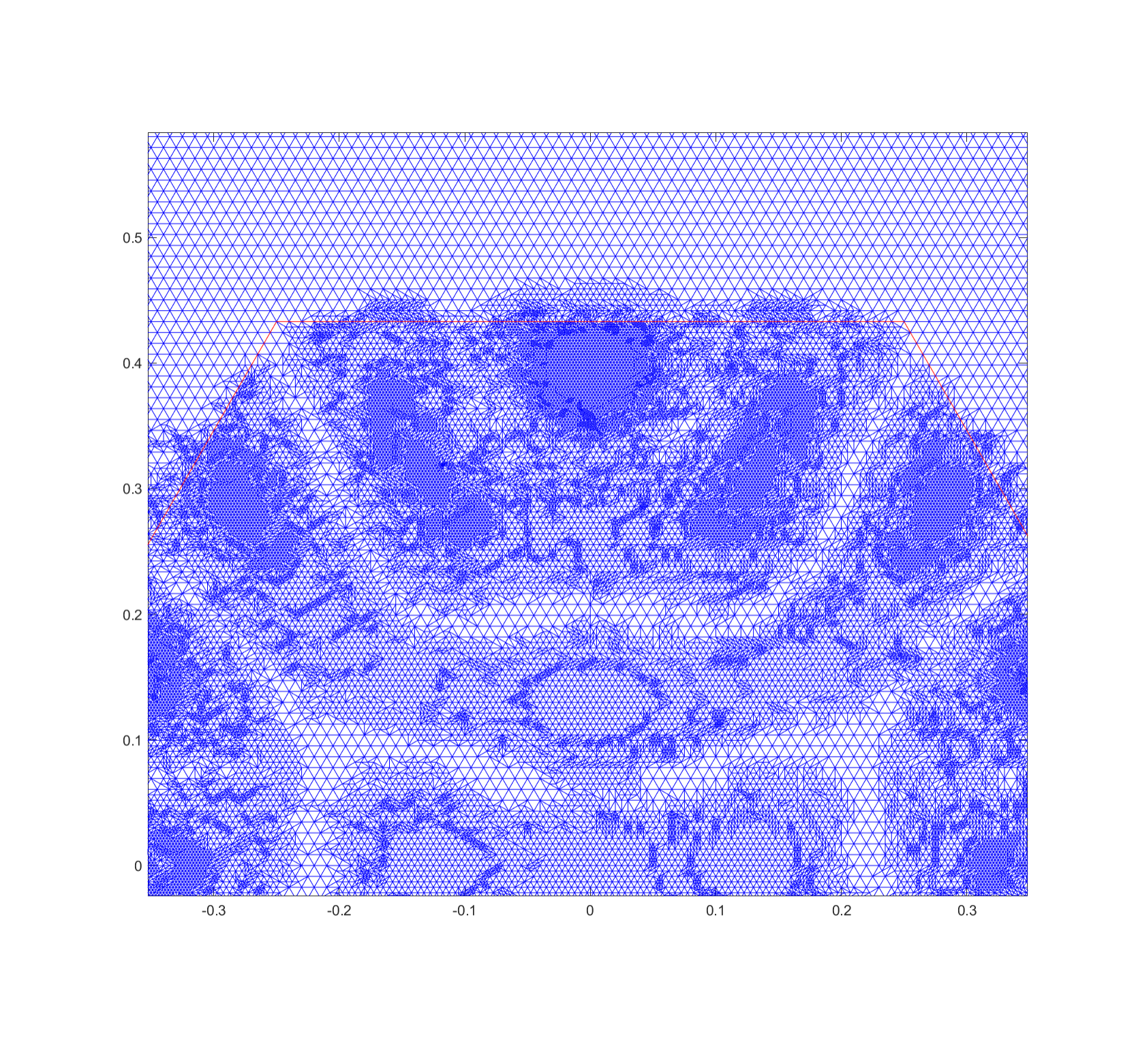}
 \caption{\itshape{$\abs{u_h}_{1,T}$, the global and local adaptive mesh at the point {\rm C}, respectively.}}
\label{fig:ex3pet}
\end{figure}




\appendix
\renewcommand{\theequation}{A.\arabic{equation}}
\renewcommand{\thetheorem}{A.\arabic{theorem}}
\renewcommand{\thefigure}{A.\arabic{figure}}
\setcounter{equation}{0}
\setcounter{theorem}{0}
\setcounter{figure}{0}
\section*{Appendix A}
In this appendix, we give a proof of the first estimate in Lemma \ref{lem:interior-estimate}.
\begin{lemma}
  If $k\vep\Linf{\psi}^2\leq \theta_0, k^3h_\T^{1+\alpha}\leq C_0$ for a positive constant $C_0$, then the solution of \eqref{eq:auxiliary dp1} satisfies
  \begin{equation}\label{eq:w_h stability}
    \He{w_h}\ls \MFg.
  \end{equation}
  \begin{proof}
    Firstly, we take $v_h=w_h$ in \eqref{eq:auxiliary dp1} to get by computing the imaginary and real parts
      \begin{align}
        k\LtD{w_h}{\GaI}^2-k^2\vep\Im(\psi^2\overline{w_h},w_h)&\leq 2\abs{(F,w_h)}+k^{-1}\LtD{g}{\GaI}^2,\label{eq:appendix eq1}\\
        \Lt{\nabla w_h}^2-k^2\Lt{w_h}^2-2k^2\vep\LtD{\psi w_h}{\Omz}^2-k^2\vep\Re(\psi^2\overline{w_h},w_h)&\leq 2\abs{(F,w_h)}+k^{-1}\LtD{g}{\GaI}^2.\label{eq:appendix eq2}
      \end{align}
  Then let $z$ be the solution of this problem:
    \begin{align}
      -\Delta z&-k^2(1+2\vep\oneo\abs{\psi}^2)z-k^2\vep\oneo\overline{\psi}^2\overline z=w_h\quad \text{in}\ \Om,\label{eq:appendix duality problem}\\
      &\frac{\pa z}{\pa n}-\mbi kz=0 \quad \text{on}\ \GaI, \quad
      z=0\quad \text{on}\ \GaD.\notag
    \end{align}
  We find that $\bar z$ is the solution to \eqref{eq:auxiliary cp} with $F=\overline{w_h}, g=0$.
  So we deduce from Lemma \ref{lem:auxiliary stability}, \eqref{eq:elliptic projection estimate}, and \eqref{eq:singular interpolation error} that
  \eqn{
    \Lt{z-P_hz}+h_\T^{\frac\alpha 2}\LtD{z-P_hz}{\GaI}\ls h_\T^{\alpha}(kh_\T+k^{-\frac12}(kh_\T)^{\alpha})\Lt{w_h} \qaq \Lt{z}\ls k^{-1}\Lt{w_h}.
  }
  Now we multiply \eqref{eq:appendix duality problem} by $\overline{w_h}$ and then integrate by parts
  \eqn{
    \begin{aligned}
      \Lt{w_h}^2&=a_0(\nabla w_h,\nabla z)-2k^2\vep(\abs{\psi}^2w_h,z)_{\Omz}-k^2\vep(\psi^2w_h,\overline z)_{\Omz}\\
      &=b(w_h,P_hz)-(k^2(1+2\vep\oneo\abs{\psi}^2)w_h+w_h,z)-k^2\vep(\psi^2w_h,\overline z)_{\Omz}+\mbi k\langle w_h,z\rangle_{\GaI}\\
      &=(F,P_hz)+\langle g,P_hz\rangle_{\GaI}+\mbi k\langle w_h,z-P_hz\rangle_{\GaI}-k^2\vep(\psi^2w_h,\overline z)_{\Omz}+k^2\vep(\psi^2\overline{w_h},P_hz)_{\Omz}\\
      &\quad -k^2((1+2\vep\oneo\abs{\psi}^2)w_h+w_h,z-P_hz)\\
      &=(w_\psi,w_h)+(F,P_hz-z)+\langle g,P_hz-z\rangle_{\GaI}+\mbi k\langle w_h,z-P_hz\rangle_{\GaI}+k^2\vep\big((\psi^2w_\psi,\overline z)_{\Omz}-(\psi^2\overline{w_\psi},z)_{\Omz}\big)\\
      &\quad -k^2((1+2\vep\oneo\abs{\psi}^2)w_h+w_h+\vep\oneo\psi^2\overline{w_h},z-P_hz)+k^2\vep\big(-(\psi^2w_h,\overline z)_{\Omz}+(\psi^2\overline{w_h},z)_{\Omz}\big)\\
      &\ls k^{-1}\MFg(1+kh_\T^{\frac\alpha 2}(kh_\T+k^{-\frac12}(kh_\T)^{\alpha})+k\theta_0k^{-1})\Lt{w_h}\\
      &\quad +(kh_\T^{\frac\alpha 2}+k^2h_\T^{\alpha})(kh_\T+k^{-\frac12}(kh_\T)^{\alpha})\Lt{w_h}^2+k\theta_0k^{-1}\Lt{w_h}^2,
    \end{aligned}
  }
  where we use $(F,z)+\langle g,z\rangle_{\GaI}=(w_\psi,w_h)+k^2\vep(\psi^2w_\psi,\overline z)_{\Omz}-k^2\vep(\psi^2\overline{w_\psi},z)_{\Omz}$ and \\
  $k\LtD{w_h}{\GaI}\ls k\Lt{w_h}+k\sqrt{\theta_0}\Lt{w_h}+\MFg\ls k\Lt{w_h}+\MFg$ from \eqref{eq:appendix eq1}.
  Then we conclude
  \eqn{
    \Lt{w_h}\ls k^{-1}\MFg.
  }
 Combining the above estimate with \eqref{eq:appendix eq2}, we deduce
  \eqn{
    \Lt{\nabla w_h}^2&\ls (k^2+k\theta_0)\Lt{w_h}^2+k^{-2}\Lt{F}^2+k^{-1}\LtD{g}{\Gamma}^2\ls \MFg^2.
  }
  Finally, \eqref{eq:w_h stability} comes from  $\He{w_h}\eqsim\Lt{\nabla w_h}+k\Lt{w_h}$.
 \end{proof}
\end{lemma}

\bibliographystyle{siam}
\bibliography{ref}

\end{document}

%% file: NLHAFEM_def.tex
\newcommand{\abs}[1]{\left\vert #1 \right\vert}
\newcommand{\norm}[1]{\left\Vert #1\right\Vert}
\newcommand{\nnorm}[1]{\| #1 \|}
\newcommand{\He}[1]{\left\Vert{\hskip -2.7pt}\left\vert #1 \right\vert{\hskip -2.7pt}\right\Vert}
\newcommand{\he}[1]{\big\vert\kern-0.25ex\big\vert\kern-0.25ex\big\vert #1 \big\vert\kern-0.25ex\big\vert\kern-0.25ex\big\vert}
\newcommand{\LtD}[2]{\left\Vert #1\right\Vert_{0,#2}}
\newcommand{\Ltd}[2]{\big\Vert #1\big\Vert_{0,#2}}
\newcommand{\Lt}[1]{\left\Vert #1\right\Vert_{0}}
\newcommand{\Ho}[1]{\left\Vert #1\right\Vert_{1}}

\newcommand{\HoD}[2]{\left\Vert #1\right\Vert_{1,#2}}
\newcommand{\sHoD}[2]{\left\vert #1\right\vert_{1,#2}}
\newcommand{\Ht}[1]{\left\Vert #1\right\Vert_{2}}

\newcommand{\sHtD}[2]{\left\vert #1\right\vert_{2,#2}}
\newcommand{\sHt}[1]{\left\vert #1\right\vert_{2}}
\newcommand{\ine}[2]{\langle #1,#2 \rangle_e}

\newcommand{\Linf}[1]{\norm{#1}_{L^{\infty}(\Omz)}}
\newcommand{\jump}[1]{[\![#1]\!]}

\newcommand{\R}{\mathbb{R}}

\newcommand{\T}{\mathcal{T}}
\newcommand{\mG}{\mathcal{G}}
\newcommand{\M}{\mathcal{M}}

\newcommand{\E}{\mathcal{E}}

\newcommand{\osc}{{\rm osc}}

\newcommand{\pihi}{\Pi_{hi}}
\newcommand{\phu}{P_hu}
\newcommand{\mbi}{{\rm\mathbf i}}
\newcommand{\pT}{{\partial T}}

\newcommand{\Om}{\Omega}

\newcommand{\na}{\nabla}

\newcommand{\GaD}{\Gamma_{\rm Dir}}
\newcommand{\GaI}{\Gamma_{\rm imp}}
\newcommand{\pa}{\partial}

\newcommand{\wt}{\widetilde}
\newcommand{\HoG}{H^1_{\GaD}(\Om)}

\newcommand{\ls}{\lesssim}
\newcommand{\gs}{\gtrsim}

\newcommand{\oneo}{{\bf 1}_{\Omz}}
\newcommand{\Mfg}{M(f,g)}
\newcommand{\MFg}{M(F,g)}
\newcommand{\hMfg}{\widehat{M}(f,g)}
\newcommand{\hMFg}{\widehat{M}(F,g)}
\newcommand{\vep}{\varepsilon}
\newcommand{\Omz}{\Omega_{0}}

\newcommand{\lnhmin}{\abs{\ln h_{\T,{\rm min}}}^{\frac12}}
\newcommand{\lnhminn}{\abs{\ln h_{\T,{\rm min}}}}
\newcommand{\uinc}{u_{\rm inc}}

\newcommand{\eq}[1]{\begin{align}#1\end{align}}
\newcommand{\eqn}[1]{\begin{align*}#1\end{align*}}

\newcommand*{\DP}[2]{\left\langle{#1},{#2}\right\rangle} 

\newcommand{\qaq}{\quad\mbox{and}\quad}
